\documentclass[onefignum,onetabnum]{siamart190516}

\usepackage{lipsum}
\usepackage{amsfonts}
\usepackage{graphicx}
\usepackage{epstopdf}
\usepackage{algorithmic}

\usepackage{tikz}
\usetikzlibrary{shapes.arrows, patterns}
\usepackage{bbm}
\usepackage{upgreek}

\newcommand*{\arXiv}[1]{\bgroup\color{blue}\href{http://arxiv.org/abs/#1}{arXiv:#1}\egroup}
\newcommand*{\doi}[1]{\bgroup\color{blue}\href{http://dx.doi.org/#1}{doi:#1}\egroup}
\renewcommand*{\url}[1]{\bgroup\color{blue}\href{#1}{#1}\egroup}
\newcommand{\todo}[1]{\bgroup\color{red}#1\egroup}

\usepackage{mathtools}

\makeatletter
\renewcommand{\paragraph}{%
  \@startsection{paragraph}{4}%
  {\z@}{1.0ex \@plus .5ex \@minus .2ex}{-.7em}%
  {\normalfont\normalsize\bfseries}%
}
\makeatother

\def\<{\big\langle}
\def\>{\big\rangle}

\newcommand*{\FRO}{\operatorname{FRO}}
\newcommand*{\Train}{\operatorname{Tr}}
\newcommand*{\Pred}{\operatorname{Pr}}

\newcommand*{\BigO}{\mathrm{O}}
\newcommand*{\defeq}{\coloneqq}

\newcommand*{\Expect}{\mathbb{E}}
\newcommand*{\Cov}{\operatorname{Cov}}
\newcommand*{\Naturals}{\mathbb{N}}
\newcommand*{\Reals}{\mathbb{R}}
\newcommand*{\dx}{\,\mathrm{d}x}
\newcommand*{\dy}{\,\mathrm{d}y}
\newcommand*{\dP}{\,\mathrm{d}P}
\newcommand*{\dQ}{\,\mathrm{d}Q}

\DeclareMathOperator{\dist}{dist}
\DeclareMathOperator{\disttt}{\mathtt{dist}}
\DeclareMathOperator{\trace}{trace}
\DeclareMathOperator{\logdet}{logdet}
\DeclareMathOperator{\chol}{chol}
\DeclareMathOperator{\diag}{diag}

\definecolor{lightblue}{HTML}{a1b4c7}
\definecolor{orange}{HTML}{ea8810}
\definecolor{silver}{HTML}{b0aba8}
\definecolor{rust}{HTML}{b8420f}
\definecolor{seagreen}{HTML}{23553c}

\colorlet{lightsilver}{silver!30!white}
\colorlet{darkorange}{orange!85!black}
\colorlet{darksilver}{silver!85!black}
\colorlet{darklightblue}{lightblue!85!black}
\colorlet{darkrust}{rust!85!black}
\colorlet{darkseagreen}{seagreen!85!black}

\hypersetup{colorlinks=true,linkcolor=darkrust,citecolor=darklightblue,urlcolor=darksilver}

\newcommand*{\quark}{\setbox0\hbox{$x$}\hbox to\wd0{\hss$\cdot$\hss}}

\newcommand*{\K}{G}
\newcommand*{\KM}{\Theta}
\newcommand*{\IK}{\mathcal{L}}
\newcommand*{\IKM}{A}

\newcommand*{\Id}{\textup{Id}}

\renewcommand*{\O}{\mathcal{O}}
\newcommand*{\I}{I}
\newcommand*{\J}{J}
\newcommand*{\N}{\mathcal{N}}
\newcommand*{\argmin}{\operatorname{argmin}}
\newcommand*{\argmax}{\operatorname{argmax}}

\newcommand*{\SpSet}{\mathcal{S}}

\DeclarePairedDelimiterX{\infdivx}[2]{(}{)}{%
  #1\;\delimsize\|\;#2%
}
\newcommand{\KL}{\mathbb{D}_{\operatorname{KL}}\infdivx}

\newtheorem{condition}[theorem]{Condition}

\ifpdf
  \DeclareGraphicsExtensions{.eps,.pdf,.png,.jpg}
\else
  \DeclareGraphicsExtensions{.eps}
\fi

\newsiamremark{remark}{Remark}
\newsiamremark{hypothesis}{Hypothesis}
\crefname{hypothesis}{Hypothesis}{Hypotheses}
\crefname{condition}{Condition}{Condition}
\newsiamthm{claim}{Claim}

\headers{Factorization by Kullback-Leibler minimization
}{Florian Sch{\"a}fer, Matthias Katzfuss, and Houman Owhadi}

\title{Sparse Cholesky factorization by Kullback-Leibler minimization }

\author{%
  Florian\ Sch{\"a}fer\thanks{California Institute of Technology, MC 305-16, 1200 East California Boulevard, Pasadena, CA 91125, USA, \newline \email{florian.schaefer@caltech.edu}, Phone: (626) 395-3531,  Fax: (626) 578-0124, \newline Corresponding Author} 
  \and
  Matthias Katzfuss\thanks{Department of Statistics, Texas A\&M University}
  \and
  Houman Owhadi\thanks{Department of Computing + Mathematical Sciences, California Institute of Technology}
}

\usepackage{amsopn}

\ifpdf
\hypersetup{
  pdftitle={Sparse Cholesky factorization by Kullback-Leibler minimization},
  pdfauthor={Florian Sch{\"a}fer, Matthias Katzfuss, and Houman Owhadi}
}
\fi

\begin{document}

\maketitle

\begin{abstract}
 We propose to compute a sparse approximate inverse Cholesky factor $L$ of a dense covariance matrix $\KM$ by minimizing the Kullback-Leibler divergence between the Gaussian distributions $\mathcal{N}(0, \KM)$ and $\mathcal{N}(0, L^{-\top} L^{-1})$, subject to a sparsity constraint.
 Surprisingly, this problem has a closed-form solution that can be computed efficiently, recovering the popular Vecchia approximation in spatial statistics.
 Based on recent results on the approximate sparsity of inverse Cholesky factors of $\KM$ obtained from pairwise evaluation of Green's functions of elliptic boundary-value problems at points $\{x_{i}\}_{1 \leq i \leq N} \subset \Reals^{d}$, we propose an elimination ordering and sparsity pattern that allows us to compute $\epsilon$-approximate inverse Cholesky factors of such $\KM$ in computational complexity $\O(N \log(N/\epsilon)^d)$ in space and $\O(N \log(N/\epsilon)^{2d})$ in time.
 To the best of our knowledge, this is the best asymptotic complexity for this class of problems.
 Furthermore, our method is embarrassingly parallel, automatically exploits low-dimensional structure in the data, and can perform Gaussian-process regression in linear (in $N$) space complexity. 
 Motivated by its optimality properties, we propose to apply our method to the joint covariance of training and prediction points in Gaussian-process regression, greatly improving stability and computational cost.
 Finally, we show how to apply our method to the important setting of Gaussian processes with additive noise, compromising neither accuracy nor computational complexity. 
\end{abstract}

\begin{keywords}
  	Covariance function, Vecchia approximation, kernel matrix, sparsity, transport map, factorized sparse approximate inverse.
\end{keywords}

\begin{AMS}
  	65F30 
	(42C40, 
	65F50, 
	65N55, 
	65N75, 
	60G42, 
	68W40) 
\end{AMS}

\section{Introduction}

\paragraph{The problem} 

This work is concerned with the sparse inverse--Cholesky factorization of large dense positive-definite matrices $\KM \in \Reals^{N \times N}$, frequently arising as \emph{kernel matrices} in machine-learning methods using the ``kernel trick'' \cite{hofmann2008kernel}, as \emph{covariance matrices} in Gaussian-process (GP) statistics \cite{rasmussen2006gaussian}, and as \emph{Green's matrices} in the numerical analysis of elliptic partial differential equations (PDEs).
Naive computations of quantities such as $\KM v$, $\KM^{-1}v$, $\logdet \KM$, which are required by the applications mentioned above, scale as $\O(N^2)$ or $\O(N^3)$, and become prohibitively expensive for $N > 10^5$ on present-day  hardware.

\paragraph{Existing work}

Numerous approaches have been proposed in the literature to improve this computational complexity by taking 
advantage of the structure of $\KM$. Many rely on sparse approximations to the kernel 
matrix (e.g.,~\cite{furrer2012covariance,kaufman2008covariance}), its inverse
(e.g.,~\cite{lindgren2011explicit,roininen2011correlation,roininen2013constructing,roininen2014whittle}), or the Cholesky factor of its inverse (e.g., \cite{Vecchia1988}); also popular are 
low-rank approximations 
(e.g.,~\cite{williams2001using,smola2001sparse,fine2001efficient,bach2002kernel,fowlkes2004spectral,banerjee2008gaussian}) and combinations of low-rank and sparse approximations
(e.g.,~\cite{schwaighofer2002transductive,snelson2005sparse,quinonero2005unifying,sang2012full}).
Near-linear computational complexity can be achieved by applying these mechanisms hierarchically on multiple scales.
Examples of hierarchical sparse approximations include wavelet methods (e.g.,~\cite{beylkin1991fast}), the multi-resolution approximation \cite{katzfuss2016multi,Katzfuss2017b}, and (implicitly) some versions of the Vecchia approximation 
\cite{Katzfuss2017a}.
Hierarchical application of low-rank approximations leads to \emph{hierarchical matrices}
\cite{hackbusch1999sparse,hackbusch2000sparse,hackbusch2002data,chandrasekaran2004fast,ambikasaran2013mathcal,ho2016hierarchical,coulier2016efficient,coulier2017inverse,takahashi2017application},
which are an algebraic abstraction of the fast multipole method \cite{greengard1987fast}.
\cite{Schafer2017} proposed an approximation based on incomplete Cholesky factorization that can be interpreted as 
both hierarchical sparse and hierarchical low-rank.

The best asymptotic (in $N$ and $\epsilon$) memory complexity for the 
$\epsilon$-accurate compression of an $N \times N$ kernel matrix with finitely smooth covariance function and $d$-dimensional feature space is 
$\O(N \log^d(N/\epsilon))$, which is achieved by wavelets in nonstandard form 
(\cite{beylkin1991fast}, for asymptotically smooth kernels), or sparse inverse Cholesky factors of $\KM$ (\cite{Schafer2017}, based on results in \cite{owhadi2017universal, OwhScobook2018}).
However, we are not aware of  \emph{practical} algorithms that provably compute such approximations in near-linear time from\footnote{Hidden constants in all asymptotic complexities may depend on the dimension $d$ of the dataset.} $\O(N \log^d(N/\epsilon))$ entries of $\KM$ chosen a priori.

\paragraph{Our method}

We propose to compute a sparse approximate inverse Cholesky factor $L$ of $\KM$, by minimizing with respect to $L$ and subject to a sparsity constraint, the Kullback-Leibler (KL) divergence between two centered multivariate normal distributions with covariance matrices  $\KM$ and $(LL^\top)^{-1}$.
Surprisingly, this minimization problem has a closed-form solution, enabling the efficient computation of  optimally accurate Cholesky factors for any specified sparsity pattern. 

The resulting approximation can be shown to be equivalent to the Vecchia approximation of Gaussian processes \cite{Vecchia1988},
which has become very popular for the analysis of geospatial data
(e.g.,~\cite{stein2004approximating,Datta2016,Sun2016,guinness2016permutation,Katzfuss2017a,Katzfuss2018}); to the best of our knowledge, 
rigorous convergence rates and error bounds were previously unavailable for Vecchia approximations, and this work is the first one presenting such results.
An equivalent approximation has also been proposed by \cite{kaporin1990alternative} and \cite{kolotilina1993factorized} in the literature on factorized sparse approximate inverse (FSAI) preconditioners of (typically) sparse matrices (see e.g., \cite{benzi1999comparative} for a review and comparison, \cite{chow2014preconditioned} for an application to dense kernel matrices); however, its KL-divergence optimality has not been observed before.
KL-minimization has also been used to obtain sparse lower-triangular transport maps by \cite{marzouk2017}; while this literature is mostly concerned with the efficient sampling of non-Gaussian probability measures, the present work shows that an analogous approach can be used to obtain fast algorithms for numerical linear algebra if the sparsity pattern is chosen appropriately.

\paragraph{State-of-the-art computational complexity}

The computational complexity and approximation accuracy of our approach depend on the choice of elimination ordering and sparsity pattern. 
We propose a particular choice, similar to \cite{guinness2016permutation} and \cite{Schafer2017}, that is motivated by the \emph{screening effect} (e.g.,~\cite{stein2002screening,stein2011when,bao2020screening}), which implies (approximate) conditional independence for many kernels of common interest. 
By using a grouping algorithm similar to the heuristics proposed by \cite{ferronato2015novel} and \cite{guinness2016permutation}, we can show that the approximate inverse Cholesky factor can be computed in computational complexity $\O(N \rho^{2d})$ in time and $\O(N \rho^{d})$ in space, using only $\O(N \rho^{d})$  entries of the original kernel matrix $\KM$, where $\rho$ is a tuning parameter  trading accuracy for computational efficiency.

The authors of \cite{Schafer2017} observe that recent results on  numerical homogenization and operator-adapted wavelets \cite{malqvist2014localization,kornhuber2016analysis,owhadi2017universal} imply the exponential decay of the inverse Cholesky factors of $\KM$, if the kernel function is the Green's function of an elliptic boundary-value problem. 
Using these results, we prove that in this setting, an $\epsilon$-approximation of $\KM$ can be obtained by choosing $\rho \approx \log(N/\epsilon)$.
This leads to the best-known trade-off between computational complexity and accuracy  for this class of kernel matrices.

\paragraph{Practical advantages}

Our method  has important \emph{practical} advantages complementing its theoretical and asymptotic properties. In many  GP regression applications, large values of $\rho$ are computationally intractable with present-day resources. 
By incorporating prediction points in the computation of KL-optimal inverse-Cholesky factors, we obtain a GP regression algorithm that is accurate even for small ($\approx 3$) values of $\rho$, including in settings where truncation of the \emph{true} Cholesky factor of $\KM^{-1}$ to the same sparsity pattern fails completely. 

For other hierarchy-based methods, the computational complexity depends exponentially on the dimension $d$ of the dataset.
In contrast, because the construction of the ordering and sparsity pattern only uses pairwise distances between points, our algorithms automatically adapt to low-dimensional structure in the data and operate in complexities identified by replacing $d$ with the \emph{intrinsic dimension} $\tilde{d} \leq d$ of the dataset.

An important limitation of existing methods based on the screening effect \cite{guinness2016permutation,Schafer2017,Katzfuss2018} is that they deteriorate when applied to independent sums of two GPs, such as when combining a GP with additive Gaussian white noise.
Extending ideas proposed in \cite{Schafer2017}, we are able to fully preserve both the accuracy and asymptotic complexity of our method over a wide range of noise levels.
To the best of our knowledge, this is the first time this has been achieved by a method based on the screening effect.

Finally, our algorithm is intrinsically parallel because it allows each column of the sparse factor to be computed independently (as in the setting of the Vecchia approximation, factorized sparse approximate inverses, and lower-triangular transport maps).
Furthermore, we show that in the context of GP regression, the loglikelihood, the posterior mean, and the posterior variance can be computed in $\O(N + \rho^{d})$ space complexity. 
In a parallel setting, we require $\O(\rho^d)$ communication between the different workers for every $\O(\rho^{3d})$ floating-point operations, resulting in a total communication complexity of $\O(N)$. 
Here, most of the floating-point operations arise from calls to highly optimized BLAS and LAPACK routines.

\paragraph{Outline}

The remainder of this article is organized as follows. In \cref{sec:cholesky}, we show how sparsity-constrained KL-minimization  yields a simple formula for approximating the inverse Cholesky factor of a positive-definite matrix. 
In \cref{sec:ordSparse}, we present elimination orderings and sparsity patterns that provably lead to state-of-the-art trade-off between computational complexity and accuracy when applied to Green's functions of elliptic PDEs, and that we recommend more generally for covariance matrices of Gaussian processes that are subject to a screening effect.
In \cref{sec:theory}, we bound the computational complexity of our algorithm and rigorously quantify its complexity/accuracy trade-off. 
In \cref{sec:extensions}, we showcase three extensions of our method, allowing the treatment of additive noise due to measurement errors, improving the speed and accuracy of prediction, and enabling GP regression at \emph{linear} complexity in space and communication (between workers) in a distributed setting.
In \cref{sec:numerics}, we present numerical experiments applying our method to GP regression and to boundary-element methods for the solution of elliptic PDEs. We summarize our findings in \cref{sec:conclusions}.
The proofs of the main results are deferred to an appendix. 
Further details on the construction of the ordering and sparsity pattern, as well as on the implementation of some variants of our method are provided in the supplementary material.

\section{Cholesky factorization by KL-minimization\label{sec:cholesky}}

The Kullback-Leibler divergence between two probability measures $P$ and $Q$ is defined as $\KL*{P}{Q}= \int \log(\dP/\dQ) \dP$. 
If $Q$ is an approximation of $P$, then the KL divergence is the expected difference between the associated true and approximate log-densities, and so its minimization is directly relevant for accurate approximations of GP inference, including GP prediction and likelihood-based inference on hyperparameters.
By virtue of its connection to the likelihood ratio test \cite{eguchi2006interpreting}, the KL divergence can also be interpreted as the strength of the evidence that samples from $P$ were not instead obtained from $Q$.
If $P$ and $Q$ are both $N$-variate centered normal distributions, the KL divergence is equivalent to a popular loss function for covariance-matrix estimation \cite{James1961}, and it can be written as
\begin{equation}
\label{eqn:KLnormal}
2 \, \KL*{\N(0,\KM_1)}{\N(0,\KM_2)} =  \trace(\KM_2^{-1}\KM_1) + \logdet(\KM_2) - \logdet(\KM_1) - N. 
\end{equation}

Let $\KM$ be a positive-definite matrix of size $N \times N$. Given a lower-triangular sparsity set $S \subset \I \times \I$, where $\I = \{1,\ldots,N\}$, we want to use 
\begin{equation}\label{eqn:defVarL}
  L \defeq \argmin_{\hat{L} \in \SpSet} \KL*{\N\big(0,\KM\big)}{\N\big(0, (\hat{L} \hat{L}^{\top})^{-1}\big)}
\end{equation}
as approximate Cholesky factor for $\KM^{-1}$, for 
$\SpSet \defeq \left\{A \in \Reals^{N \times N}: A_{ij} \neq 0 \Rightarrow \left(i,j\right) \in S \right\}$.
While solving the non-quadratic program~\cref{eqn:defVarL} might seem challenging, it turns out that it has a closed-form solution that can be computed efficiently:
\begin{theorem}
    \label{thm:repL}
    The nonzero entries of the $i$-th column of $L$ as defined in Equation~\eqref{eqn:defVarL} are given by
    \begin{equation}
     \label{eqn:defcolL}
     L_{s_i,i} =  \frac{\KM_{s_i, s_i}^{-1} \mathbf{e}_1}{\sqrt{\mathbf{e}_1^{\top} 
     \KM_{s_i, s_i}^{-1} \mathbf{e}_1}},
    \end{equation}
    where $s_i \defeq \left\{ j: \left(j,i\right) \in S\right\}$,
    $\KM_{s_i, s_i}^{-1}:=(\KM_{s_i, s_i})^{-1}$, $\KM_{s_i, s_i}$ is the restriction of $\KM$ to the set of indices $s_i$,
    and 
    $\mathbf{e}_1 \in \Reals^{\#s_i \times 1}$ is the vector 
    with the first entry equal to one and all other entries equal to zero.
    Using this formula, $L$ can be computed in computational complexity 
    $\mathcal{O}\big( \#S + (\max_{1 \leq i \leq N} \# s_i )^2 \big)$ in space and 
    $\mathcal{O}\big( \sum_{i=1}^N \left( \# s_i \right)^3 \big)$ in time.
\end{theorem}
\begin{proof}
See \cref{apssec:notAggregatedPattern}.
\end{proof}

Compared to ordinary sparse Cholesky factorization (see \cref{alg:ichol}), the algorithm implied by \cref{thm:repL} has the advantage of giving the \emph{best} possible Cholesky factor (as measured by KL) for a given sparsity pattern.
Furthermore, it is embarrassingly parallel --- all evaluations of Equation~\eqref{eqn:defcolL} can be performed independently for different $i$.
While the computational complexity is slightly worse than the one of in-place incomplete Cholesky factorization, we will show in \cref{thm:complexity} that for important choices of $S$, the time complexity can be reduced to 
$\mathcal{O}\big( \sum_{k =1}^N \left(\# s_k \right)^2 \big)$, matching the computational 
complexity of incomplete Cholesky factorization.

The formula in Equation~\eqref{eqn:defcolL} can be shown to be equivalent to the formula that has been used to compute the Vecchia approximation \cite{Vecchia1988} in spatial statistics, without explicit awareness of the KL-optimality of the resulting $L$. 
In the literature on factorized sparse approximate inverses, the above formula was derived for minimizers of $\|\Id - L\chol(\KM)\|_{\FRO}$ subject to the constraints $L \in \mathcal{S}$ and $\diag(L \KM L^{\top}) = 1$ \cite{kolotilina1993factorized}, and for minimizers of the Kaporin condition number $(\trace( \KM LL^{\top})/N)^N/\det(\KM(LL^{\top}))$ subject to the constraint $L\in \mathcal{S}$ \cite{kaporin1990alternative}.
The KL-divergence, as opposed to $\|\Id - L\chol(\KM)\|_{\FRO}$, strongly penalizes zero eigenvalues of $\KM L L^{\top}$, which explains the observation of \cite{eremin1998factorized} that adding the constraint $\diag(L \KM L^{\top}) = 1$ tends to improve the spectral condition number of the resulting preconditioner, despite increasing the size of the fidelity term $\|\Id - L\chol(\KM)\|_{\FRO}$.
\cite{marzouk2017} showed that the embarrassingly parallel nature of KL-minimization is even preserved when replacing the Cholesky factors with nonlinear transport maps with Knothe-Rosenblatt structure.
As part of ongoing work on the sample complexity of the estimation of transport maps, \cite{baptista2020adaptive} discovered representations very similar to Equation~\eqref{eqn:defcolL}, independently of the present work.

Based on the results above, we propose the following procedure to approximate a large positive-definite matrix $\KM$: 
\begin{enumerate}
    \item Order the degrees of freedom (i.e., rows and columns of $\KM$) according to some ordering $\prec$.
    \item Pick a sparsity set $S \subset \I \times \I$.
    \item Use Formula~\eqref{eqn:defcolL} to compute the lower-triangular matrix $L$ with nonzero entries contained in $S$ that minimizes $\KL*{\N\big(0,\KM\big)}{\N\big(0, (L L^{\top})^{-1}\big)}$.
\end{enumerate}
In the next section, we will describe how to implement all three steps of this procedure in the more concrete setting of positive-definite matrices obtained from the evaluation of a finitely smooth covariance function at pairs of points in $\Reals^d$.

\section{Ordering and sparsity pattern motivated by the screening effect}
\label{sec:ordSparse}

The quality of the approximation given by Equation~\eqref{eqn:defVarL} depends on the 
ordering of the variables and the sparsity pattern $S$.
For kernel matrices arising from finitely smooth Gaussian processes, we propose specific orderings and sparsity patterns, which can be constructed in near-linear computational complexity and which lead to good approximations for many $\KM$ of practical interest.

\subsection{The reverse-maximin ordering and sparsity pattern\label{ssec:maximin}}

\begin{figure}
	\centering	
	\includegraphics[width=0.49\textwidth]{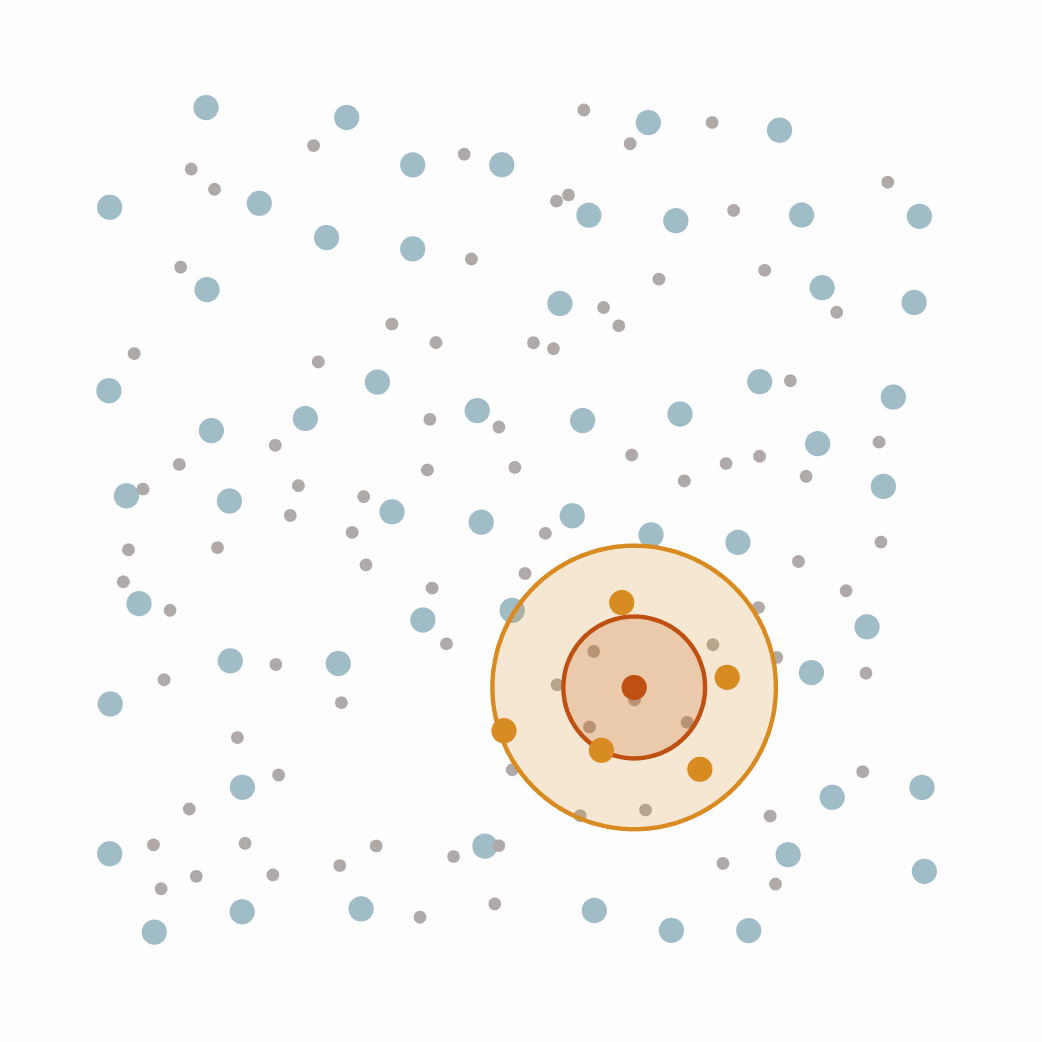}
	\includegraphics[width=0.49\textwidth]{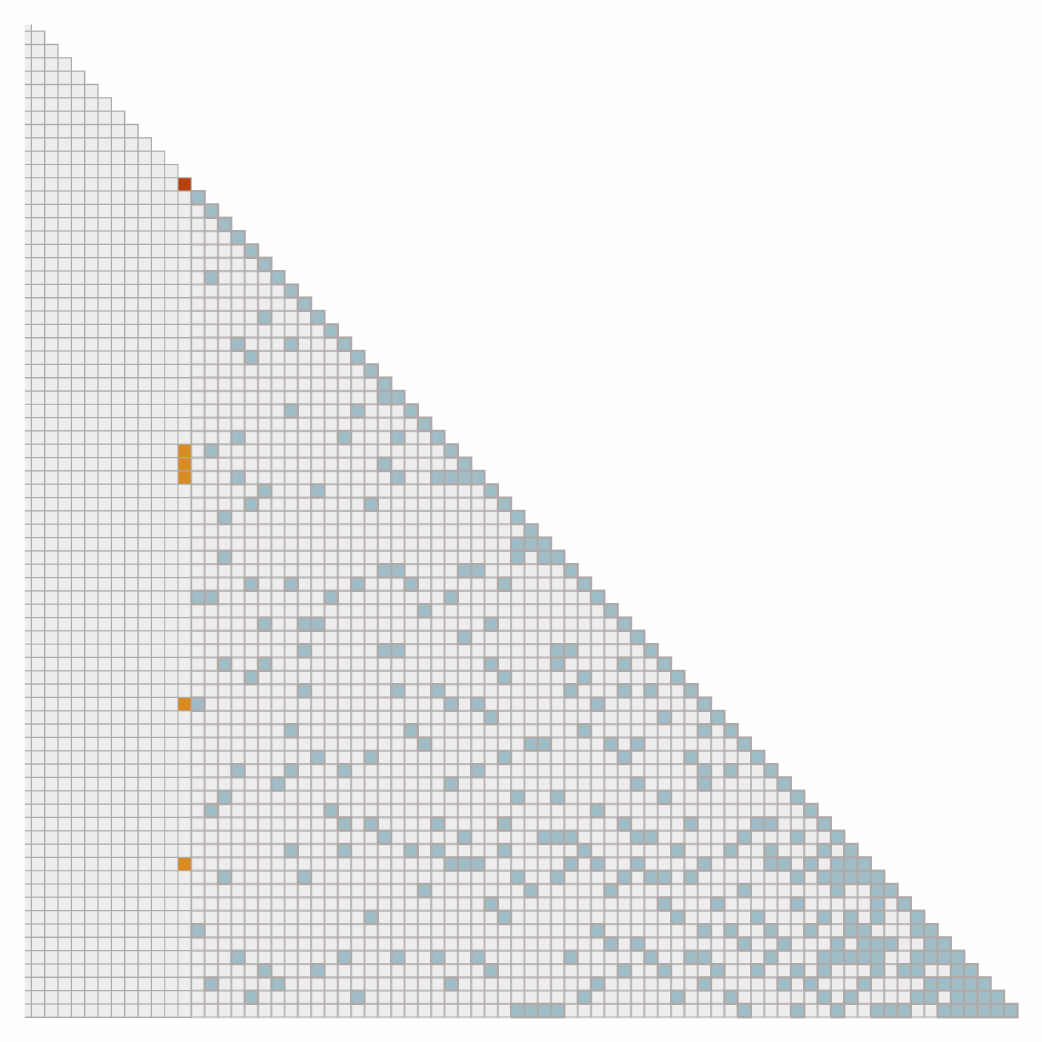}
	\caption{To obtain the reverse-maximin ordering, for $k=N-1,N-2,\ldots,1$, we successively select the {\color{darkrust} point $x_{i_k}$} that has the largest distance {\color{darkrust}$\ell_{i_k}$} to those {\color{darklightblue}points $x_{i_{k+1}},\ldots,x_{i_N}$ selected previously (shown as enlarged)}. All previously selected points {\color{darkorange} within distance $\rho \ell_i$} of {\color{darkrust}$x_{i_k}$} (here, $\rho = 2$) form the {\color{orange} $k$-th column of the sparsity pattern}. \label{fig:sortsparse}}
\end{figure}

Assume that $\K$ is the covariance function of a Gaussian process that is conditioned to be zero on (the possibly empty set) $\partial \Omega$, and the kernel matrix $\KM \in \Reals^{\I \times \I}$ is obtained as $\KM_{ij} \defeq \K(x_i,x_j)$ for a set of locations $\left\{x_i\right\}_{i \in \I} \subset \Omega$.

The \emph{reverse maximum-minimum distance (reverse-maximin) ordering} \cite{guinness2016permutation,Schafer2017} of $\left\{x_i\right\}_{i \in \I}$
is achieved by selecting the last index as
\begin{equation}
  i_N \defeq \argmax_{i \in \I} \dist\left(x_i,\partial \Omega\right)
\end{equation}
(or arbitrarily for $\partial\Omega = \emptyset$),
and then choosing sequentially for $k=N-1,N-2,\ldots,1$ the index that is furthest away from $\partial \Omega$ and those indices that were already picked:
\begin{equation}
  i_{k} \defeq \argmax_{i \in \I \setminus\left\{i_{k+1},\dots,i_N\right\}} 
  \dist \left( x_i, \left\{x_{i_{k+1}}, \dots, x_{i_N}\right\} \cup \partial \Omega \right).
\end{equation}
Write $\ell_{i_k}=\dist \left( x_{i_k}, \left\{x_{i_{k+1}}, \dots, x_{i_N}\right\} \cup \partial \Omega \right)$, and write $i \prec j$ if $i$ precedes $j$ in the reverse-maximin ordering.
We collect the $\left\{\ell_i\right\}_{i \in \I}$ into a vector denoted by $\ell$.

For a tuning parameter $\rho \in \Reals^+$, we select the sparsity set $S_{\prec,\ell,\rho} \subset \I \times \I$ as
\begin{equation}
  S_{\prec,\ell,\rho} \defeq \left\{ (i,j): i \succeq j, \dist( x_i, x_j ) \leq \rho \ell_j \right\}.
\end{equation}
The reverse-maximin ordering and sparsity pattern is illustrated in \cref{fig:sortsparse}.

By a minor adaptation of \cite[Alg.~3]{Schafer2017}, the reverse-maximin ordering and sparsity pattern can be constructed using \cref{alg:sortSparseTrunc} (see \cref{apsec:sortSparse}) in computational complexity $\O(N \log^2(N) \rho^{\tilde{d}})$ in time and $\O(N \rho^{\tilde{d}})$ in space, where $\tilde{d} \leq d$ is the intrinsic dimension of the dataset, as will be defined in \cref{cond:intrinsicDimension}.
The inverse Cholesky factors $L$ can then be computed using Equation~\eqref{eqn:defcolL}, as in \cref{alg:notAggregated}.
\begin{figure}[t]
	\begin{minipage}[t]{6.0cm}
		\vspace{0pt}
		\begin{algorithm}[H]
		    \textbf{Input:}  $\K$, $\left\{x_i\right\}_{i \in \I}$, $\prec$, $S_{\prec,l,\rho}$ \\
            \textbf{Output:} $L\in \Reals^{N \times N}$ l. triang. in $\prec$\\
		    \begin{algorithmic}[1]
    			\FOR{$k \in \I$}
    			    \FOR{ $i,j \in s_{k}$}
    			        \STATE $\left(\KM_{s_k,s_k}\right)_{ij} \leftarrow \K(x_i,x_j)$
    			    \ENDFOR 
    			\STATE $L_{s_k,k} \leftarrow \KM_{s_k,s_k}^{-1} \mathbf{e}_k$
    			\STATE $L_{s_k,k} \leftarrow L_{s_k,k} / \sqrt{L_{k,k}}$
    			\ENDFOR
    			\RETURN $L$
			\end{algorithmic}
			\caption{\label{alg:notAggregated}Without aggregation}
		\end{algorithm}
	\end{minipage}
	\begin{minipage}[t]{6.0cm}
		\vspace{0pt}
		\begin{algorithm}[H]
			\textbf{Input:} $\K$, $\left\{x_i\right\}_{i \in \I}$, $\prec$, $S_{\prec,l,\rho,\lambda}$\\
			\textbf{output:} $L\in \Reals^{N \times N}$ l. triang. in $\prec$\\
			\begin{algorithmic}[1]
			    \FOR{$\tilde{k} \in \tilde{\I}$}
			        \FOR{ $i,j \in s_{\tilde{k}}$}
			            \STATE $\left(\KM_{s_{\tilde{k}},s_{\tilde{k}}}\right)_{ij} \leftarrow \K(x_i,x_j)$
			        \ENDFOR
			        \STATE $U \leftarrow P^{\updownarrow}\chol( P^{\updownarrow} \KM_{s_{\tilde{k}}, s_{\tilde{k}}} P^{\updownarrow})P^{\updownarrow}$
			        \FOR{$k \leadsto \tilde{k}$}
     		    	    \STATE $L_{s_k,k} \leftarrow U^{-\top} \mathbf{e}_k$
			        \ENDFOR
			    \ENDFOR
			    \RETURN $L$
			\end{algorithmic}
			\caption{\label{alg:aggregated} With aggregation}
		\end{algorithm}
	\end{minipage}
	\caption{KL-minimization with and without using aggregation. For notational convenience, all matrices are assumed to have row-- and column ordering according to $\prec$. $P^{\updownarrow}$ denotes the order-reversing permutation matrix, and $\mathbf{e}_k$ is the vector with $1$ in the $k$-th component and zero elsewhere.} 
\end{figure}

\subsection{Aggregated sparsity pattern\label{ssec:aggregated}}

It was already observed by \cite{ferronato2015novel} in the context of sparse approximate inverses, and by \cite{stein2004approximating,guinness2016permutation} in the context of the Vecchia approximation, that a suitable grouping of the degrees of freedom makes it possible to \emph{reuse} Cholesky factorizations of the matrices $\KM_{s_i,s_i}$ in Equation~\eqref{eqn:defcolL} to update multiple columns at once.
The authors of \cite{guinness2016permutation,ferronato2015novel} propose grouping heuristics based on the sparsity graph of $L$ and show empirically that they
lead to improved performance.
In contrast, we propose a grouping procedure based on geometric information and prove rigorously that it allows us to reach the best asymptotic complexity in the literature, in a more concrete setting.

Assume that we have already computed the reverse-maximin ordering $\prec$ and sparsity pattern $S_{\prec, \ell, \rho}$, and that we have access to the $\ell_i$ as defined above. 
We will now aggregate the points into groups called \emph{supernodes}, consisting of points that are close in both location and ordering. To do so, we pick at each step the first (w.r.t.\ $\prec$) 
index $i \in \I$ that has not been aggregated into a supernode yet and then we aggregate into a common supernode the indices in $\{j: (i,j) \in S_{\prec, \ell, \rho}, \ell_j \leq \lambda \ell_i\}$ for some $\lambda > 1$ ($\lambda \approx 1.5$ is typically a good choice) that have not been aggregated yet. 
We proceed with this procedure until every node has been aggregated into a supernode.
We write $\tilde{I}$ for the set of all supernodes; for $i \in \I, \tilde{i} \in  \tilde{I}$, we write $i \leadsto \tilde{i}$ 
if $\tilde{i}$ is the supernode to which $i$ has been aggregated.
We furthermore define $s_{\tilde{i}} \defeq \left\{ j : \exists i \leadsto \tilde{i}, j \in s_i\right\}$ and introduce the aggregated sparsity pattern $\tilde{S}_{\prec,\ell,\rho,\lambda} \defeq \bigcup_{k \leadsto \tilde{k}} \left\{ (i,k): k \preceq i \in s_{\tilde{k}}\right\}$.
This sparsity pattern, while larger than $S_{\prec,\ell,\rho}$, can be represented efficiently by keeping track of the set of \emph{parents} (the $k \in I$ such that $k \leadsto s_{\tilde{k}}$) and \emph{children} (the $i \in s_{\tilde{k}}$) of each supernode, rather than the individual entries (see \cref{fig:aggregation} for an illustration).
For well-behaved (cf.\ \cref{thm:complexity}) sets of points, we obtain $\O(N \rho^{-d})$ supernodes, each with $\O(\rho^d)$ parents and children, thus improving the cost of storing the sparsity pattern from $\O(N \rho^{d})$ to $\O(N)$.

\begin{figure}
	\centering
		\begin{tikzpicture}[scale=0.7]
			\input{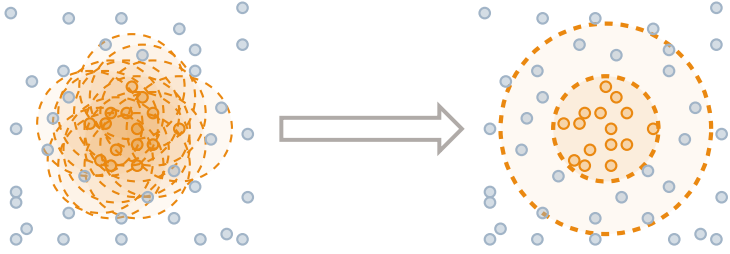}
		\end{tikzpicture}
    \caption{The left figure illustrates the original pattern $S_{\prec,\ell,\rho}$. For each orange point, we need to keep track of its interactions with all points within a circle of radius $\approx \rho$. In the right figure, the points have been collected into a supernode, which can be represented by a list of \emph{parents} (the orange points within an inner sphere of radius $\approx \rho$) and \emph{children} (all points within a radius $\approx 2\rho$).  }
    \label{fig:aggregation}
\end{figure}

While the above aggregation procedure can be performed efficiently once $\prec$ and $S_{\prec,\ell,\rho}$ are computed, it is possible 
to directly compute $\prec$ and an outer approximation $\bar{S}_{\prec,\ell,\rho,\lambda} \supset \tilde{S}_{\prec,\ell,\rho,\lambda}$ in computational complexity $\O(N)$ in space and $\O(N \log(N))$ in time.
$\bar{S}_{\prec,\ell,\rho,\lambda}$ can either be used directly, or it can be used to compute $\tilde{S}_{\prec,\ell,\rho,\lambda}$ in $\O(N)$ in space and $\O(N \log(N) \rho^d)$ in time, using a simple and embarrassingly parallel algorithm.
Details are given in \cref{apsec:sortSparse}.

In addition to reducing the memory cost, the aggregated ordering and sparsity pattern allows us to compute the Cholesky factors (in reverse ordering) $\KM_{s_{\tilde{k}},s_{\tilde{k}}} = UU^{\top}$ once for each supernode and then use it to compute the $L_{s_k,k}$ for all $k \leadsto \tilde{k}$ as in \cref{alg:aggregated} (see \cref{fig:supernodes} for an illustration).
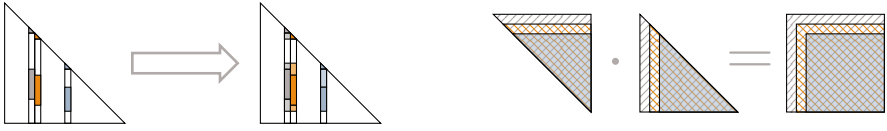
\begin{figure}
		\centering
		\begin{minipage}{.4\textwidth}
			\begin{tikzpicture}[scale=0.40]
	\draw (-4, 4) -- (-4, 0) -- (0, 0) -- cycle; 
	\filldraw[fill=silver] (-3.2, 3.2) -- (-3.2, 3.00 ) -- (-3.00, 3.00) -- cycle;
	\draw (-3.2, 3.00) rectangle (-3.0, 0.00);
	\draw[fill=silver] (-3.2, 1.80) rectangle (-3.0, 0.80);

	\filldraw[fill=orange] (-3.0, 3.0) -- (-3.0, 2.80 ) -- (-2.80, 2.80) -- cycle;
	\draw (-3.0, 2.80) rectangle (-2.8, 0.00);
	\draw[fill=orange] (-3.0, 1.60) rectangle (-2.8, 0.60);

	\filldraw[fill=lightblue] (-2.0, 2.0) -- (-2.0, 1.8 ) -- (-1.80, 1.8) -- cycle;
	\draw (-2.0, 1.80) rectangle (-1.8, 0.00);
	\draw[fill=lightblue] (-2.0, 1.20) rectangle (-1.8, 0.40);

	\node[draw, color=silver, line width=1.0, single arrow,
		minimum height=14mm,
		minimum width=2mm,
      single arrow head extend=1mm,
		anchor=west, rotate=0] at (0.2,2) {};
		
\begin{scope}[xshift=8.5cm]
	\draw (-4, 4) -- (-4, 0) -- (0, 0) -- cycle; 
	\filldraw[fill=silver] (-3.2, 3.2) -- (-3.2, 3.00 ) -- (-3.00, 3.00) -- cycle;
	\draw (-3.2, 3.00) rectangle (-3.0, 0.00);
	\draw[fill=silver, fill opacity = 0.6] (-3.2, 3.0) rectangle (-3.0, 2.80);
	\draw[fill=silver] (-3.2, 1.80) rectangle (-3.0, 0.80);
	\draw[fill=silver, fill opacity = 0.6] (-3.2, 2.00) rectangle (-3.0, 1.80);
	\draw[fill=silver, fill opacity = 0.6] (-3.2, 0.8) rectangle (-3.0, 0.40);

	\filldraw[fill=orange] (-3.0, 3.0) -- (-3.0, 2.80 ) -- (-2.80, 2.80) -- cycle;
	\draw (-3.0, 2.80) rectangle (-2.8, 0.00);
	\draw[fill=orange, fill opacity = 0.6] (-3.0, 2.00) rectangle (-2.8, 1.60);
	\draw[fill=orange] (-3.0, 1.60) rectangle (-2.8, 0.60);
	\draw[fill=orange, fill opacity = 0.6] (-3.0, 0.60) rectangle (-2.8, 0.40);

	\filldraw[fill=lightblue] (-2.0, 2.0) -- (-2.0, 1.8 ) -- (-1.80, 1.8) -- cycle;
	\draw (-2.0, 1.80) rectangle (-1.8, 0.00);
	\draw[fill=lightblue, fill opacity = 0.6] (-2.0, 1.80) rectangle (-1.8, 1.20);
	\draw[fill=lightblue] (-2.0, 1.20) rectangle (-1.8, 0.40);
\end{scope}
			\end{tikzpicture}
		\end{minipage}
		~~~~~~~~~
		\begin{minipage}{.4\textwidth}
			\begin{tikzpicture}[scale=0.65]
				    \draw[pattern=north east lines, pattern color=silver] (-2.0, 2.0) -- (-2.0, 0.0) -- (0.0, 0.0) -- cycle;   
		\node at (-2.5, 1.0) {\Huge \textcolor{silver}{$\cdot$}};

		\draw[pattern=north west lines, pattern color=orange] (-1.8, 1.8) -- (-1.8, 0.0) -- (0.0, 0.0) -- cycle;   

		\filldraw[fill = lightblue, fill opacity = 0.5] (-1.6, 1.6) -- (-1.6, 0.0) -- (0.0, 0.0) -- cycle;

		\begin{scope}[xshift=-3.0cm, yscale=-1, rotate=90]
			\draw[pattern=north east lines, pattern color=silver] (-2.0, 2.0) -- (-2.0, 0.0) -- (0.0, 0.0) -- cycle;   

			\draw[pattern=north west lines, pattern color=orange] (-1.8, 1.8) -- (-1.8, 0.0) -- (0.0, 0.0) -- cycle;   

			\filldraw[fill = lightblue, fill opacity = 0.5] (-1.6, 1.6) -- (-1.6, 0.0) -- (0.0, 0.0) -- cycle;   
		\end{scope}

		\node at (0.25, 1.0) {\Huge \textcolor{silver}{$=$}};

		\draw[pattern=north east lines, pattern color=silver] (1.0, 2.0) rectangle (3.0, 0.0);

    \draw[pattern=north west lines, pattern color=orange] (1.2, 1.8) rectangle (3.0, 0.0);

    \filldraw[fill = lightblue, fill opacity = 0.5] (1.4, 1.6) rectangle (3.0, 0.0);
			\end{tikzpicture}
		\end{minipage}
    \caption{(Left:) By adding a few nonzero entries to the sparsity pattern, the sparsity patterns of columns in $s_{\tilde{k}}$ become subsets of one another. 
    (Right:) Therefore, the matrices $\{\KM_{s_k,s_k}\}_{k \leadsto \tilde{k}}$, which need to be be inverted to compute the columns $L_{:,k}$ for $k \leadsto \tilde{k}$, become submatrices of one another. 
    Thus, submatrices of the Cholesky factors of $\KM_{s_{\tilde{k}}, s_{\tilde{k}}}$ can be used as factors of $\KM_{s_k,s_k}$ for any $k \leadsto \tilde{k}$.} 
    \label{fig:supernodes}
\end{figure}

As we show in the next section, this allows us to reduce the computational complexity from $\O(N\rho^{3d})$ to $\O(N\rho^{2d})$ for sufficiently well-behaved sets of points.

\subsection{Theoretical guarantees}
\label{sec:theory}

We now present our rigorous theoretical result bounding the computational complexity and approximation error of our method. 
Proofs and additional details are deferred to \cref{apsec:proofs}.

\begin{remark}
	As detailed in \cref{apsec:proofs}, the results below apply to more general \emph{reverse $r$-maximin} orderings, which can be computed in complexity $\O(N\log(N))$, improving over reverse-maximin orderings by a factor of $\log(N)$.
\end{remark}

\subsubsection{Computational complexity}

We can derive the following bounds on the computational complexity depending on $\rho$ and $N$.

\begin{theorem}[Informal]
\label{thm:complexity}
    Under mild assumptions on $\{x_i\}_{i \in \I} \subset \Reals^{d}$, the KL-minimizer $L$ is computed in complexity $CN\rho^{d}$ in space and $CN\rho^{3d}$ in time when using \cref{alg:notAggregated} with $S_{\prec, \ell, \rho}$
    and in complexity $CN\rho^{d}$ in space and $C_{\lambda, \ell}CN\rho^{2d}$ in time when using \cref{alg:aggregated} with $\tilde{S}_{\prec, \ell, \rho, \lambda}$.
    Here, the constant $C$ depends only on $d$, $\lambda$, and the cost of evaluating entries of $\KM$.
\end{theorem}
A more formal statement and a proof of \cref{thm:complexity} can be found in \cref{apsec:proofs}.

As can be seen from \cref{thm:complexity}, using the aggregation scheme decreases the computational cost by a factor $\rho^d$. 
This is because each supernode has $\approx \rho^d$ members that can all be updated by reusing the same Cholesky factorization.

\begin{remark}
	As described in \cref{apsec:proofs}, the computational complexity only depends on the intrinsic dimension of the dataset (as opposed to the potentially much larger ambient dimension $d$). 
	This means that the algorithm automatically exploits low-dimensional structure in the data to decrease the computational complexity. 
\end{remark}

\subsubsection{Approximation error}
\label{sssec:appproxerr}

We derive rigorous bounds on the approximation error from results on the localization of stiffness matrices of \emph{gamblets} (a class of operator-adapted wavelets) proved by \cite{owhadi2017universal, OwhScobook2018}, and their interpretation as Cholesky factors introduced by \cite{Schafer2017}.
Thus, the bounds hold in the setting of the above references. 
We assume for the purpose of this section that $\Omega$ is a bounded domain of $\mathbb{R}^d$ with Lipschitz boundary, and for an 
integer $s > d/2$, we write  $H_0^{s}\left(\Omega\right)$ for usual Sobolev the space of functions with zero Dirichlet boundary values and order $s$ derivatives in $L^2$, and $H_0^{-s}\left(\Omega\right)$ for its dual. 
Let the operator
\begin{equation}
  \IK: H_0^{s}\left(\Omega\right) \mapsto H^{-s}\left(\Omega\right), 
\end{equation}
be linear, symmetric ($\int u \IK v= \int v \IK u$), positive ($\int u \IK u\geq 0$), bijective, bounded (write $\|\IK\|:=\sup_u \|\IK u\|_{H^{-s}(\Omega)}/\|u\|_{H^s_0(\Omega)}$ for its operator norm), and local in the sense that $\int u \IK v \dx = 0$, for all 
$u,v \in H^s_0\left(\Omega\right)$ with disjoint support.
By the Sobolev embedding theorem, we have $H_0^{s}\left(\Omega\right) \subset  C_0\left( \Omega \right)$ and hence 
$\left\{ \updelta_x \right\}_{x \in \Omega} \subset H^{-s}\left(\Omega\right)$.
We then define $\K$ as the Green's function of $\IK$,
\begin{equation}
  \K\left(x_1,x_2\right) \defeq \int \updelta_{x_1} \IK^{-1} \updelta_{x_2} \dx.
\end{equation}
A simple example when $d=1$ and $\Omega = (0,1)$, is $\IK  = - \Delta $, and 
$\K(x,y) = \mathbbm{1}_{x<y} \frac{1 - y}{1 -x} + \mathbbm{1}_{y\leq x} \frac{y}{x} $.
Let us define the following measure of \emph{homogeneity} of 
the distribution of $\{x_i\}_{i \in \I}$,
\begin{equation}
\delta \defeq \frac{\min_{x_i,x_j \in \I}\, \dist(x_i, \{x_j\} \cup \partial \Omega)}
{\max_{x \in \Omega}\, \dist(x, \{x_i\}_{i \in \I} \cup \partial \Omega)}.
\label{eq:deltadef}
\end{equation}
Using the above definitions, we can rigorously quantify the 
increase in approximation accuracy as $\rho$ increases.
\begin{theorem}
\label{thm:accuracy}
There exists a constant $C$ depending only on $d$, $\Omega$, $\lambda$, $s$, $\|\IK\|$, $\|\IK^{-1}\|$, and $\delta$, such that  for $\rho \geq C \log(N/\epsilon)$, we have 
\begin{equation}
\textstyle    \KL*{\N\left(0,\KM\right)}{ \N(0, \left(L^\rho L^{\rho,\top})^{-1}\right)}  
   \, + \, \left\| \KM - ( L^{\rho} L^{\rho,\top} )^{-1} \right\|_{\FRO}
   \, \leq  \,\epsilon
\end{equation}
Thus, \cref{alg:notAggregated} computes an $\epsilon$-accurate approximation of $\KM$ in computational complexity $CN\log^d( N/\epsilon )$ in space and $CN \log^{3d}( N/\epsilon )$ in time, from  $CN\log^d( N/\epsilon )$ entries of $\KM$.
Similarly, \cref{alg:aggregated} computes an $\epsilon$-accurate approximation of $\KM$ in computational complexity $CN\log^d( N/\epsilon )$ in space and $CN \log^{2d}( N/\epsilon )$  in time, from $CN\log^d( N/\epsilon )$ entries of $\KM$.
\end{theorem}
To the best of our knowledge, the above result is the best known complexity/accuracy trade-off for kernel matrices based on Green's functions of elliptic boundary value problems.
Some related but slower or less practically useful approaches were presented in 
\cite{Schafer2017}, who showed that the Cholesky factors of $\KM$ (as opposed to those of $\KM^{-1}$) can be approximated in computational complexity $\O(N \log^2(N) \log^{2 d}(N/\epsilon))$ in time and $\O(N \log(N)\log^{d}(N/\epsilon))$ in space using zero-fill-in incomplete Cholesky factorization (\cref{alg:ichol}) applied to $\KM$. 
Similarly, they showed that the Cholesky factors of $\KM^{-1}$ can be approximated in computational complexity $\O(N  \log^{2 d}(N/\epsilon))$ in time and $\O(N \log^{d}(N/\epsilon))$ in space using zero-fill-in incomplete Cholesky factorization applied to $\KM^{-1}$.
While they also observed that the near-sparsity of the Cholesky factors of $\KM^{-1}$ implies that they can in principle can be computed in computational complexity $\O(N  \log^{2 d}(N/\epsilon))$ from entries of $\KM$ by a recursive algorithm (thus improving the complexity of inverting $\KM$), they did not provide an explicit algorithm for this purpose.
Indeed, we have found that recursive algorithms based on truncation are unstable to the point of being useless in practice, when used to compute the Cholesky factors of $\KM^{-1}$ from entries of $\KM$.

\subsubsection{Screening in theory and practice} 
The theory described in the last section covers any self-adjoined operator $\IK$ with an associated quadratic form 
\begin{equation*}
	\IK[u] \coloneqq \int_{\Omega} u \IK u \dx = \sum_{k = 0}^s \int \sigma^{(k)}(x) \|D^{(k)}u(x)\|^2 \dx
\end{equation*}
and $\sigma^{(s)} \in L^{2}(\Omega)$ positive almost everywhere.
That is, $\IK[u]$ is a weighted average of the squared norms of derivatives of $u$ and thus measures the roughness of $u$.
A Gaussian process with covariance function given by $\K$ has density $\sim \exp(-\IK[u]/2)$ and therefore assigns exponentially low probability to ``rough'' functions, making it a prototypical smoothness prior. 
\cite{Schafer2017} prove that these Gaussian processes are subject to an exponentially strong screening effect in the sense that, after conditioning a set of $\ell$-dense points, the conditional covariance of a given point decays exponentially with rate $\sim \ell^{-1}$, as shown in the first panel of \cref{fig:screening}.
The most closely related model in common use is the Mat{\'e}rn covariance function \cite{matern1960spatial} that is the Green's function of an elliptic PDE of order $s$, when choosing the ``smoothness parameter'' $\nu$ as $\nu = s - d/2$.  
While our theory only covers $s \in \Naturals$, \cite{Schafer2017} observe that Mat{\'e}rn kernels with non-integer values of $s$ and even the ``Cauchy class'' \cite{gneiting2004stochastic} seem to be subject to similar behavior.
In the second panel of \cref{fig:screening}, we show that as the distribution of conditioning points becomes more irregular, the screening effect weakens.
In our theoretical results, this is controlled by the upper bound on $\delta$ in \cref{eq:deltadef}.
The screening effect is significantly weakened close to the boundary of the domain, as illustrated in the third panel of \cref{fig:screening} (see also \cite[Section 4.2]{Schafer2017}). 
This is the reason why our theoretical results, different from the Mat{\'e}rn covariance, are restricted to Green's functions with zero Dirichlet boundary condition, which corresponds to conditioning the process to be zero on $\partial \Omega$.
A final limitation is that the screening effect weakens as we take the order of smoothness to infinity, obtaining, for instance the Gaussian kernel.
However, as described in \cite[Section 2.4]{Schafer2017}, this results in matrices that have efficient low-rank approximations instead. 

\begin{figure}
	\centering
	\includegraphics[width=0.3\textwidth]{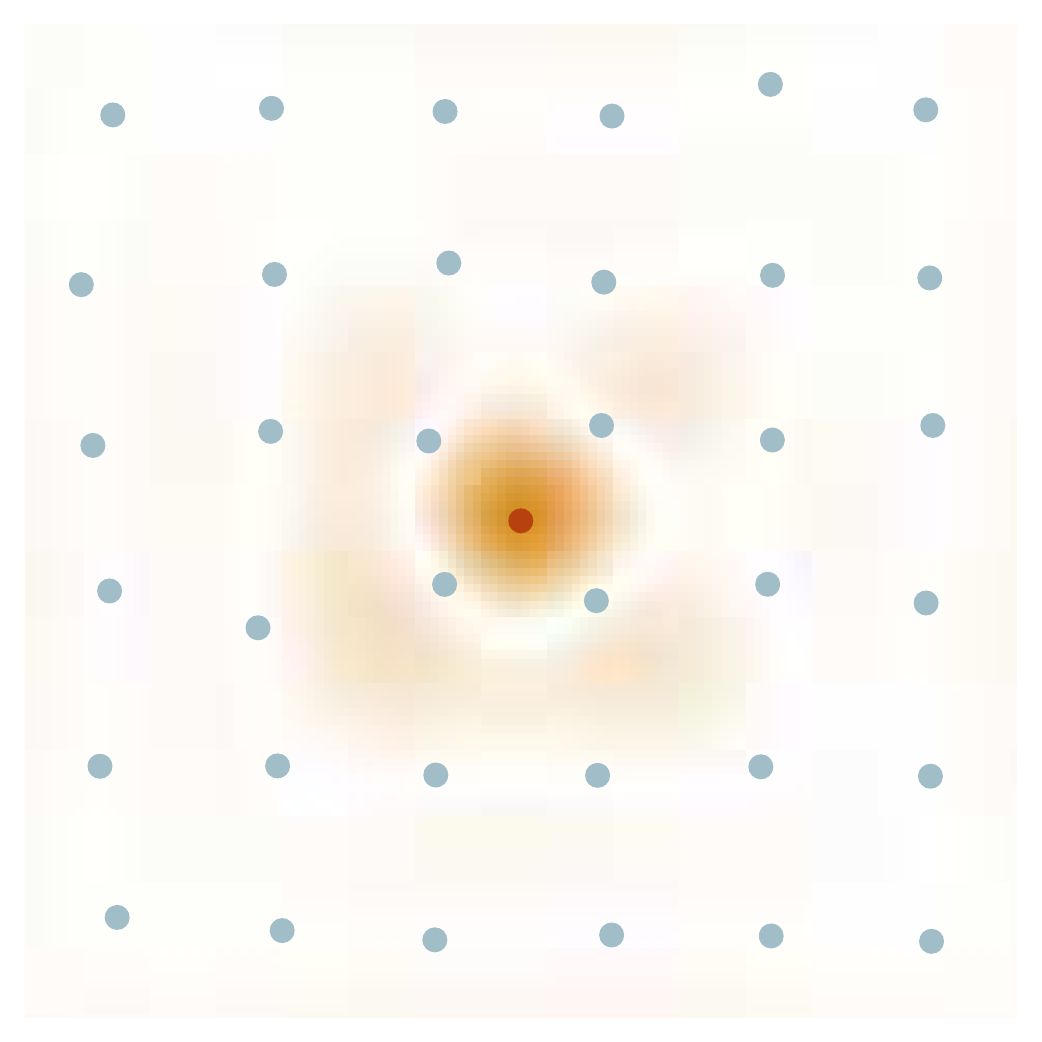}
	\includegraphics[width=0.3\textwidth]{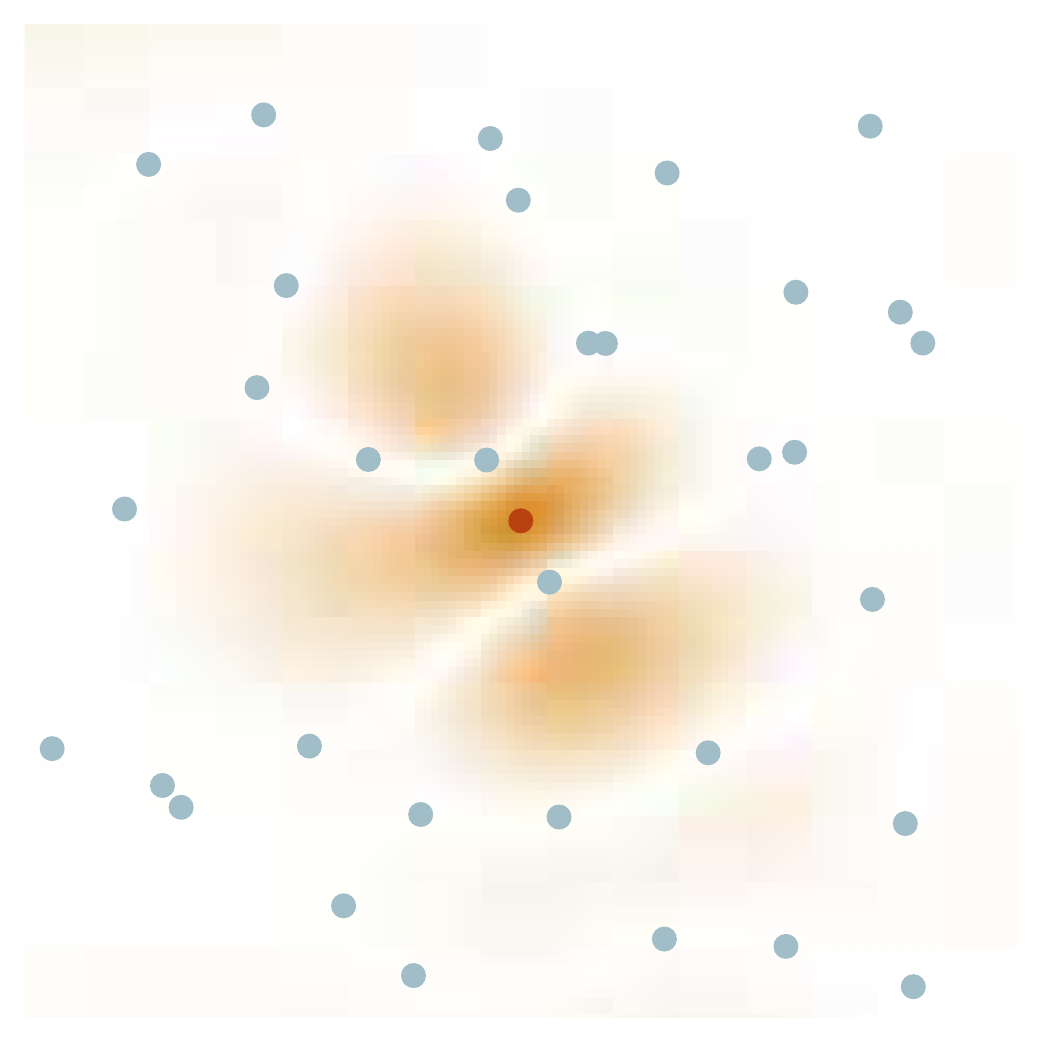}
	\includegraphics[width=0.3\textwidth]{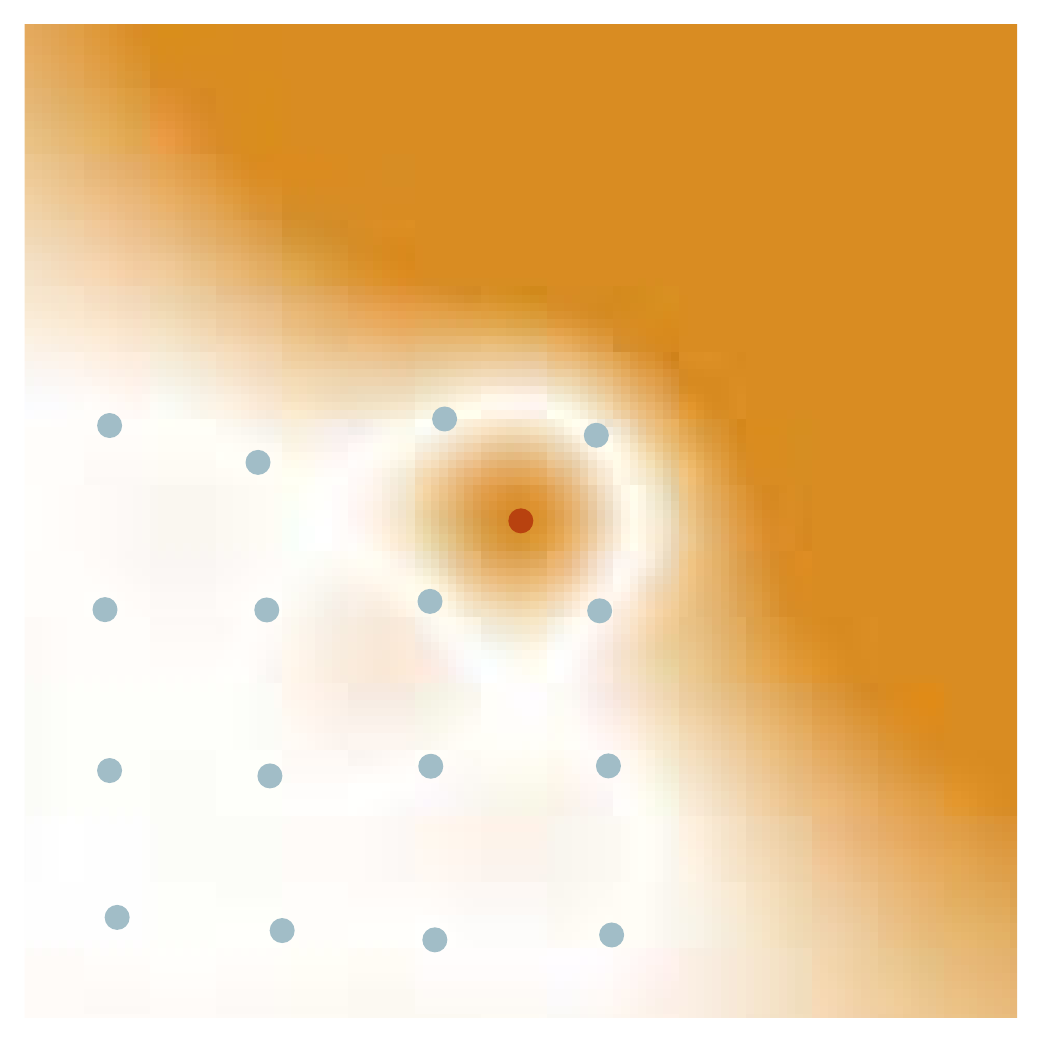}
	\caption{To illustrate the screening effect exploited by our methods, we plot the {\color{darkorange}conditional correlation} with {\color{darkrust} the point in red} conditional on {\color{darklightblue} the blue points}. In the first panel, the points are evenly distributed, leading to a rapidly decreasing conditional correlation. 
	In the second panel, the same number of points is irregularly distributed, slowing the decay.
	In the last panel, we are at the fringe of the set of observations, weakening the screening effect. \label{fig:screening}}
\end{figure}

\section{Extensions\label{sec:extensions}}

We now present extensions of our method that improve its performance in practice.
In \cref{ssec:noise}, we show how to improve the approximation when $\KM$ is replaced by $\KM + R$, for $R$ diagonal, as is frequently the case in statistical inference where $R$ is the covariance matrix of additive, independent noise. 
In \cref{sec:prediction}, we show how including the prediction points can improve the computational complexity (\cref{sssec:predfirst}) or accuracy (\cref{sssec:predlast}) of the posterior mean and covariance.
In \cref{ssec:lowmemory}, we discuss memory savings and parallel computation for GP inference when we are only interested in computing the likelihood and the posterior mean and covariance (as opposed to, for example, sampling from $\N(0,\KM)$ or computing products $v \mapsto \KM v$).

We note that it is not possible to combine the variant in \cref{ssec:noise} with that in \cref{ssec:lowmemory}, and that the combination of the variants in \cref{ssec:noise,sssec:predlast} might diminish  accuracy gains from the latter.
Furthermore, while \cref{ssec:lowmemory} can be combined with \cref{sssec:predfirst} to compute the posterior mean, this combination cannot be used to compute the full posterior covariance matrix. 

\subsection{Additive noise}
\label{ssec:noise}

Assume that a diagonal noise term is added to $\KM$, so that $\Sigma = \KM + R$, where $R$ is diagonal.
Extending the Vecchia approximation to this setting has been a major open problem  in  spatial statistics    \cite{Datta2016,Katzfuss2017a,Katzfuss2018}. 
Applying our approximation directly to $\Sigma$  would not work well because the noise term attenuates the exponential decay. 
Instead, given the approximation $\hat\KM^{-1}=L L^\top$ obtained using our method, we can write, following \cite{Schafer2017}: 
\[
\Sigma \approx \hat\KM + R = \hat\KM(R^{-1} + \hat\KM^{-1})R.
\]
Applying an incomplete Cholesky factorization with zero fill-in (\cref{alg:ichol}) to $R^{-1} + \hat\KM^{-1} \approx \tilde L \tilde L^\top$, we have
\[
\Sigma \approx (LL^\top)^{-1}\tilde L \tilde L^\top R.
\]
The resulting procedure, given in \cref{alg:withNoise}, has asymptotic complexity $\O(N \rho^{2d} )$, because every column of the triangular factors has at most $\O(\rho^d)$ entries.

Following the intuition that $\KM^{-1}$ is \emph{essentially} an elliptic partial differential operator, $\KM^{-1} + R^{-1}$ is \emph{essentially} a partial differential operator with an added zero-order term, and its Cholesky factors can thus be expected to satisfy an exponential decay property just as those of $\KM^{-1}$. 
Indeed, as observed by \cite[Fig.~2.3]{Schafer2017}, the exponential decay of the Cholesky factors of $R^{-1} + \KM^{-1}$ is as strong as for $\KM^{-1}$, even for large $R$.
We suspect that this could be proved rigorously by adapting the proof of exponential decay in \cite{OwhScobook2018} to the discrete setting.
We note that while independent noise is most commonly used, the above argument leads to an efficient algorithm whenever $R^{-1}$ is approximately given by an elliptic PDE (possibly of order zero).

For small $\rho$, the additional error introduced by the incomplete Cholesky factorization can harm accuracy, which is why we recommend using the conjugate gradient algorithm  (CG) to invert $(R^{-1} + \hat{\KM}^{-1})$ using $\tilde{L}$ as a preconditioner.
In our experience, CG converges to single precision in a small number of iterations ($\sim 10$).

\begin{figure}[t]
	\begin{minipage}[t]{6.3cm}
		\vspace{0pt}
		\begin{algorithm}[H]
		
			\textbf{Input:} $\K$, $\left\{x_i\right\}_{i \in \I}$, $\rho$, ($\lambda$,) and $R$\\
		    \textbf{Output:} $L, \tilde{L} \in \Reals^{N \times N}$ l.\ triang. in $\prec$\\
		    \begin{algorithmic}[1]
    			\STATE Comp. $\prec$ and $S \leftarrow S_{\prec,\ell,\rho}$($S_{\prec,\ell,\rho, \lambda}$)
    			\STATE Comp. $L$ using Alg.~\ref{alg:notAggregated}(\ref{alg:aggregated})
    			\FOR{$(i,j) \in S$}
    			    \STATE $A_{ij} \leftarrow \langle L_{i,:}, L_{j,:} \rangle$
    			\ENDFOR
    			\STATE $A \leftarrow A + R$
    			\STATE $\tilde{L} \leftarrow \texttt{ichol}(A,S)$
    			\RETURN $L$, $\tilde{L}$
    		\end{algorithmic}
			\caption{\label{alg:withNoise}Including independent noise with covariance matrix $R$ }
		\end{algorithm}
	\end{minipage}
	\begin{minipage}[t]{6.5cm}
		\vspace{0pt}
		\begin{algorithm}[H]
            \textbf{Input:} $A \in \Reals^{N \times N}$, $S$\\
			\textbf{Output:} {$L\in \Reals^{N \times N}$ l.\ triang. in $\prec$}\\
			\begin{algorithmic}[1]
			    \STATE $L \leftarrow (0,\ldots,0)(0,\ldots,0)^{\top}$
			    \FOR{$j \in \{ 1, \dots, N \}$}
			    	\FOR{$i \in \{ j, \dots, N \}$ : $(i, j) \in S$}
			    	    \STATE $L_{ij} \leftarrow A_{ij} - \langle L_{i,1:(j-1)}, L_{j,1:(j-1)} \rangle $
			    	\ENDFOR
			        \STATE $L_{:i} \gets A_{:i} / \sqrt{A_{ii}}$
			    \ENDFOR
			    \RETURN $L$
			\end{algorithmic}
			\caption{\label{alg:ichol} Zero fill-in incomplete Cholesky factorization ($\texttt{ichol(A,S)}$)}
		\end{algorithm}
	\end{minipage}
	\caption{Algorithms for approximating covariance matrices with added independent noise $\KM + R$ (left), using the zero fill-in incomplete Cholesky factorization (right). See \cref{ssec:noise}.}
\end{figure}

Alternatively, higher accuracy can be achieved by using the sparsity pattern of $L L^{\top}$ (as opposed to that of $L$) to compute the incomplete Cholesky factorization of $A$ in \cref{alg:withNoise}; in fact, in our numerical experiments in \cref{ssec:noisecomparison}, this approach was as accurate as using the exact Cholesky factorization of $A$ over a wide range of $\rho$ values and noise levels. The resulting algorithm still requires $\O(N \rho^{2d})$ time, albeit with a larger constant.
This is because for an entry $(i,j)$ to be part of the sparsity pattern of $L L^{\top}$, there needs to exist a $k$ such that both $(i,k)$ and $(j,k)$ are part of the sparsity pattern of $L$.
By the triangle inequality, this implies that $(i,j)$ is contained in the sparsity pattern of $L$ obtained by doubling $\rho$.
In conclusion, we believe that the above modifications allow us to compute an $\epsilon$--accurate factorization in $\O(N \log^{2d}(N/\epsilon))$ time and $\O(N \log^{d}(N/\epsilon))$ space, just as in the noiseless case.

\subsection{Including the prediction points}
\label{sec:prediction}

In GP regression, we are given $N_{\Train}$ points of training data and want to compute predictions at $N_{\Pred}$ points of test data.
We denote as $\KM_{\Train, \Train}$, $\KM_{\Pred, \Pred}$, $\KM_{\Train, \Pred}, \KM_{\Pred, \Train}$ the covariance matrix of the training data, the covariance matrix of the test data, and the covariance matrices of training and test data.
Together, they form the joint covariance matrix $\begin{psmallmatrix} \KM_{\Train, \Train} & \KM_{\Train, \Pred}  \\ \KM_{\Pred, \Train} & \KM_{\Pred, \Pred} \end{psmallmatrix}$ of training and test data.
In GP regression with training data $y \in \Reals^{N_{\Train}}$ we are interested in:
\begin{itemize}
    \item Computation of the log-likelihood $\sim y^{\top} \KM_{\Train, \Train}^{-1} y + \logdet \KM_{\Train, \Train} + N \log(2\pi)$
    \item Computation of the posterior mean $ y^{\top} \KM_{\Train, \Train}^{-1} \KM_{\Train,\Pred}$
    \item Computation of the posterior covariance $\KM_{\Pred,\Pred} - \KM_{\Pred,\Train} \KM_{\Train, \Train}^{-1} \KM_{\Train,\Pred}$
\end{itemize}
In the setting of \cref{thm:accuracy}, our method can be applied to accurately approximating the matrix $\KM_{\Train, \Train}$ in near-linear cost.
The training covariance matrix can then be replaced by the resulting approximation for all downstream applications.

However, approximating instead the joint covariance matrix of training and prediction variables improves (1)  stability and accuracy compared to computing the KL-optimal approximation of the training covariance alone,  (2)  computational complexity by circumventing the computation of most of the $N_{\Train} N_{\Pred}$ entries of the off-diagonal part $\KM_{\Train, \Pred}$ of the covariance matrix.

We can add the prediction points before or after the training points in the elimination ordering.

\subsubsection{Ordering the prediction points first, for rapid interpolation}
\label{sssec:predfirst}
 The computation of the mixed covariance matrix $\KM_{\Pred,\Train}$ can be prohibitively expensive when interpolating with a large number of prediction points.
This situation is  common in spatial statistics when estimating  a stochastic field throughout a large domain.
In this regime, we propose to order the $\left\{x_i\right\}_{i \in \I}$ by first computing the reverse maximin ordering $\prec_{\Train}$ of only the training points as described in \cref{ssec:maximin} using the original $\Omega$, writing $\ell_{\Train}$ for the corresponding length scales. 
We then compute the reverse maximin ordering $\prec_{\Pred}$ of the prediction points using the modified $\tilde{\Omega} \defeq \Omega \cup \left\{x_i\right\}_{i \in I_{\Train}}$, obtaining the length scales $\ell_{\Pred}$. 
Since $\tilde{\Omega}$ contains $\left\{x_i\right\}_{i \in I_{\Train}}$,
when computing the ordering of the prediction points, 
prediction points close to the training set will tend to have a smaller length-scale than in the naive application of the algorithm, and thus, the resulting sparsity pattern will have fewer nonzero entries.
We then order the prediction points before the training points and compute $S_{(\prec_{\Pred}, \prec_{\Train}),(\ell_{\Pred}, \ell_{\Train}),\rho}$ or $S_{(\prec_{\Pred}, \prec_{\Train}),(\ell_{\Pred}, \ell_{\Train}),\rho,\lambda}$ following the same procedure as in \cref{ssec:maximin,ssec:aggregated}, respectively.
The distance of each point in the prediction set to the training set can be computed in near-linear complexity using, for example, a minor variation of \cite[Alg.~3]{Schafer2017}.
Writing  $L$ for the resulting Cholesky factor of the joint precision matrix, we can approximate $\KM_{\Pred,\Pred} \approx L_{\Pred,\Pred}^{-\top} L_{\Pred,\Pred}^{-1}$ and $\KM_{\Pred,\Train} \approx  L_{\Pred,\Pred}^{-\top}L_{\Train,\Pred}^{\top}$ based on submatrices of $L$.
See \cref{apsssec:predfirst} and \cref{alg:predfirst} for additional details.
We note that the idea of ordering the prediction points first (last, in their notation) has already been proposed by \cite{Katzfuss2018} in the context of the Vecchia approximation, although without providing an explicit algorithm.

If one does not use the method in \cref{ssec:noise} to treat additive noise, then the method described in this section amounts to making each prediction using only $\O(\rho^d)$ nearby datapoints.
In the extreme case where we only have a single prediction point, this means that we are only using $\O(\rho^d)$ training values for prediction. 
On the one hand, this can lead to improved robustness of the resulting estimator, but on the other hand, it can lead to some training data being missed entirely.

\subsubsection{Ordering the prediction points last, for improved robustness}
\label{sssec:predlast}

If we want to use the improved stability of including the prediction points, maintain near-linear complexity, and use all $N_{\Train}$ training values for the prediction of even a single point, we have to include the prediction points \emph{after} the training points in the elimination ordering.
Naively, this would lead to a computational complexity of $\O(N_{\Train} (\rho^{d} + N_{\Pred})^2)$, which might be prohibitive for large values of $N_{\Pred}$.
If it is enough to compute the posterior covariance only among $m_{\Pred}$ small \emph{batches} of up to $n_{\Pred}$ predictions each (often, it makes sense to choose $n_{\Pred} = 1$), we can avoid this increase of complexity by performing prediction on groups of only $n_{\Pred}$ at once, with the computation for each batch only having computational complexity $\O(N_{\Train} (\rho^{d} + n_{\Pred})^2)$.
A naive implementation would still require us to perform this procedure $m_{\Pred}$ times, eliminating any gains due to the batched procedure.
However, careful use of the Sherman-Morrison-Woodbury matrix identity allows to to reuse the biggest part of the computation for each of the batches, thus reducing the computational cost for prediction and computation of the covariance matrix to only $\O(N_{\Train}( (\rho^{d} + n_{\Pred})^2 + (\rho^{d} + n_{\Pred}) m_{\Pred}))$.
This procedure is detailed in \cref{apsssec:predlast} and summarized in \cref{alg:predlast}.

\subsection{GP regression in \texorpdfstring{$\O(N + \rho^{2d})$}. space complexity}
\label{ssec:lowmemory}

When deploying direct methods for approximate inversion of kernel matrices, a major difficulty is the superlinear memory cost that they incur.
This, in particular, poses difficulties in a distributed setting or on graphics processing units.
In the following, $\I = \I_{\Train}$ denotes the indices of the training data, and we write $\KM \coloneqq \KM_{\Train, \Train}$, while $\I_{\Pred}$ denotes those of the test data.
In order to compute the log-likelihood, we need to compute the matrix-vector product $L^{\rho, \top}y$, as well as the diagonal entries of $L^{\rho}$. 
This can be done by computing the columns $L_{:,k}^{\rho}$ of $L^{\rho}$ individually using \cref{eqn:defcolL} and setting $(L^{\rho, \top}y)_k = (L_{:,k}^{\rho})^{\top} y$, $L^{\rho}_{kk} = (L_{:,k}^{\rho})_k$, without ever forming the matrix $L^{\rho}$.
Similarly, in order to compute the posterior mean, we only need to compute $\KM^{-1} y = L^{\rho,\top} L^{\rho}y$, which only requires us to compute each column of $L^{\rho}$ twice, without ever forming the entire matrix.
In order to compute the posterior covariance, we need to compute the matrix-matrix product $L^{\rho,\top} \KM_{\Train, \Pred}$, which again can be performed by computing each column of $L^{\rho}$ once without ever forming the entire matrix $L^{\rho}$. 
However, it does require us to know beforehand at which points we want to make predictions.
The submatrices $\KM_{s_i,s_i}$ for all $i$ belonging to the supernode $\tilde{k}$ (i.e., $i \leadsto \tilde{k}$) can be formed from a list of the elements of $\tilde{s}_k$.
Thus, the overall memory complexity of the resulting algorithm is $\O(\sum_{k \in \tilde{I}} \# \tilde{s}_k) = O(N_{\Train} + N_{\Pred} + \rho^{2d})$.
The above described procedure is implemented in \cref{alg:lowMemNotAggregated,alg:lowMemAggregated} in \cref{apssec:linearMemory}.
In a distributed setting with workers $W_1,W_2,\ldots$, this requires communicating only $\O(\# \tilde{s}_k)$ floating-point numbers to worker $W_k$, which then performs $\O((\# \tilde{s}_k)^3)$ floating-point operations; a naive implementation would require the communication of $\O((\# \tilde{s})^2)$ floating-point numbers to perform the same number of floating-point operations.

\section{Applications and numerical results}
\label{sec:numerics}

We conclude with numerical experiments studying the practical performance of our method.
The Julia code can be found under \url{https://github.com/f-t-s/cholesky\_by\_KL\_minimization}.

\subsection{Gaussian-process regression and aggregation}

\begin{figure}
    \centering
    \includegraphics[width=0.25\textwidth]{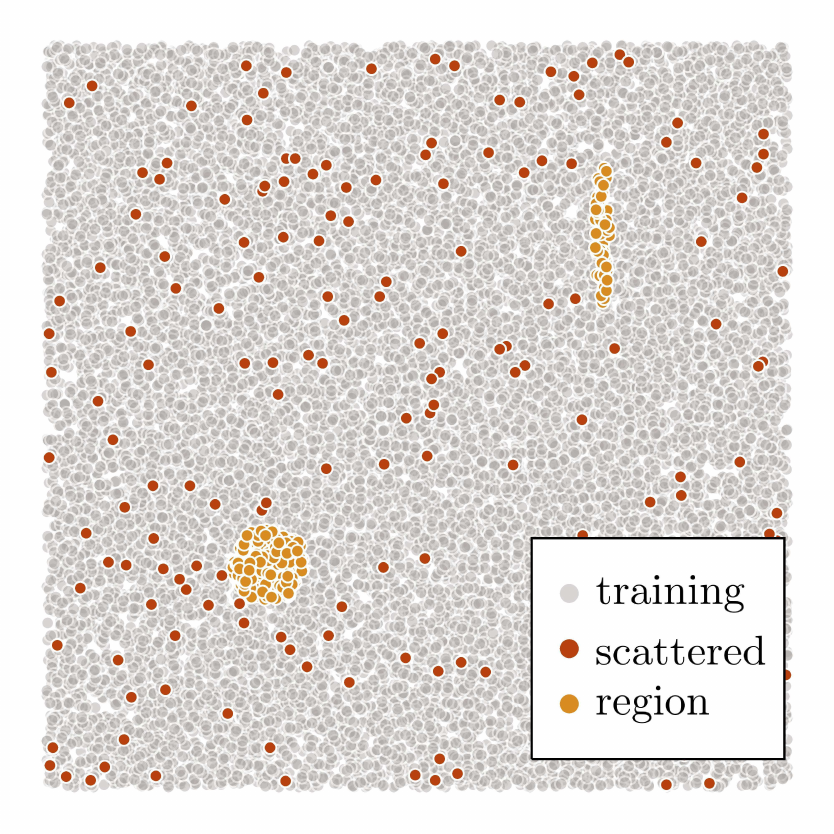}
    \includegraphics[width=0.35\textwidth]{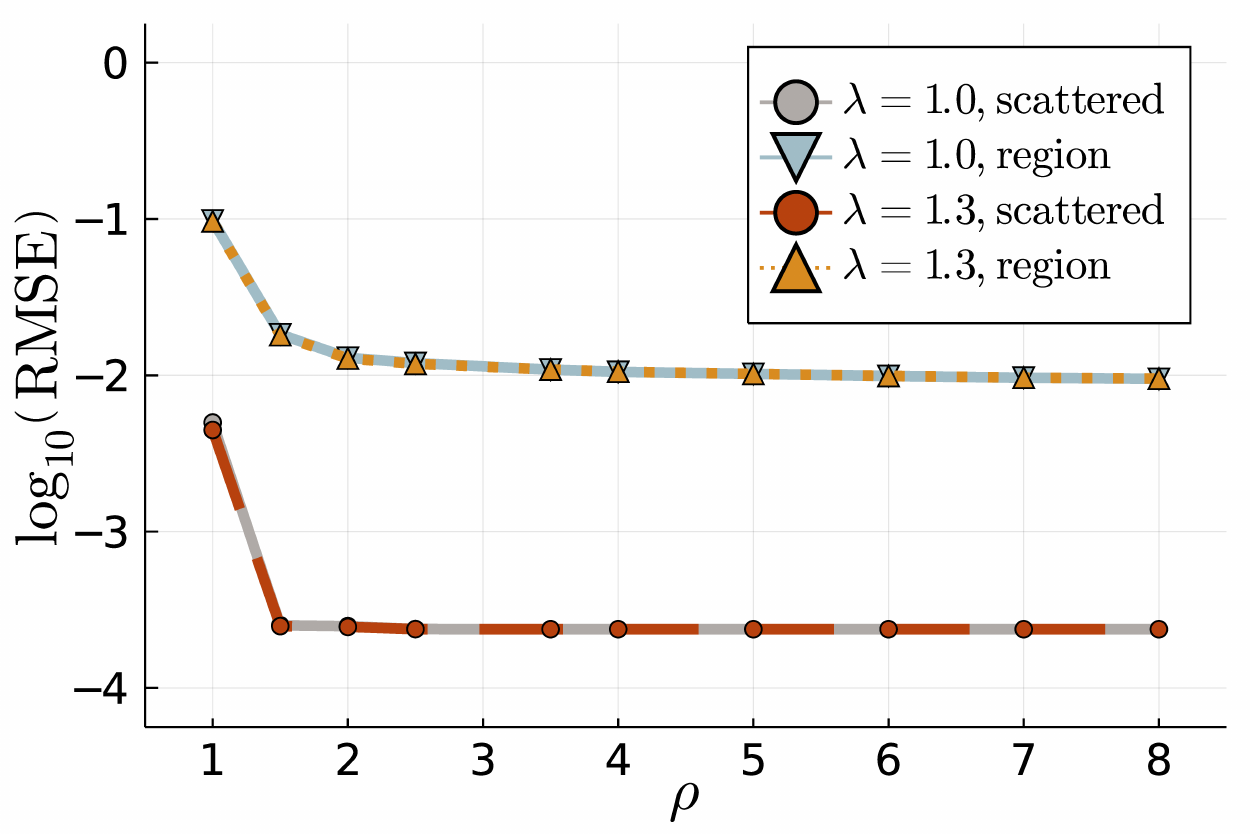}
    \includegraphics[width=0.35\textwidth]{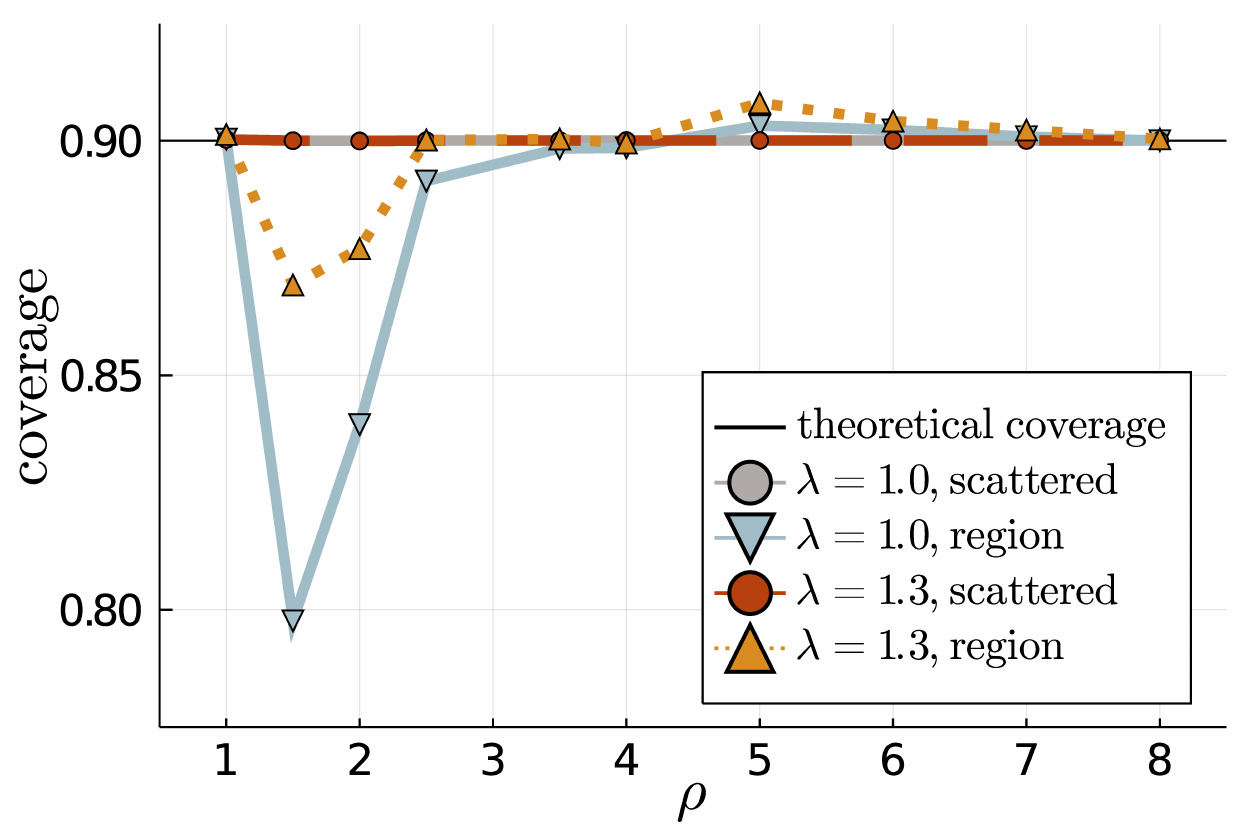}
    \caption{Accuracy of our approximation with and without aggregation for a Gaussian process with Mat{\'ern} covariance ($\nu = 3/2$) on a grid of size $10^6$ on the unit square. (Left:) Randomly sampled 2 percent of the training and prediction points. (Middle:) RMSE, averaged over prediction points and 1,000 realizations. (Right:) Empirical coverage of $90\%$ prediction intervals computed from the posterior covariance.}
    \label{fig:FFT}
\end{figure}

We begin our numerical experiments with two-dimensional ($d = 2$) synthetic data.
We use circulant embeddings \cite{stein2002fast,graham2018analysis}[\url{https://github.com/PieterjanRobbe/GaussianRandomFields.jl}] for the creation of $10^3$ samples of a Gaussian process with exponential covariance function at $10^6$ locations on a regular grid in $\Omega = [0,1]^2$. 
From these $10^6$ locations, we select $2 \times 10^4$ prediction points, and use the remaining points as training data.
As illustrated in \cref{fig:FFT} (left panel), half of the prediction points form two elliptic regions devoid of any training points (called ``region''), while the remaining prediction points are interspersed among the training points (called ``scattered'').
We then use the ``prediction points first'' approach of \cref{sssec:predfirst} and the aggregated sparsity pattern $\tilde{S}_{\prec,\ell,\rho,\lambda}$ of \cref{ssec:aggregated} with $\lambda \in  \left\{1.0, 1.3\right\}$, to compute the posterior distributions at the prediction points from the values at the training points.
In \cref{fig:FFT}, we report the RMSE of the posterior means, as well as the empirical coverage of the $90\%$ posterior intervals, averaged over all $10^3$ realizations, for a range of different $\rho$.
Note that while the RMSE between the aggregated ($\lambda = 1.3$) and non-aggregated ($\lambda = 1.0$) is almost the same, the coverage converges significantly faster to the correct value with $\lambda=1.3$.

\begin{figure}
    \centering
    \includegraphics[scale=0.145]{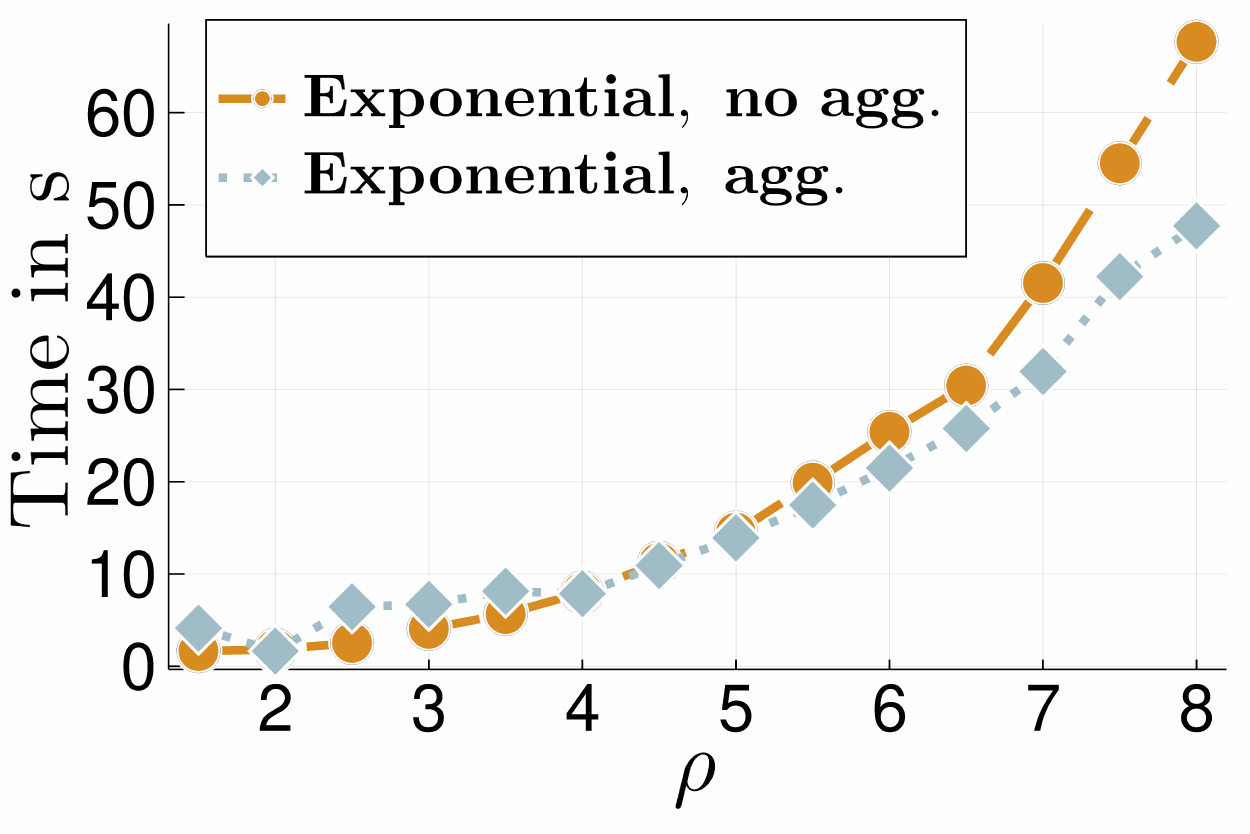}
    \includegraphics[scale=0.145]{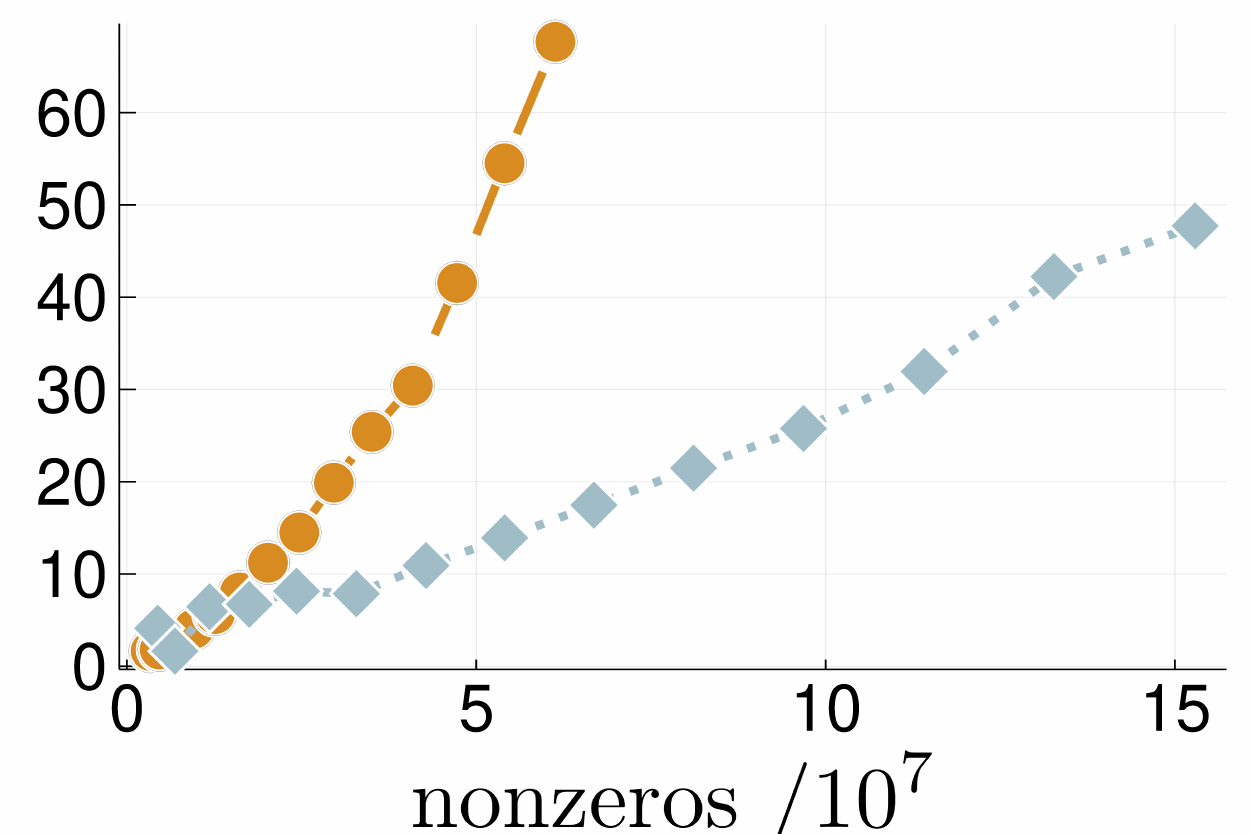}
    \includegraphics[scale=0.145]{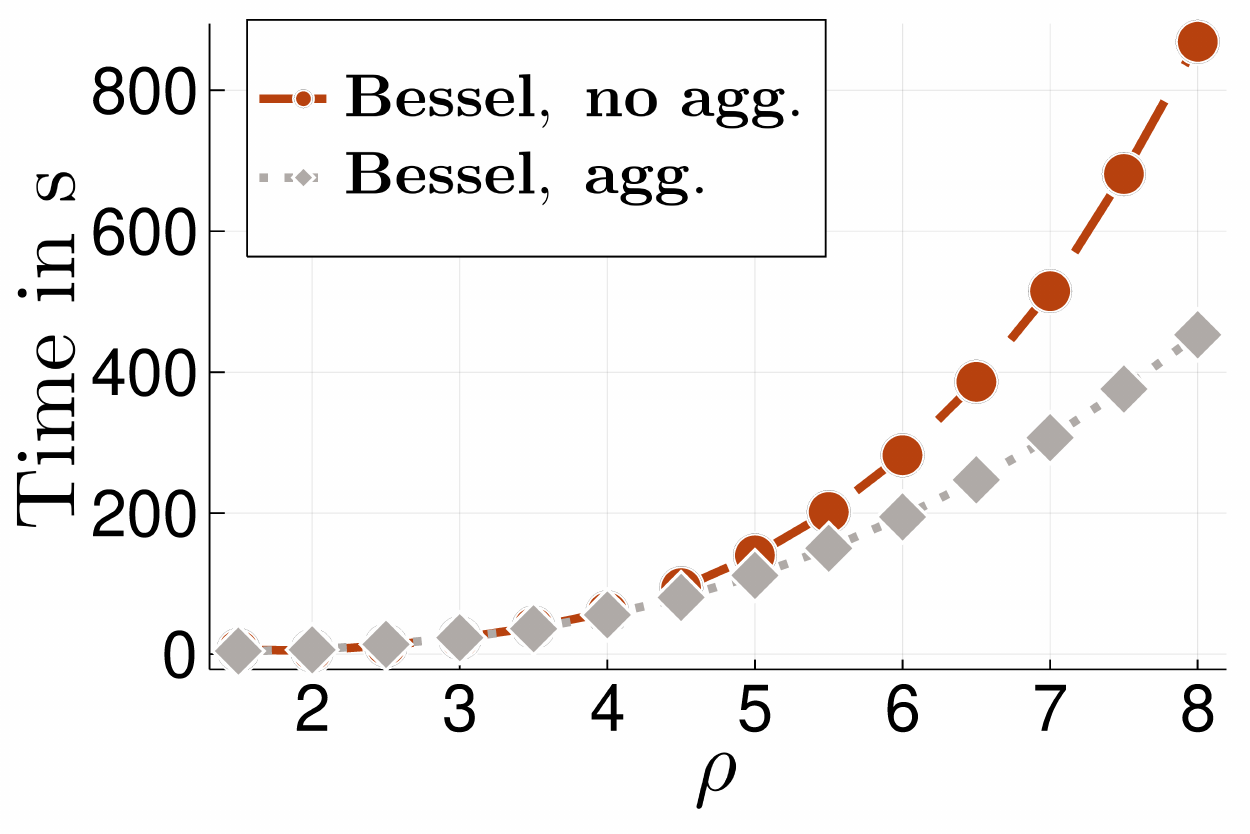}
    \includegraphics[scale=0.145]{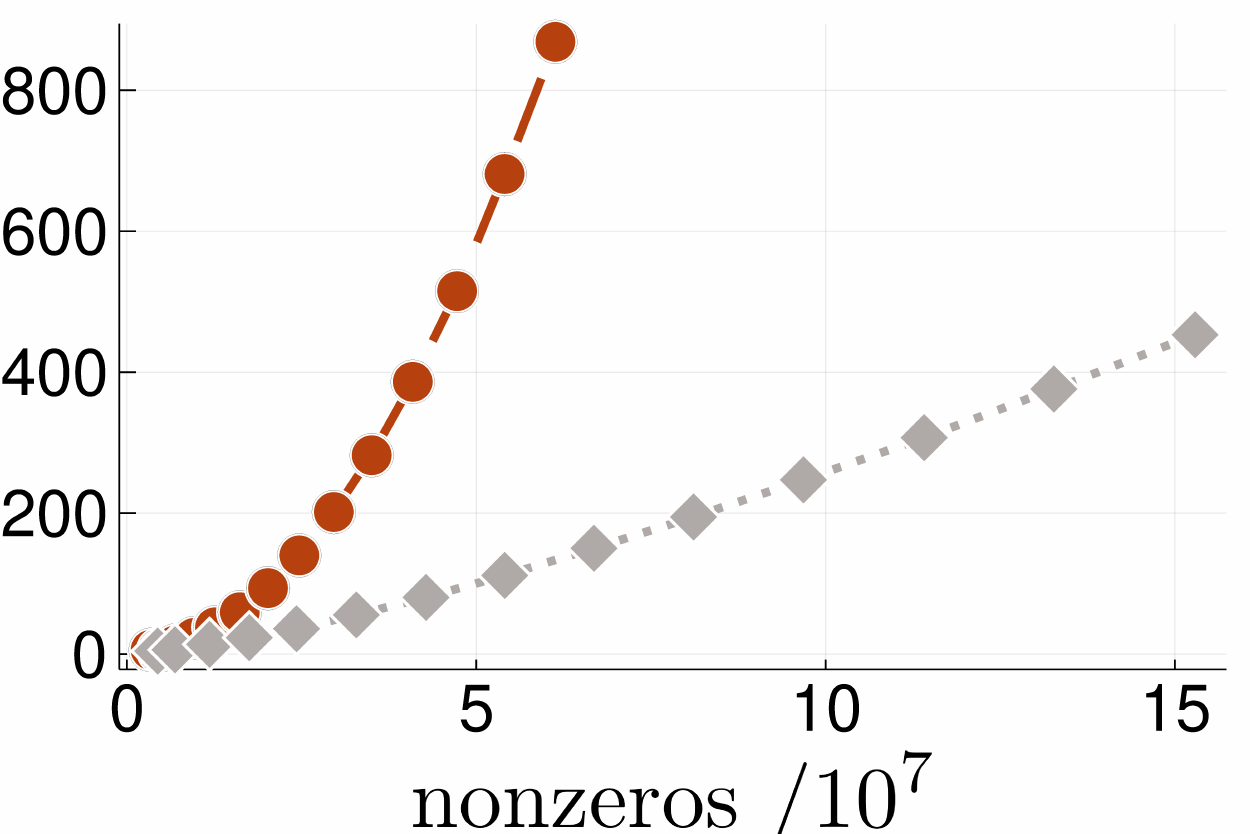}
	\caption{Time for computing the factor $L^{\rho}$ with or without aggregation ($N = 10^6$), as a function of $\rho$ and of the number of nonzero entries. For the first two panels, the Mat{\'e}rn covariance function was computed using a representation in terms of exponentials, while for the second two panels they were computed using (slower) Bessel function evaluations. Computations performed on an Intel\textsuperscript{\textregistered}Core\texttrademark i7-6400 CPU with 4.00GHz and 64\,GB of RAM. 
	The second and fourth panel show that aggregation leads to faster computation despite producing much denser Cholesky factors (and hence higher accuracy).}
    \label{fig:timings}
\end{figure}

We further provide timing results for $10^6$ training points uniformly distributed in $\left[0, 1\right]^2$ comparing the aggregated and non-aggregated version of the algorithm in \cref{fig:timings}.
As predicted by the theory, the aggregated variant scales better as we are increasing $\rho$.
This holds true both when using \texttt{Intel}\textsuperscript{\textregistered} \texttt{oneMKL Vector Mathematics functions} library to evaluate the exponential function, or when using \texttt{amos} to instead evaluate the modified Bessel function of the second kind.
While the former is faster and emphasizes the improvement from $\O(N \rho^{3d} )$ to $\O(N \rho^{2d} )$ for the complexity of computing the factorization, the latter can be used to evaluate Mat\'ern kernels with arbitrary smoothness.
Due to being slower, using Bessel functions highlights the improvement from needing $\O(N \rho^{2d} )$ matrix evaluations without the aggregation to just $\O(N \rho^{d} )$.
By plotting the number of nonzeros used for the two approaches, we see that the aggregated version is faster to compute despite using many more entries of $\KM$ than the non-aggregated version.
Thus, aggregation is both faster and more accurate for the same value of $\rho$, which is why we recommend using it over the non-aggregated variant.

\subsection{Adding noise\label{ssec:noisecomparison}}

\begin{figure}
    \centering
    \includegraphics[scale=0.20]{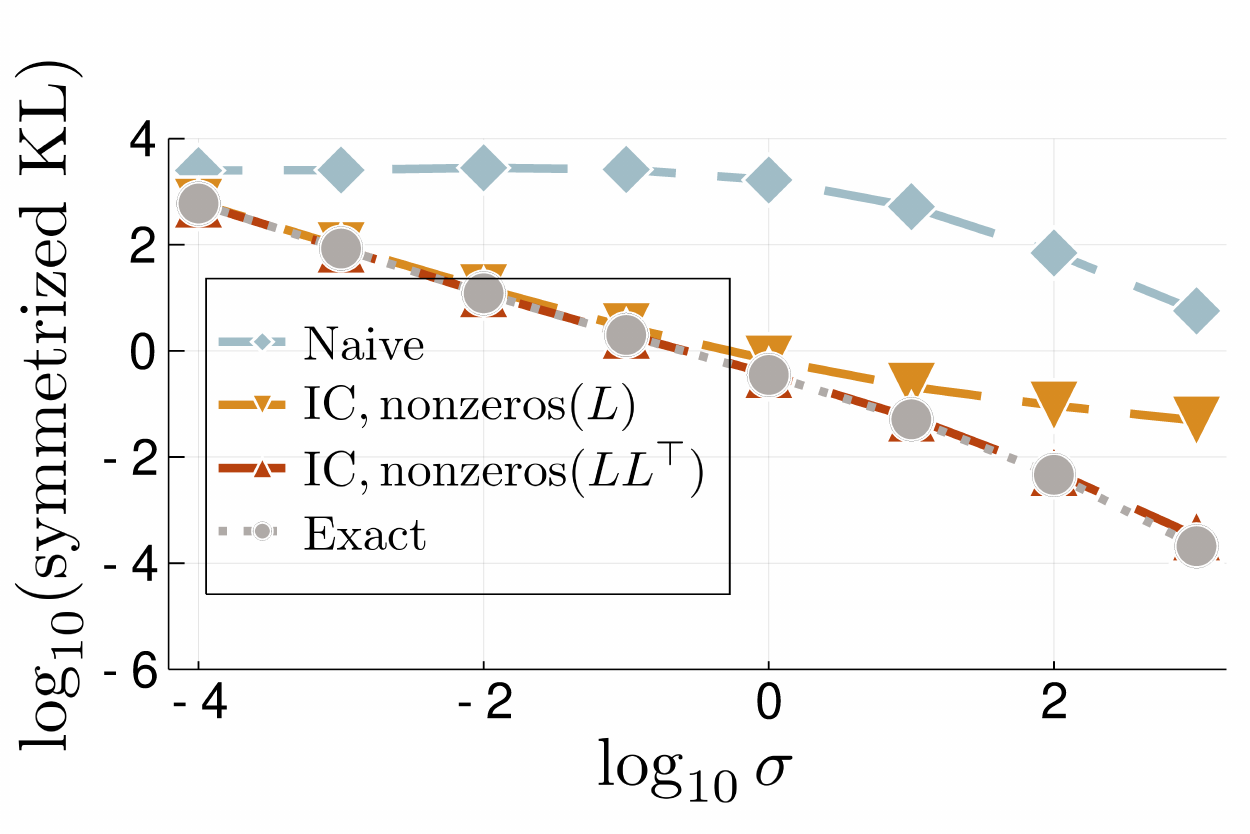}
    \includegraphics[scale=0.20]{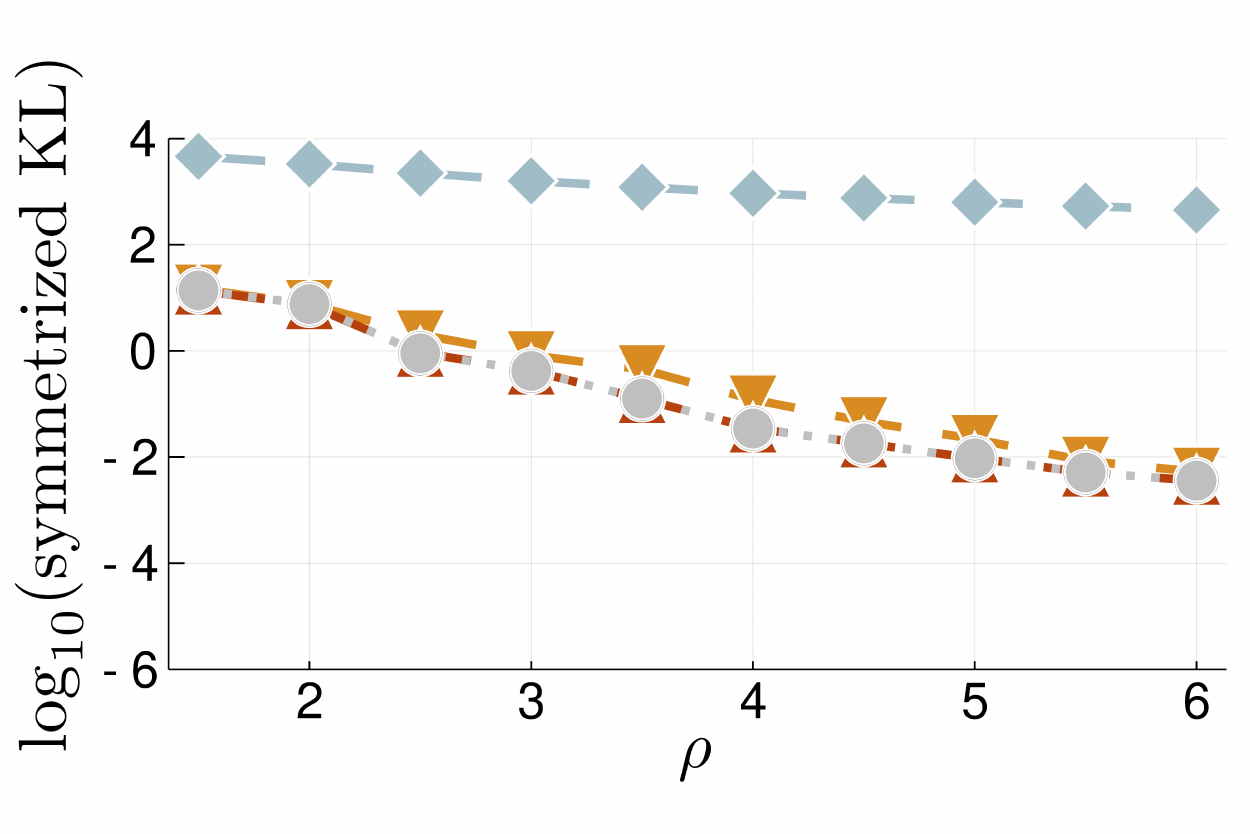}
    \includegraphics[scale=0.20]{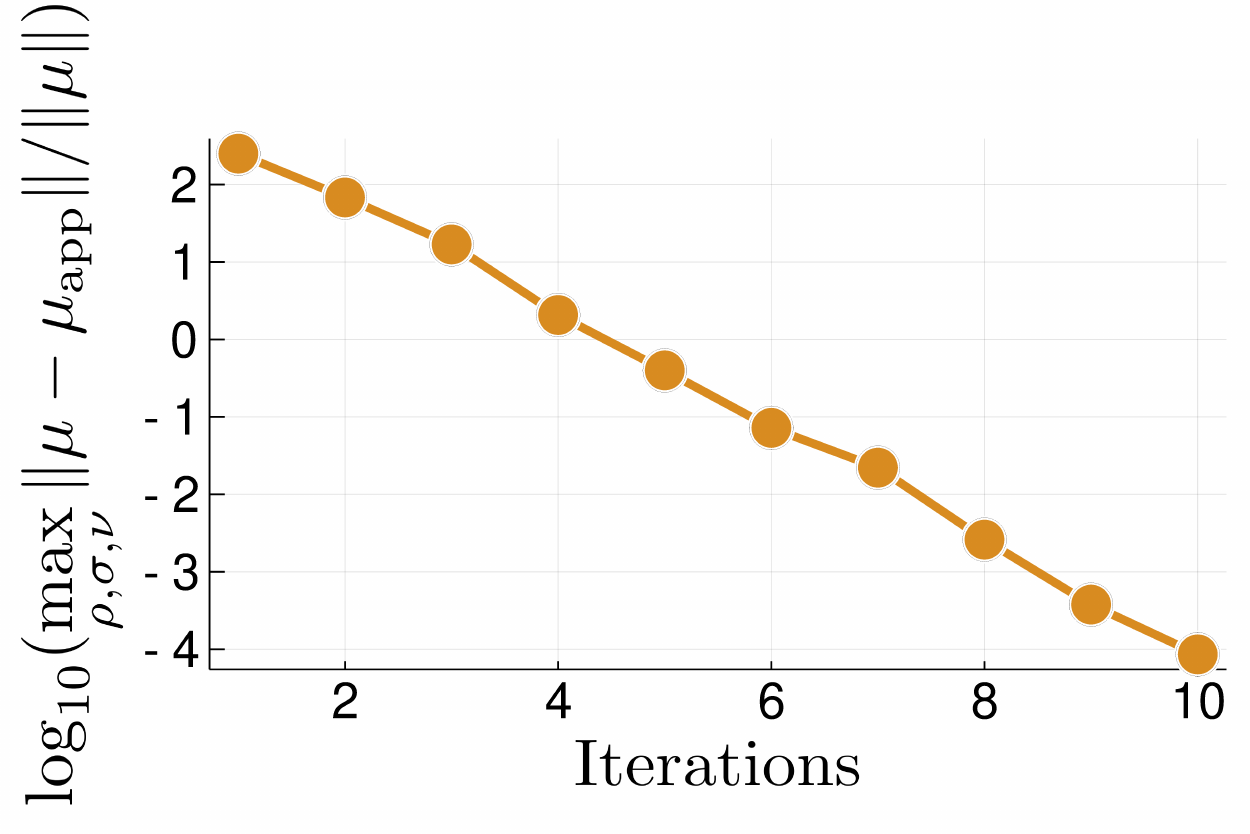}
    \caption{Comparison of the methods proposed in \cref{ssec:noise} for approximating $\Sigma = \KM + R$, where $\KM$ is based on a Mat\'ern covariance with range parameter $0.5$ and smoothness $\nu = 3/2$ at $N = 10^4$ uniformly sampled locations on the unit square, and $R=\sigma^2 I$ is additive noise. For each approximation, we compute the symmetrized KL divergence (the sum of the KL-divergences with either ordering of the two measures) to the true covariance. 
    ``Naive'': Directly apply \cref{alg:aggregated} to $\Sigma$.
    ``Exact'': Apply \cref{alg:aggregated} to $\KM$, then compute $\tilde{L}$ as the exact Cholesky factorization of $A \defeq R^{-1} + \hat{\KM}^{-1}$.
    ``IC'': Apply \cref{alg:aggregated} to $\KM$, then compute $\tilde{L}$ using incomplete Cholesky factorization of $A$ on the sparsity pattern of either $L$ or $LL^{\top}$. 
    (Left:) Varying $\sigma$, fixed $\rho = 3.0$. 
    (Middle:) Varying $\rho$, fixed $\sigma = 1.0$.
    (Right:) Maximal relative error (over the above $\sigma$, $\rho$, $\nu \in \{1/2, 3/2, 5/2\}$ and $10$ random draws) of inverting $A$ using up to $10$ conjugate-gradient iterations ($x$-axis), with IC, nonzeros(L) as preconditioner.}
    \label{fig:nugget}
\end{figure}

We now experimentally verify the claim that the methods described in \cref{ssec:noise} enable accurate approximation in the presence of independent noise, while preserving the sparsity, and thus computational complexity, of our method.
To this end, pick a set of $N = 10^4$ points uniformly at random in $\Omega = \left[0, 1\right]^2$, use a Mat{\' e}rn kernel with smoothness $\nu = 3/2$, and add I.I.D.\ noise with variance $\sigma^2$. We use an aggregation parameter $\lambda = 1.5$.
As shown in \cref{fig:nugget}, our approximation stays accurate over a wide range of values of both $\rho$ and $\sigma$, even for the most frugal version of our method.
The asymptotic complexity for both incomplete-Cholesky variants is $\O(N \rho^{2d})$, with the variant using the sparsity pattern of $LL^{\top}$ being roughly equivalent to doubling $\rho$.
Hence, to avoid additional memory overhead, we recommend using the sparsity pattern of $L$ as a default choice; the accuracy of the resulting log-determinant of $\Sigma$ should be sufficient for most settings, and the accuracy for solving systems of equations in $\Sigma$ can easily be increased by adding a few iterations of conjugate gradient.

\subsection{Including prediction points}

\begin{figure}
    \centering
    \includegraphics[width=0.25\textwidth]{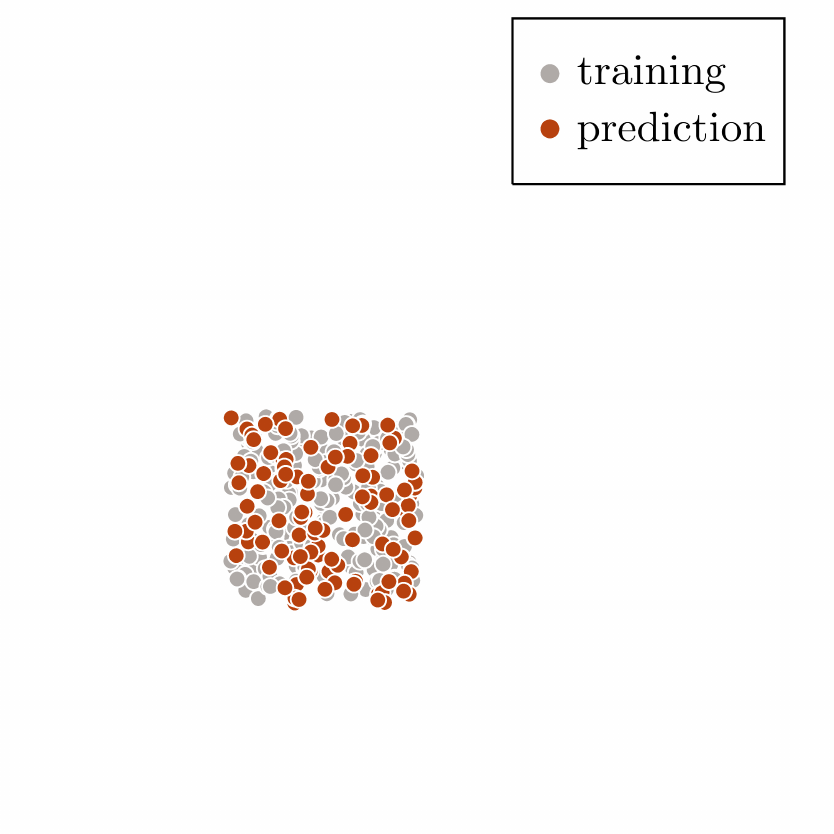}
    \includegraphics[width=0.35\textwidth]{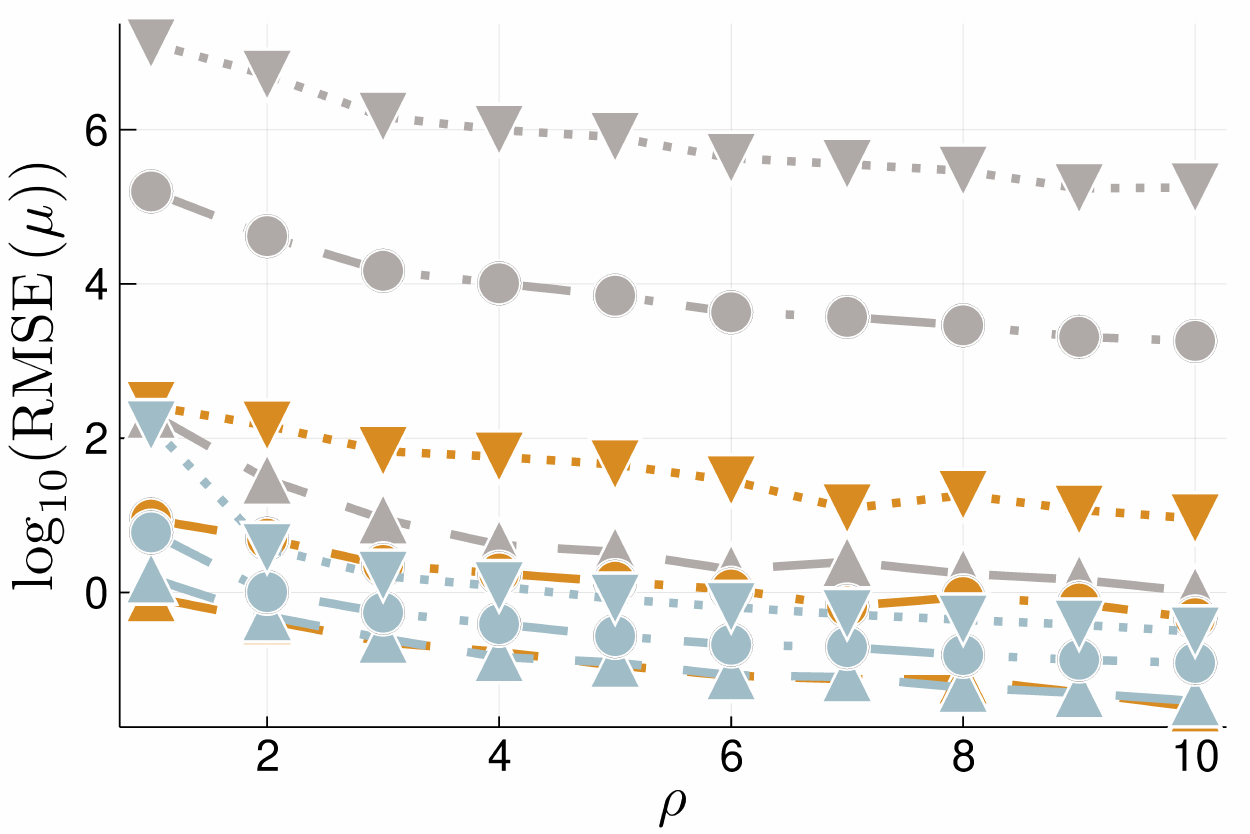}
		\includegraphics[width=0.35\textwidth]{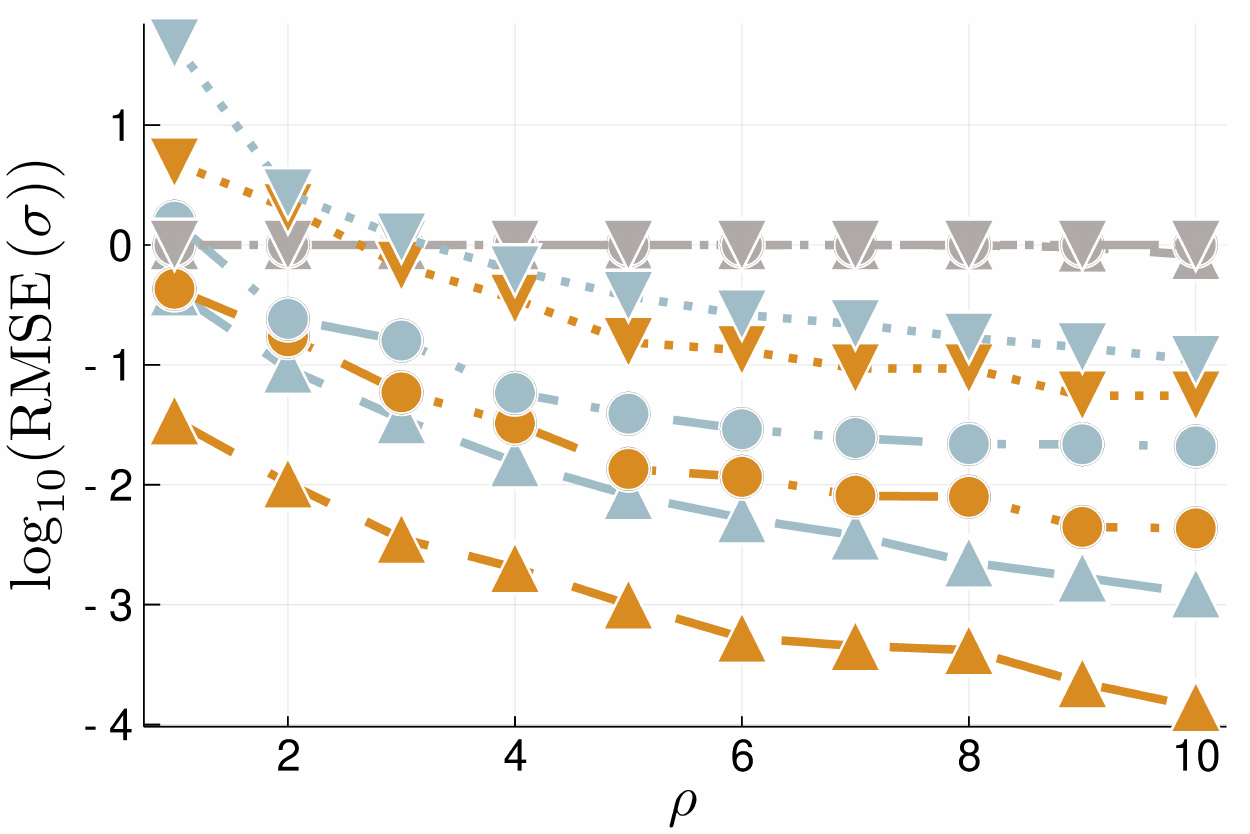}
    \includegraphics[width=0.25\textwidth]{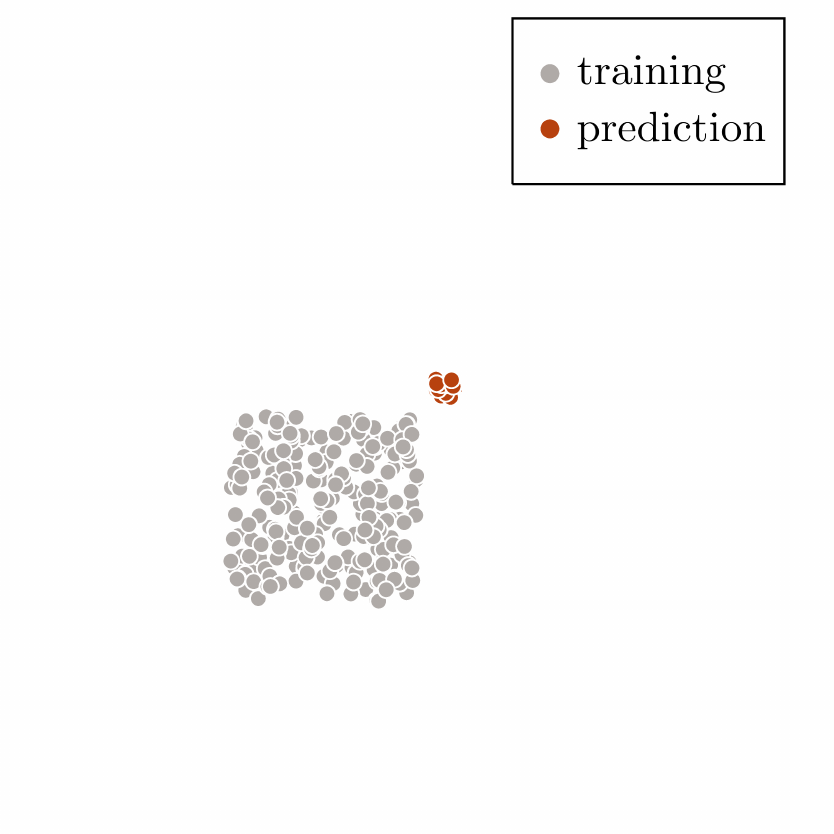}
    \includegraphics[width=0.35\textwidth]{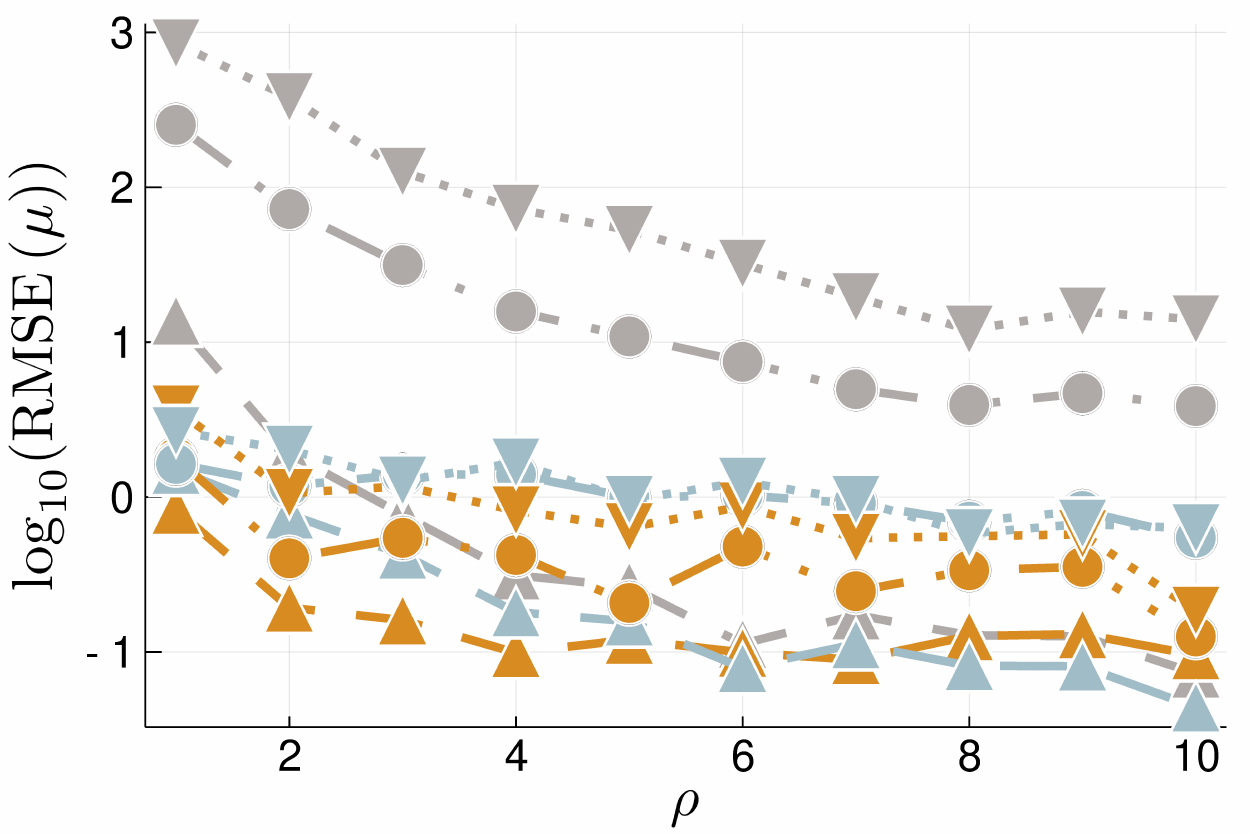}
		\includegraphics[width=0.35\textwidth]{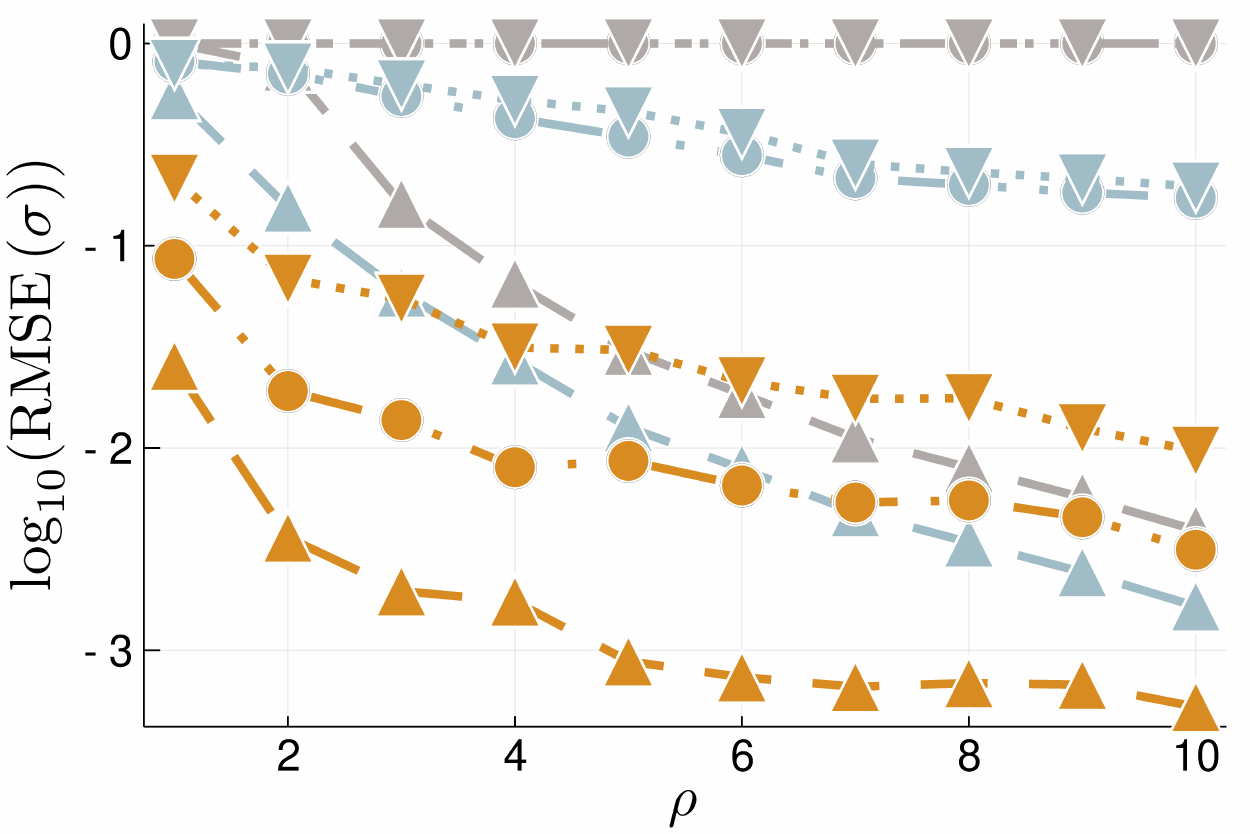}
    \includegraphics[width=0.25\textwidth]{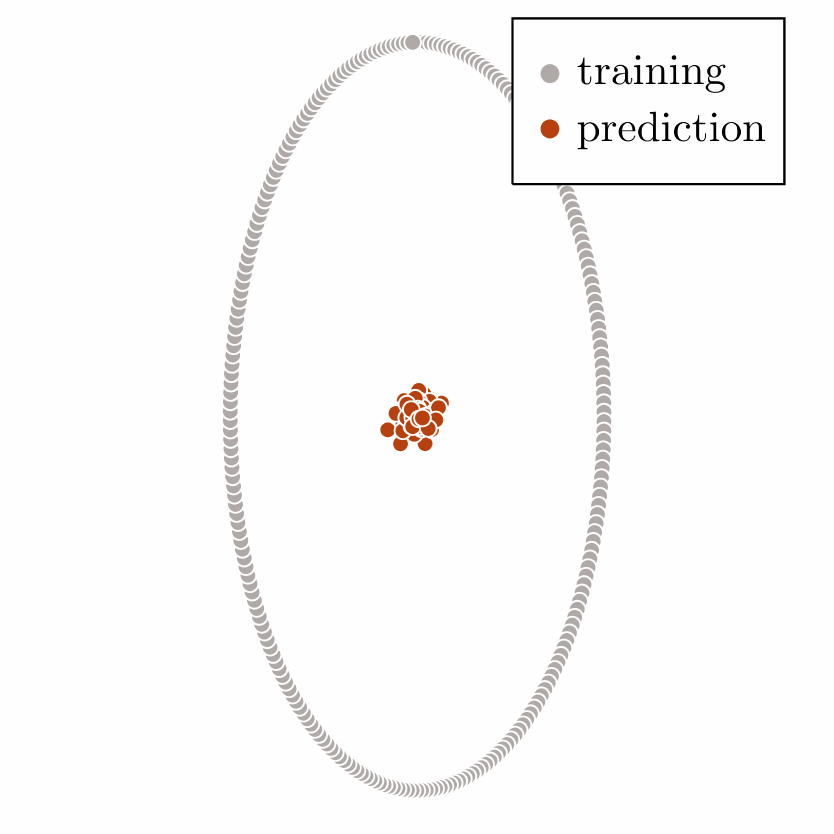}
    \includegraphics[width=0.35\textwidth]{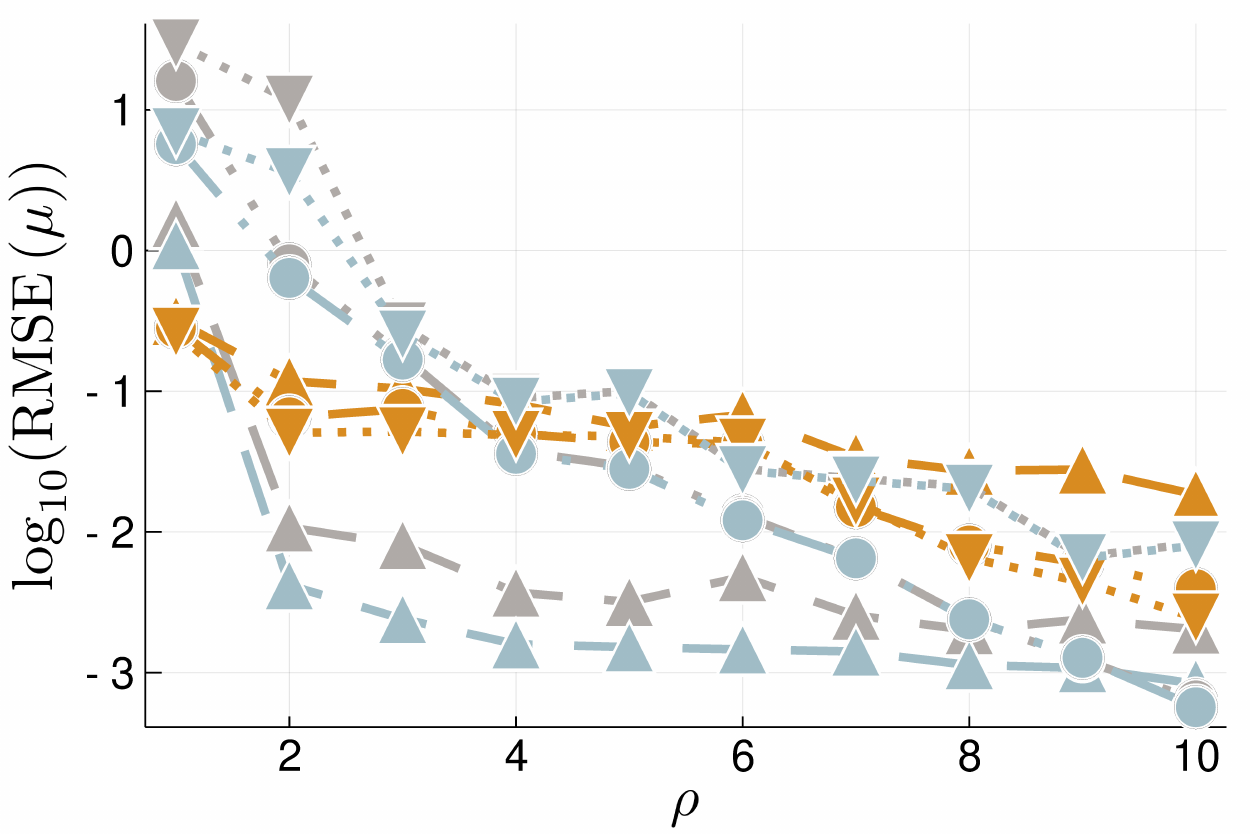}
    \includegraphics[width=0.35\textwidth]{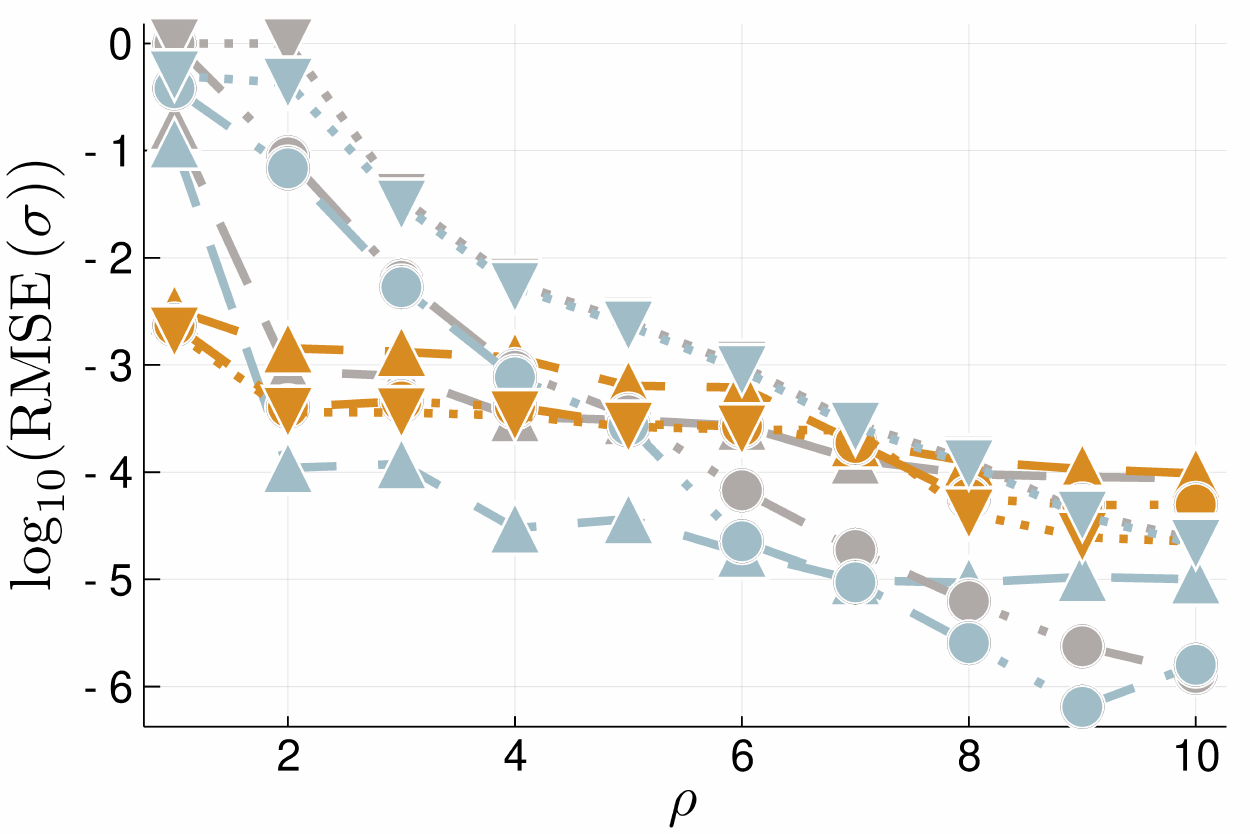}
    \includegraphics[width=0.32\textwidth]{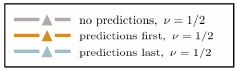}
    \includegraphics[width=0.32\textwidth]{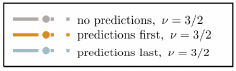}
    \includegraphics[width=0.32\textwidth]{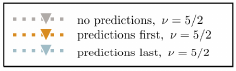}
    \label{fig:includePred}
		\caption{To analyze the effects of including the prediction points into the approximation, we consider three datasets. Each consists of $5 \times 10^4$ training points and $10^2$ test points, averaged over ten independent realizations of the Gaussian process. We use Mat{\'e}rn kernels with range parameter $0.5$ and smoothness $\nu \in \{1/2, 3/2, 5/2\}$, with $\rho$ ranging from $1.0$ to $10.0$. 
		We do not use aggregation since it might lead to slightly different sparsity patterns for the three variants, possibly polluting the results.
		On the $y$-axis we plot the RMSE of the  posterior mean and standard deviation, scaled in each point by the reciprocal of the true posterior standard deviation.
		In almost all cases, including the prediction points into the approximation improves the accuracy. 
		The comparison between ordering the predictions first or last is complicated, but ``predictions-last''  seems to perform better for lower smoothness and ``predictions-first'' for higher smoothness.}
\end{figure}

We continue by studying the effects of including the prediction points in the approximation, as described in \cref{sssec:predfirst,sssec:predlast}.
We compare not including the predictions points in the approximation with including them either before or after training points in the approximation.
We compare the accuracy of the approximation of the posterior mean and standard deviation over three different geometries and a range of different values for $\rho$.
The results, displayed in \cref{fig:includePred}, show that including the prediction points can increase the accuracy by multiple orders of magnitude.
The performance difference between the two schemes for including prediction points varies over different geometries, degrees of regularity, and values of $\rho$.
If the number of prediction points is comparable to the number of training points, the only way to avoid quadratic scaling in the number of points is to order the prediction points first, making this approach the method of choice.
If we only have few prediction points, ordering the prediction variables last can improve the accuracy for low orders of smoothness, especially in settings in which only a small part of the training data is used in the prediction-variables-first approach (e.g., second row in \cref{fig:includePred}).

\subsection{Comparison to HSS matrices\label{ssec:HSSComparison}}

As described in the introduction, there are many existing methods for the approximation and inversion of dense covariance matrices.
Hierarchically semiseparable (HSS) matrices \cite{chandrasekaran2004fast,xia2010fast} are a natural candidate for comparison with our method, because they are amenable to a Cholesky factorization \cite{li2012new}, implementations of which are available in existing software packages.
They are also closely related to hierarchically off-diagonal low-rank (HODLR) matrices, which have been promoted as tools for Gaussian process regression \cite{ambikasaran2014fast}.
We consider a regression problem with $50^3$ training points on a randomly perturbed regular grid and $50$ test points distributed uniformly at random in the unit cube. 
Using the Mat{\'e}rn covariance with $\nu = 3/2$ and length scale $l = 0.2$, we compute the posterior mean and standard deviation for $50$ samples using the method described in \cref{sssec:predlast} and the HSS implementation of \texttt{H2Pack} \cite{huang2020h2pack}, both using eight threads on an Intel\textsuperscript{\textregistered} Skylake \texttrademark CPU with 2.10GHz and 192\,GB of RAM. 
In \cref{fig:HSS-comparison}, we report the computational time and accuracy for a wide range of tuning parameters ($\rho$ for our method, error tolerance and diagonal shift for HSS).
We ignore the setup cost for both methods, which includes the selection of the ``\emph{numerical proxy points}'' for the HSS approach.
Our experiments show that for a given target accuracy, our method is an order of magnitude faster than HSS, despite the highly optimized implementation of the latter.
For very high accuracies, the gap between the methods closes but the memory cost of HSS approaches that of the dense problem, preventing us from further increasing the target accuracy.
We note that for three-dimensional problems, $\mathcal{H}^2$-matrices have better asymptotic complexity than HSS matrices, making them a possibly stronger competitor; however, the Cholesky factorization of $\mathcal{H}^2$-matrices is considerably more complicated and not implemented in \texttt{H2Pack}.
Another possible approach is the inversion of an $\mathcal{H}^2$ approximation using conjugate gradient methods, using our method or HSS matrices (\cite{xing2020efficient}) as a preconditioner.
We defer a more comprehensive comparison to the various kinds of hierarchical matrices to future work.

\begin{figure}
    \centering
    \includegraphics[width=0.45 \columnwidth]{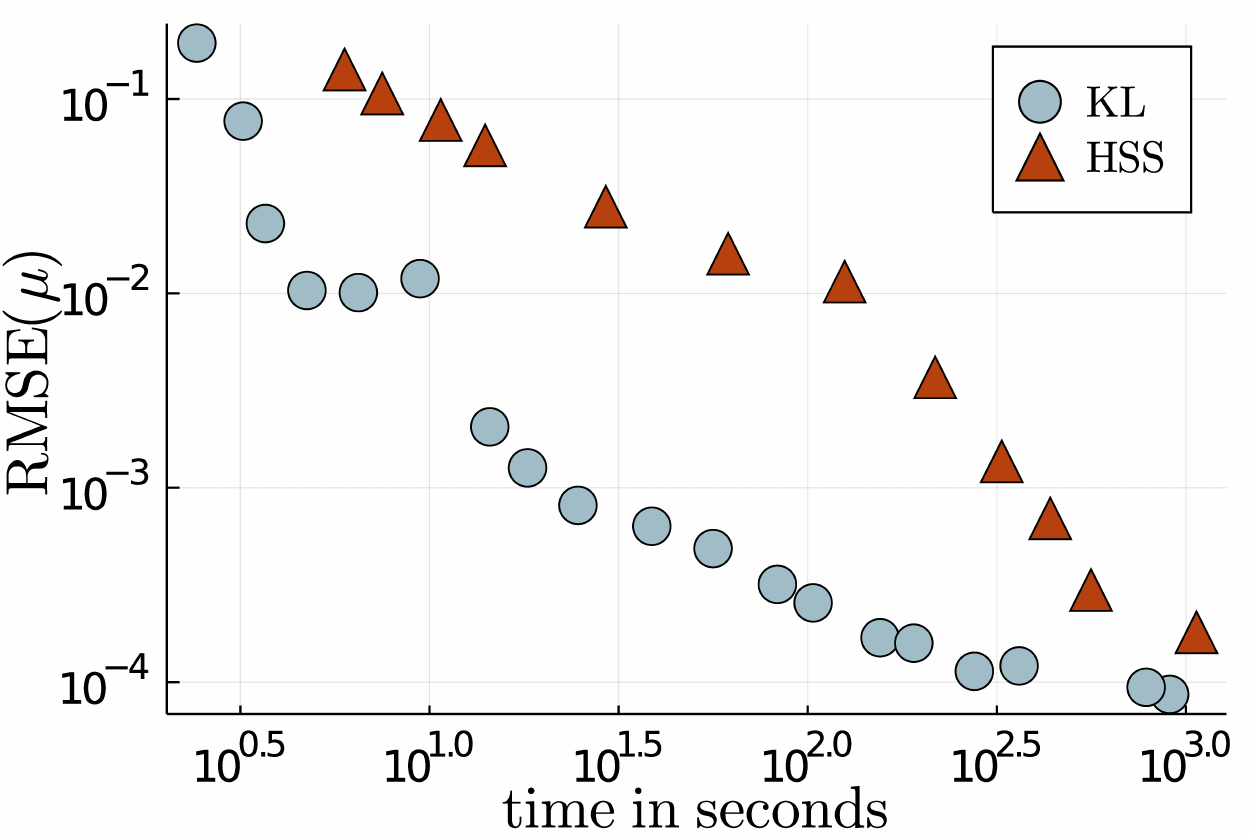}
    \includegraphics[width=0.45 \columnwidth]{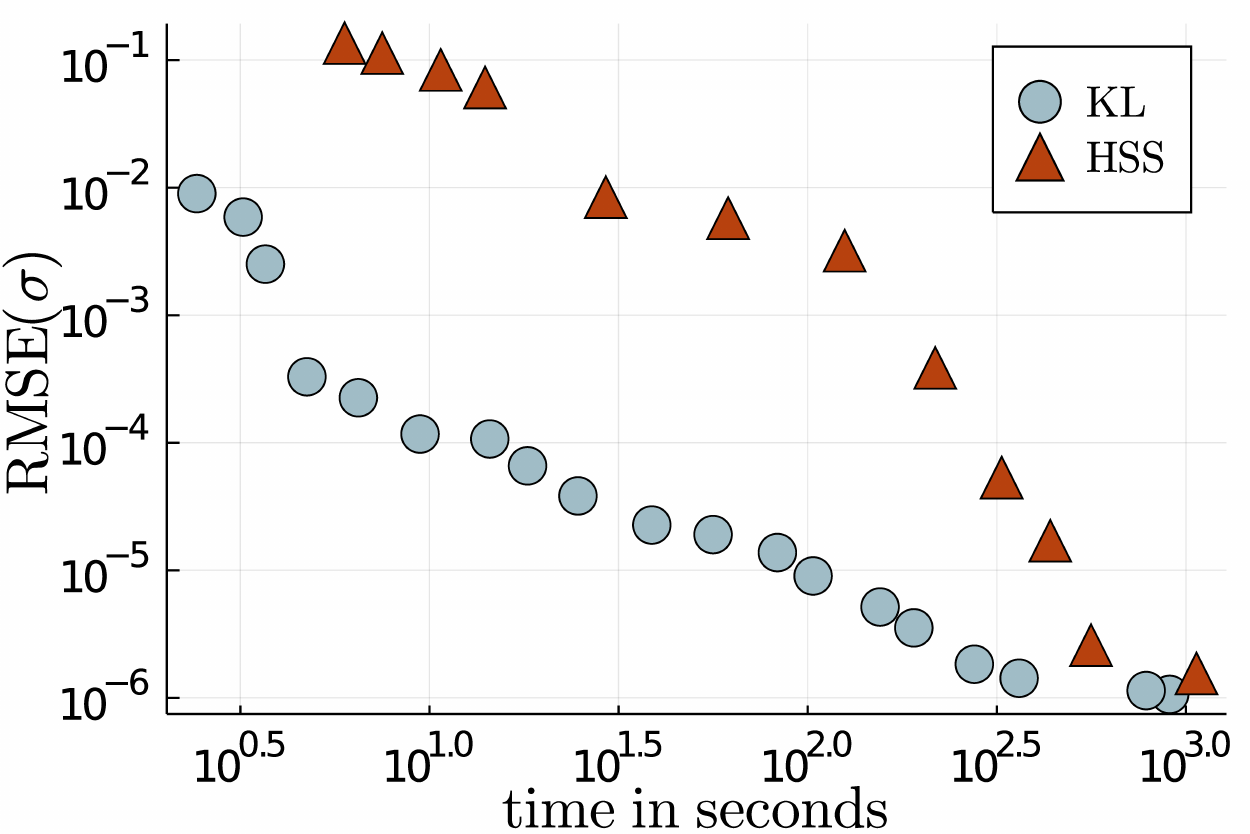}
    \label{fig:HSS-comparison}
	\caption{We compare the accuracy and computational time of our method described in \cref{sssec:predlast} with the HSS implementation of \texttt{H2Pack} \cite{huang2020h2pack}.
	Each point corresponds to a different run with different parameters ($\rho$, tolerance, and diagonal shift).
	Throughout this experiment, we use the aggregation scheme described in \cref{ssec:aggregated} with $\lambda = 1.25$.
	The left plot shows the $\mathrm{RMSE}$ of the posterior mean and the right plot that of the posterior standard deviation.
	Our method is significantly faster for a wide range of target accuracies.}
\end{figure}

\subsection{Single-layer boundary element methods}
\label{ssec:BEM}

We now provide an application to boundary element methods.
For a domain $\Omega \in \Reals^d$ with boundary $\partial \Omega$, let us assume that we want to solve the Dirichlet boundary value problem
\begin{alignat*}{2}
    - \Delta u(x) &= 0, \quad & \forall  x&\in \Omega \\
     u(x) &= g(x),      \quad & \forall  x&\in \partial \Omega.
\end{alignat*}
%
For $d = 3$, the Green's function of the Laplace operator is given by the gravitational / electrostatic potential 
\begin{equation*}
    \K_{\Reals^3}(x,y) = \frac{1}{4 \pi |x - y|}.
\end{equation*}
Under mild regularity assumptions one can verify that 
\begin{equation*}
    u = \int \limits_{x \in \partial \Omega} \K_{\Reals^3}(x, \cdot) h(x) \dx, \quad 
\text{ for $h$ the solution of } \quad
    g = \int \limits_{x \in \partial \Omega} \K_{\Reals^3}(x, \cdot) h(x) \dx.
\end{equation*}

Let us choose finite dimensional basis functions $\left\{\phi_{i}\right\}_{i \in \I_{\Pred}}$  in the interior of $\Omega$ and $\left\{\phi_{i}\right\}_{i \in \I_{\Train}}$ on the boundary of $\Omega$.
We form the matrix $\KM \in \Reals^{\left(\I_{\Train} \cup \I_{\Pred}\right) \times \left(\I_{\Train} \cup \I_{\Pred}\right) }$ as 
\begin{equation}
    \KM_{ij} \defeq \int_{x \in \mathcal{D}_i} \int_{y \in \mathcal{D}_j} \phi_{i}(x) \K_{\Reals^3}\left(x, y\right) \phi_{j}(y) \dy \dx,
    \quad \mathrm{where} \quad 
    \mathcal{D}_p =
    \begin{cases}
    \partial \Omega, \; &\mathrm{for} \; p \in \I_{\Train} \\
    \Omega, \; &\mathrm{for} \; p \in \I_{\Pred} 
    \end{cases}
\end{equation}
and denote as $\KM_{\Train, \Train}, \KM_{\Train, \Pred}, \KM_{\Pred, \Train}, \KM_{\Pred, \Pred}$ its restrictions to the rows and columns indexed by $\I_{\Train}$ or $\I_{\Pred}$.
Defining
\begin{equation*}
    \vec{g}_i \defeq \int \limits_{x \in \partial \Omega} \phi_i(x) g(x) \dx, \quad \forall i \in \I_{\Train} \quad \mathrm{and} \quad
    \vec{u}_i \defeq \int \limits_{x \in \partial \Omega} \phi_i(x) u(x) \dx, \quad \forall i \in \I_{\Pred},
\end{equation*}
we approximate $\vec{u}$ as 
\begin{align}
    \label{eqn:bem}
    \vec{u} &\approx \KM_{\I_{\Pred}, \I_{\Train}} \KM_{\I_{\Train}, \I_{\Train}}^{-1}\vec{g}.
\end{align}
This is a classical technique for the solution of partial differential equations, known as single layer boundary element methods \cite{sauter2010boundary}.
However, it can also be seen as Gaussian process regression with $u$ being the conditional mean of a Gaussian process with covariance function $\K$, conditional on the values of the process on $\partial \Omega$. 
Similarly, it can be shown that the zero boundary value Green's function is given by the posterior covariance of the same process.

The Laplace operator in three dimensions does not satisfy $s > d/2$ (cf. \cref{sssec:appproxerr}). 
Therefore, the variance of pointwise evaluations at $x \in \Reals^3$ given by $\K_{\Reals^3}(x,x)$ is infinite and we cannot let $\left\{\phi_i\right\}_{i \in \I_{\Pred}}$ be Dirac-functions as in other parts of this work.

Instead, we recursively subdivide the boundary $\partial \Omega$ and use  Haar-type wavelets as \cite[Ex.~3.2]{Schafer2017} for $\left\{\phi_{i}\right\}_{i \in \I_{\Train}}$. 
For our numerical experiments, we will consider $\Omega \defeq [0,1]^3$ to be the three-dimensional unit cube.
On each face of $\partial \Omega$, we then obtain a multiresolution basis by hierarchical subdivision, as shown in \cref{fig:basis-functions}.
In this case, the equivalent of a maximin ordering is an ordering from coarser to finer levels, with an arbitrary ordering within each level.
We construct our sparsity pattern as 
\begin{equation}
    \mathcal{S}_{\prec, \ell_j,  \rho} \defeq \{\, (i, j) : i \succeq j, \dist(x_i, x_j) \leq \rho \ell_j + \sqrt{2}(\ell_i + \ell_j) \,\},
\end{equation}
where for $i \in \I_{\Train}$, $x_i$ is defined as the center of the support of $\phi_i$ and $\ell_i$ as half of the side-length of the (quadratic) support of $\phi_i$. 
The addition of $\sqrt{2}(\ell_i + \ell_j)$ to the right-hand side ensures that the entries corresponding to neighboring basis functions are always added to the sparsity pattern.

\begin{figure}
    \centering
    \includegraphics[width=\columnwidth]{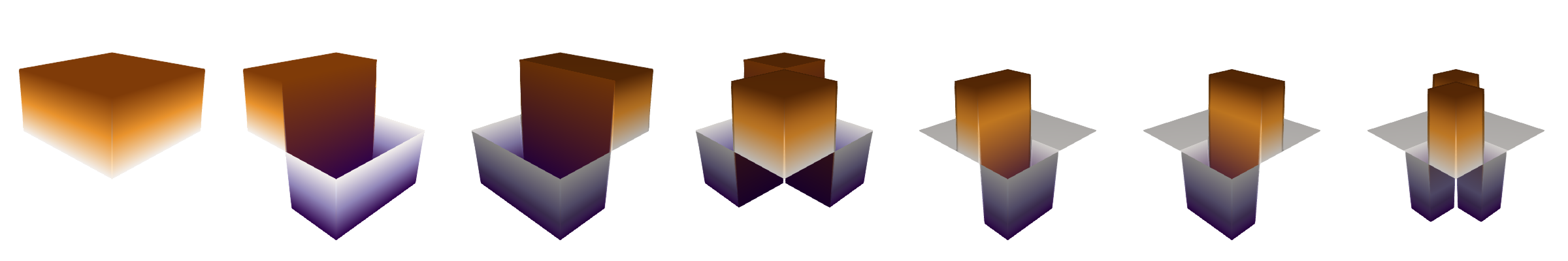}
    \label{fig:basis-functions}
    \caption{We recursively divide each panel of $\partial \Omega$. The basis functions on finer levels are constructed as linear combinations of indicator functions that are orthogonal to functions on coarser levels.}
\end{figure}

We construct a solution $u$ of the Laplace equation in $\Omega$ as the sum over $N_c = 2000$ charges with random signs $\left\{s_i\right\}_{1 \leq i \leq N_c}$ located at points $\left\{c_i\right\}_{1 \leq i \leq N_c}$
We then pick a set of $N_{\Pred}$ points $\left\{x_{i}\right\}_{i \in \I_{\Pred}}$ inside of $\Omega$ and try to predict the values $\left\{u(x_i)\right\}_{i \in \I_{\Pred}}$ using Equation~\eqref{eqn:bem} and the method described in \cref{sssec:predlast}.
We compare the computational time, the number of entries in the sparsity pattern, and the mean accuracy of the approximate method for $\rho \in \left\{1.0, 2.0, 3.0\right\}$, as well as the exact solution of the linear system.
We use different levels of discretization $q \in \{3, \ldots, 8\}$, leading to a spatial resolution of up to $2^{-8}$.
As shown in \cref{fig:BEM_plots}, even using $\rho=1.0$ leads to near-optimal accuracy, at a greatly reduced computational cost.
\begin{figure}
	\centering 
	\includegraphics[scale=0.30]{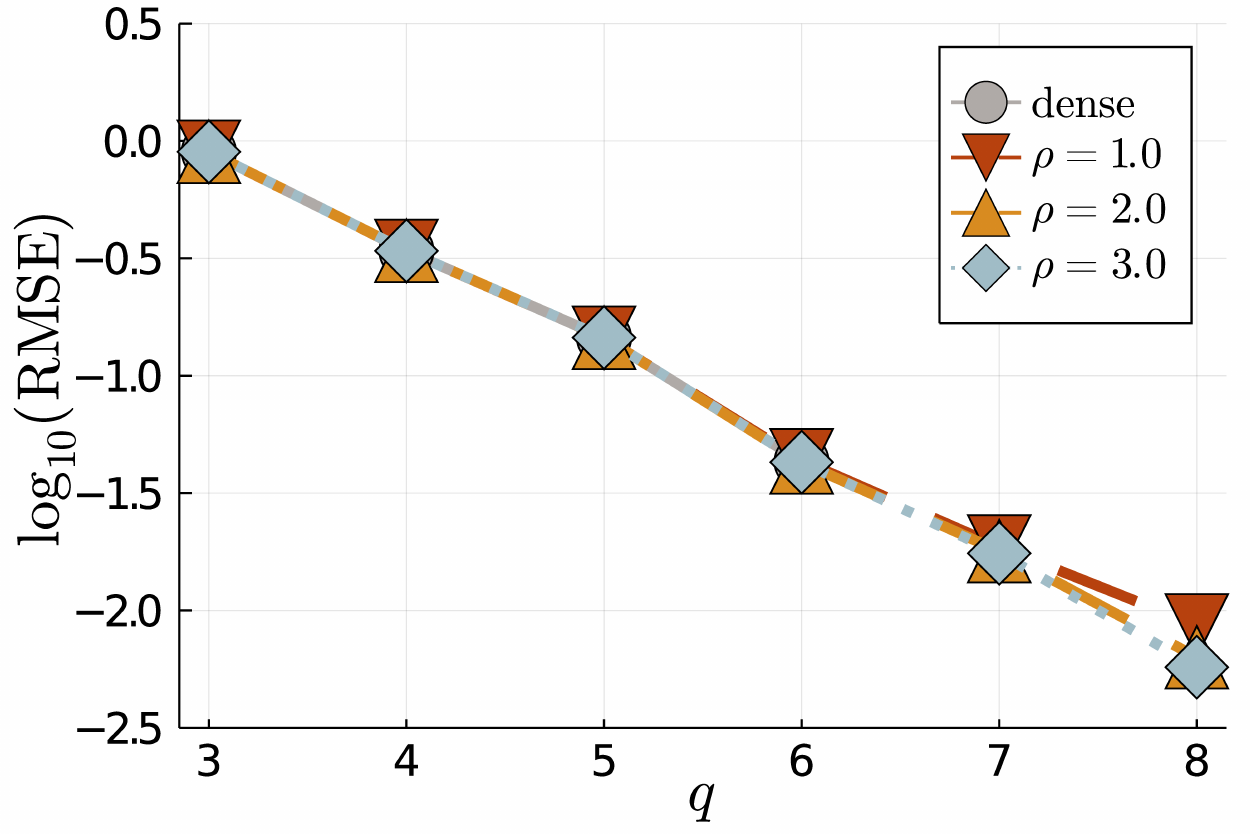}
	\includegraphics[scale=0.30]{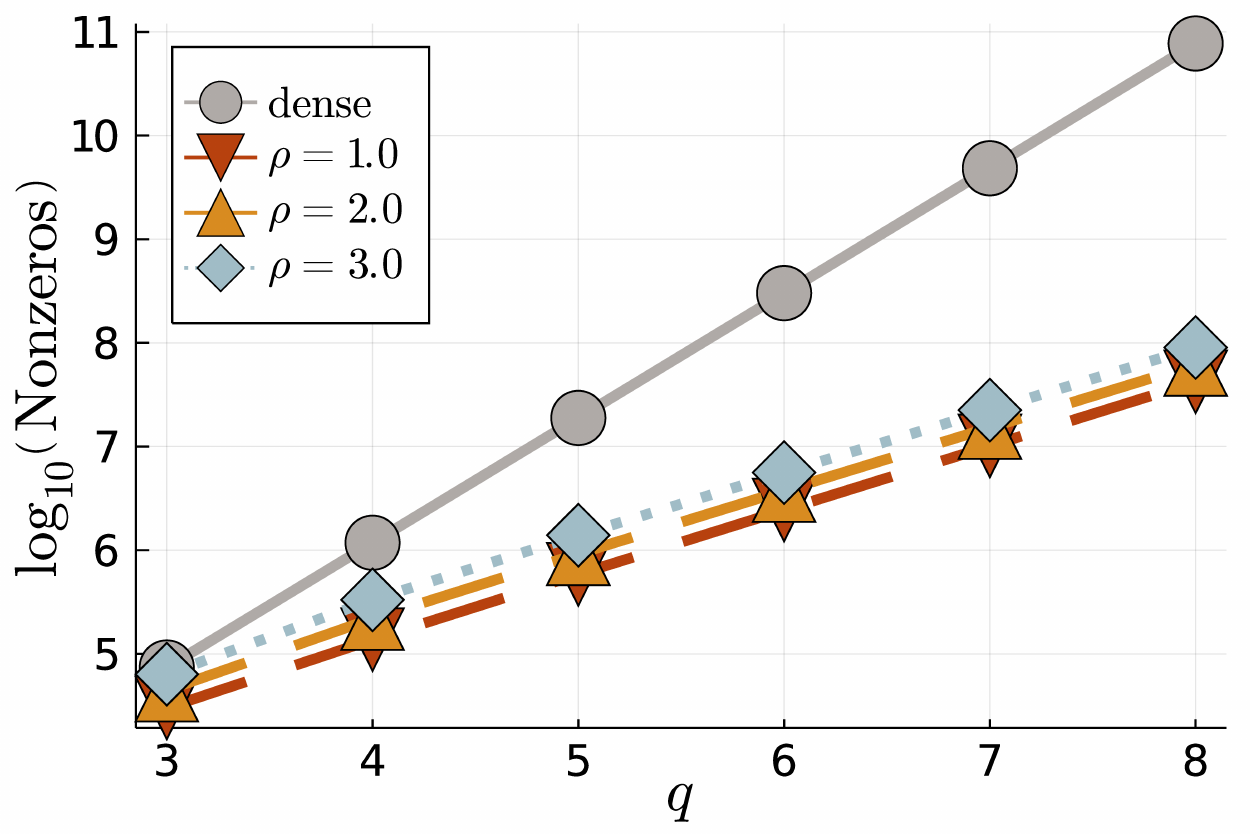}
	\includegraphics[scale=0.30]{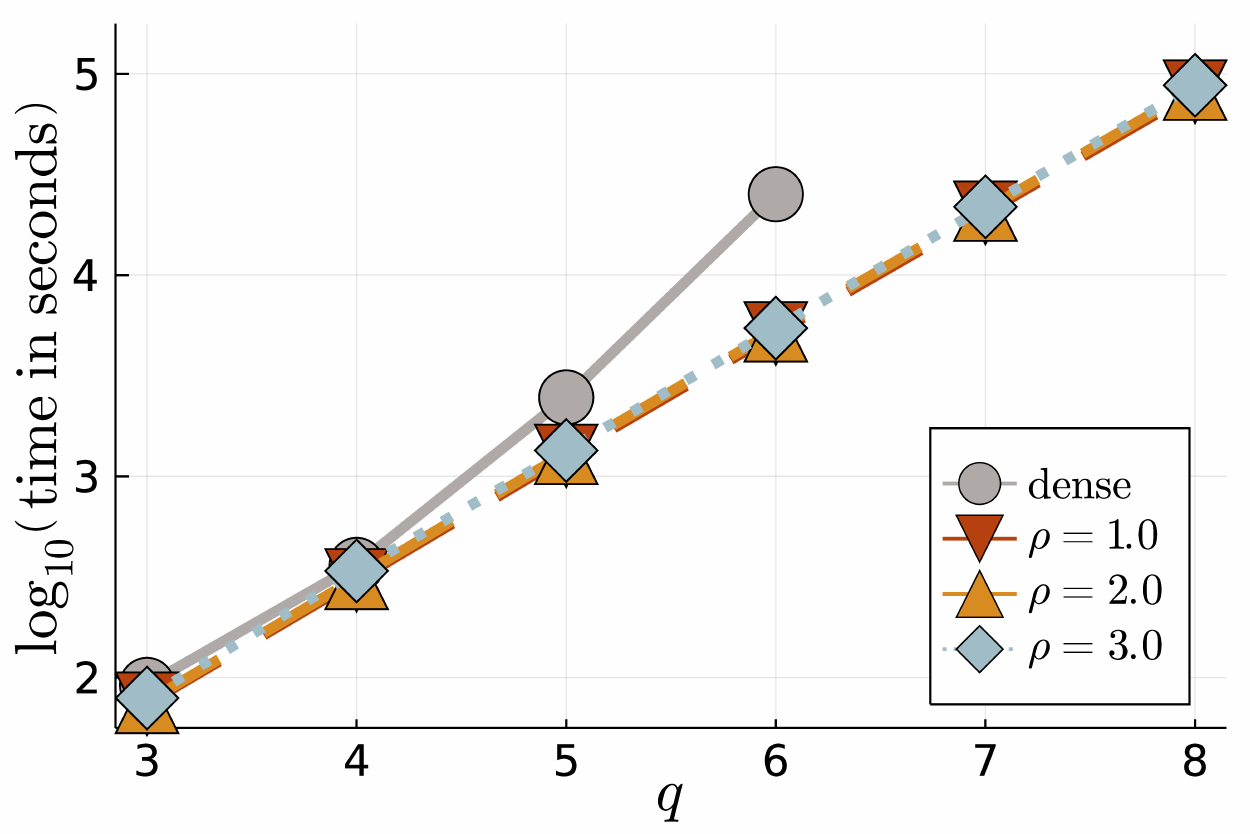}
	\includegraphics[scale=0.30]{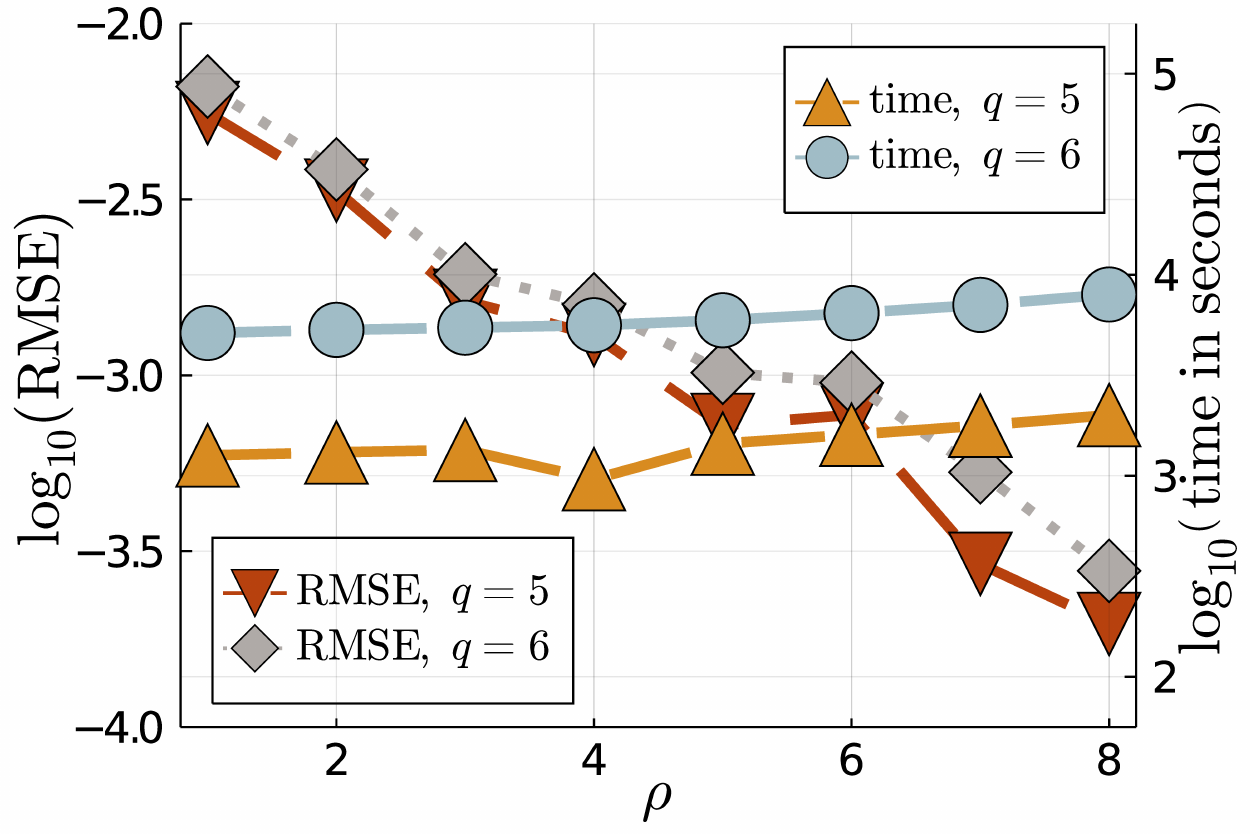}
	\label{fig:BEM_plots}
	\caption{Accuracy and computational complexity in boundary value problem. We compare the root mean square error, number of nonzeros of sparsity pattern, and the computational time for the exact boundary element method and using our approximation for $\rho \in \left\{1, 2, 3\right\}$.
	The dense solution is prohibitively expensive for $q > 6$, which is why accuracy and computational time for these cases are missing. The reason that the computational time is hardly affected by different choices of $\rho$ is due to the fact that entries $\left(\KM_{\Train, \Train}\right)_{ij}$ for nearby $\phi_i, \phi_j$ are significantly more expensive to compute than for distant ones when using an off-the-shelf adaptive quadrature rule. The computations were performed on 32 threads of an Intel\textsuperscript{\textregistered} Skylake \texttrademark CPU with 2.10GHz and 192\,GB of RAM.
	In the first figure, we plot the RMSE compared to the true solution of the PDE as a function of $q \approx \log(N)$. In the last figure, we compute the RMSE between \emph{dense} computation and our method, as well as its computational time, as a function of $\rho$. }
\end{figure}

There exists a rich literature on the numerical solution of boundary element equations \cite{sauter2010boundary}, and we are not yet claiming improvement over the state of the art.
Presently, the majority of the computational time is spent computing the matrix entries of $\KM_{\Train, \Train}$.
In order to compete with the state of the art in terms of wall-clock times, we would need to implement more efficient quadrature rules, which is beyond the scope of this paper. 
Due to the embarrassing parallelism of our method, together with the high accuracy obtained even for small values of $\rho$, we hope that it will become a useful tool for solving boundary integral equations, but we defer a detailed study to future work.

\section{Conclusions}
\label{sec:conclusions}

In this work, we have shown that, surprisingly, the optimal (in KL-divergence) inverse Cholesky factor of a positive definite matrix, subject to a sparsity pattern, can be computed in closed form.
In the special case of Green's matrices of elliptic boundary value problems in $d$ dimensions, we show that by applying this method to the elimination orderings and sparsity patterns proposed by \cite{Schafer2017}, one can compute the sparse inverse Cholesky factor with accuracy $\epsilon$ in computational complexity $\O(N \log^{2d}(N/\epsilon))$ using only $\O(N \log^{d}(N/\epsilon))$ entries of the dense Green's matrix. 
This improves upon the state of the art in this classical problem.
We also propose a variety of improvements, capitalizing on the improved stability, parallelism, and memory footprint of our method.
Finally, we show how to extend our approximation to the setting with additive noise, resolving a major open problem in spatial statistics.

\section*{Acknowledgements}
We thank the two anonymous referees for their constructive feedback that helped us to improve the article.
FS and HO gratefully acknowledge support by the Air Force Office of Scientific Research under award number FA9550-18-1-0271 (Games for Computation and Learning), and the Office of Naval Research under award N00014-18-1-2363 (Toward scalable universal solvers for linear systems). MK's research was partially supported by National Science Foundation (NSF) grants DMS--1654083, DMS--1953005, and CCF--1934904. The computations in \cref{ssec:BEM} were conducted on the Caltech High-Performance Cluster partially supported by a grant from the Gordon and Betty Moore Foundation.

\addcontentsline{toc}{section}{Acknowledgements}


\bibliographystyle{siamplain}
\bibliography{references}

\newpage
\appendix

\section{Computation of the KL-minimizer}

\subsection{Computation without aggregation}
\label{apssec:notAggregatedPattern}

Recall that we write $\I$ for the set indexing the degrees of freedom,  $\prec$ for a $r$-reverse-maximin ordering, and  $S = S_{\prec,\ell,\rho}$ for the associated sparsity pattern (which we assume to be fixed).  
Unless explicitly mentioned, we assume all matrices have rows and columns ordered according to $\prec$.
For $k \in \I$, we then write $s_{k} \defeq \left\{ \left(i,k\right): k \preceq i, \left(i,k\right) \in S\right\}$ for the sparsity set of the $k$-th column $L_{:,k}$ of $L$.
As before, $\mathbf{e}_k$ is the vector that is $1$ on the $k$-th coordinate and zero everywhere else. 

\begin{proof}[Proof of \cref{thm:repL}]
    By using the formula for the KL-divergence of two Gaussian random variables in \eqref{eqn:KLnormal}, we obtain
    \begin{align}
        L &=
        \argmin_{\hat{L} \in \SpSet} \left(
        \trace( \hat{L} \hat{L}^{\top} \KM ) - \logdet(\hat{L} \hat{L}^\top) - \logdet( \KM ) - N  \right) \\
        &= \argmin_{\hat{L} \in \SpSet} 
         \left( \trace(  \hat{L}^{\top} \KM \hat{L} ) 
        -\logdet(\hat{L} \hat{L}^\top) \right)\\
        &= \argmin_{\hat{L} \in \SpSet} 
        \sum \limits_{k=1}^N \left( \hat{L}_{s_k,k}^{\top} \KM_{s_k,s_k} \hat{L}_{s_k,k} 
        - 2 \log(\hat{L}_{k,k})\right).
    \end{align}   
    The $k$-th summand depends only on the $k$-th column of $\hat{L}$.
    Thus, taking the derivative with respect to the $k$-the column of $L$ and setting
    it to zero, we obtain 
    $\KM_{s_k,s_k} \hat{L}_{s_k,k} = \frac{ \mathbf{e}_1 }{\hat{L}_{k,k} } \Leftrightarrow 
    \hat{L}_{s_k,k}  = \frac{\KM_{s_k,s_k}^{-1} \mathbf{e}_1}{\hat{L}_{k,k}}$.
    Therefore, $\hat{L}_{s_k,k}$ can be written as $\lambda \KM^{-1}_{s_k,s_k} \mathbf{e}_1$ 
    for a $\lambda \in \Reals$.
    By plugging this ansatz into the equation, we obtain 
    $\lambda = \sqrt{\left( \KM_{s_k,s_k}^{-1} \mathbf{e}_1\right)_{1}} 
    = \sqrt{\mathbf{e}_1^{\top} \KM_{s_k,s_k}^{-1} \mathbf{e}_1}$ and hence 
    Equation~\eqref{eqn:defcolL}.
    By using dense Cholesky factorization to invert the $\KM_{s_k,s_k}$, the right-hand 
    side of Equation~\eqref{eqn:defcolL} can be computed in computational complexity 
    $\BigO\left( \#\left(s_k \right)^2\right)$ in space and 
    $\BigO\left( \#\left(s_k \right)^3\right)$ in time, from which follows the result.
\end{proof}
\cref{alg:notAggregated} is a direct implementation of the above formula.

\subsection{Computation for the aggregated sparsity pattern}
\label{apssec:aggregatedPattern}
We first introduce some additional notation, defined in terms of an $r$-maximin ordering $\prec$ (see \cref{apsec:proofs}) and aggregated sparsity set $S = \tilde{S}_{\prec,\ell,\rho,\lambda}$, which we assume to be fixed. 
As before, $\I$ is the index set keeping track over the degrees of freedom, and $\tilde{\I}$ is the index set indexing the supernodes.
For a matrix $A$ and sets of indices $\tilde{i}$ and $\tilde{j}$, we denote as the $A_{\tilde{i},\tilde{j}}$ the submatrix obtained by restricting the indices of $A$ to $\tilde{i}$ and $\tilde{j}$, and as $A_{\tilde{i},:}$ ($A_{:,\tilde{j}}$) the matrix obtained by only restricting the row (column) indices.
We adopt the convention of indexing having precedence over inversion, i.e. $A_{\tilde{i},\tilde{j}}^{-1} = (A_{\tilde{i},\tilde{j}})^{-1}$.
For a supernode $\tilde{k} \in \tilde{I}$ and a degree of freedom $j \in \I$, we write $j \in \tilde{k}$ if there exists a $k \leadsto \tilde{k}$ such that $k \preceq j$ and $(k,j) \in S$, and we accordingly form submatrices $A_{\tilde{i}, \tilde{j}} \defeq (A_{ij})_{i \in \tilde{i}, j \in \tilde{j}}$.
Note that by definition of the supernodes, we have $s_k \subset \tilde{k}$ for all $k \leadsto \tilde{k}$.
Since we assume the sparsity pattern $S$ to contain the diagonal, we furthermore have $k \leadsto \tilde{k} \Rightarrow k \in \tilde{k}$.

We first show how to efficiently compute the inverse Cholesky factor for the aggregated sparsity pattern (as has been observed before by \cite{ferronato2015novel}, and \cite{guinness2016permutation}).  
For $\tilde{k} \in \tilde{\I}$, we define $U^{\tilde{k}}$ as the unique upper triangular matrix such that $\KM_{\tilde{k},\tilde{k}} = U^{\tilde{k}} U^{\tilde{k},\top}$. $U^{\tilde{k}}$ can be computed in complexity $\O((\# \tilde{k})^3)$ in time and $\O((\# \tilde{k})^2)$ in space by computing the Cholesky factorization of $\KM_{\tilde{k},\tilde{k}}$ after reverting the ordering of its rows and columns, and then reverting the order of the rows and columns of the resulting Cholesky factor.
The upper triangular structure of $U^{\tilde{k}}$ implies the following properties
\begin{align}
\label{eqn:restrictUUT}
    \KM_{s_k, s_k} &=  U^{\tilde{k}}_{s_k,s_k} U^{\tilde{k},\top}_{s_k,s_k},&
    \quad U_{s_k, s_k}^{\tilde{k},-1} \mathbf{1} &= \frac{1}{U_{kk}^{\tilde{k}}} \mathbf{e}_1,\\
    \quad U_{s_k, s_k}^{\tilde{k},-\top} \mathbf{1} &= \left(U^{\tilde{k},-\top} \mathbf{e}_k\right)_{s_k, s_k},&
    \quad U_{s_k, s_k}^{\tilde{k},-1} v_{s_k} &= \left(U^{\tilde{k},-1} v\right)_{s_k},
\end{align}
where $v \in \Reals^{\tilde{k}}$ is chosen arbitrarily.
For any $k \leadsto \tilde{k}$, the first three properties above imply 
\begin{equation}
L^{\rho}_{:,k} = \frac{\KM_{s_k}^{-1} \mathbf{e}_1}{\sqrt{\mathbf{e}_1^{\top} \KM_{s_k}^{-1} \mathbf{e}_1}} = U^{\tilde{k},-\top}_{s_k,s_k}  \mathbf{e}_1
=U^{\tilde{k},-\top}  \mathbf{e}_k.
\end{equation}
Thus, computing the columns $L_{:,k}$ for all $k \leadsto \tilde{k}$ has computational complexity $\O((\# \tilde{k})^3)$ in time and $\O((\# \tilde{k})^2)$ in space.
\cref{alg:aggregated} implements the formulae derived above.

\subsection{GP regression in \texorpdfstring{$\O(N + \rho^{2d})$}. space complexity}
\label{apssec:linearMemory}

As mentioned in \cref{ssec:lowmemory}, for many important operations arising in GP regression, the inverse-Cholesky factors $L$ of the training covariance matrix need never be formed in full. 
Instead, matrix-vector multiplies with $L$ or $L^{\top}$, as well as the computation of the log-determinant of $L$ can be performed by computing the columns of $L$ in an arbitrary order, using them to update the result, and deleting them again.
For the example of computing the posterior mean $\mu$ and covariance $C$, this is done in \cref{alg:lowMemNotAggregated} (without aggregation) and~\ref{alg:lowMemAggregated} (with aggregation).
In \cref{apsec:sortSparse}, we show how to compute the reverse maximin ordering and aggregated sparsity pattern in space complexity $\O(N + \rho^{d})$, thus allowing the entire algorithm to be run in space complexity $\O(N + \rho^{d})$ when using the aggregated sparsity pattern.

\begin{figure}[t]
    \scriptsize
	\begin{minipage}[t]{6.0cm}
		\vspace{0pt}
		\begin{algorithm}[H]
			\textbf{Input:} $\K$, $\left\{x_i\right\}_{i \in \I}$, $\prec$, $S_{\prec,\ell,\rho}$\\
			\textbf{Output:} Cond. mean $\mu$ and cov. $C$\\
			\begin{algorithmic}[1]
    			\FOR{$k \in \I_{\Pred}$}
    			    \STATE $\mu_k \leftarrow 0$
    			\ENDFOR
    			\FOR{$i \in \I_{\Train}, j \in I_{\Pred}$}
    			    \STATE $(\KM_{\Train, \Pred})_{ij} \leftarrow \K(x_i,x_j)$
    			\ENDFOR
    			\FOR{$i \in \I_{\Pred}, j \in I_{\Pred}$}
    			    \STATE $(\KM_{\Pred, \Pred})_{ij} \leftarrow \K(x_i,x_j)$
    			\ENDFOR
    			\FOR{$k \in \I_{\Train}$}
    			    \FOR{ $i,j \in s_{k}$}
    			        \STATE $\left(\KM_{s_k,s_k}\right)_{ij} \leftarrow \K(x_i,x_j)$
    			    \ENDFOR
    			    \STATE $v \leftarrow \KM_{s_k,s_k}^{-1} \mathbf{e}_k$
    			    \STATE $v \leftarrow v / v_k$
    			    \STATE $\mu_{k,:} \leftarrow \mu_{k,:} + v_k \KM_{k,\Pred}$
    			    \STATE $B_{k,:} \leftarrow v^{\top} \KM_{\Train,\Pred}$
    			\ENDFOR
    			\STATE $C \leftarrow \KM_{\Pred,\Pred} - B^{\top} B$
    			\RETURN $\mu, C$
    		\end{algorithmic}
			\caption{\label{alg:lowMemNotAggregated}Without aggregation}
		\end{algorithm}
	\end{minipage}
	\scriptsize
	\begin{minipage}[t]{6.0cm}
		\vspace{0pt}
		\begin{algorithm}[H]
			\textbf{Input:}$\K$, $\left\{x_i\right\}_{i \in \I}$, $\prec$, $S_{\prec, \ell, \rho, \lambda}$\\
			\textbf{Output:} Cond. mean $\mu$ and cov. $C$\\
			\begin{algorithmic}[1]
			    \FOR{$k \in \I_{\Pred}$}
			        \STATE $\mu_k \leftarrow 0$
			    \ENDFOR
			    \FOR{$i \in \I_{\Train}, j \in I_{\Pred}$}
			        \STATE $(\KM_{\Train, \Pred})_{ij} \leftarrow \K(x_i,x_j)$
			    \ENDFOR
			    \FOR{$i \in \I_{\Pred}, j \in I_{\Pred}$}
			        \STATE $(\KM_{\Pred, \Pred})_{ij} \leftarrow \K(x_i,x_j)$
			    \ENDFOR
			    \FOR{$\tilde{k} \in \tilde{\I}$}
			        \FOR{ $i,j \in s_{\tilde{k}}$}
			            \STATE $\left(\KM_{s_{\tilde{k}},s_{\tilde{k}}}\right)_{ij} \leftarrow \K(x_i,x_j)$
			        \ENDFOR
			        \STATE $U \leftarrow P^{\updownarrow}\chol( P^{\updownarrow} K_{s_{\tilde{k}}, s_{\tilde{k}}} P^{\updownarrow})P^{\updownarrow}$
			        \FOR{$k \leadsto \tilde{k}$}
      			    \STATE $v \leftarrow U^{-\top} \mathbf{e}_k$
      			    \STATE $\mu_{k,:} \leftarrow \mu_{k,:} + v_k \KM_{k,\Pred}$
      			    \STATE $B_{k,:} \leftarrow v^{\top} \KM_{\Train,\Pred}$
			        \ENDFOR
			    \ENDFOR
			    $C \leftarrow \KM_{\Pred,\Pred} - B^{\top} B$\;
			    \RETURN $\mu, C$
			\end{algorithmic}
			\caption{\label{alg:lowMemAggregated} With aggregation}
		\end{algorithm}
	\end{minipage}
	\caption{Prediction and uncertainty quantification using KL-minimization with and without aggregation in $\O(N+\rho^{2\tilde{d}})$ memory complexity} 
\end{figure}

\section{Postponed proofs}
\label{apsec:proofs}
Our theoretical results apply to more general orderings, called reverse $r$-maximin orderings, which for $r \in (0,1]$ have the following property.
\begin{definition}
    \label{def:rmaximin}
    An elimination ordering $\prec$ is called reverse $r$-maximin with length scales $\{\ell_i\}_{i \in \I}$ if for every $j \in \I$ we have 
    \begin{equation}
    \ell_j \defeq \min_{i \succ j} \dist(x_j, \{x_i\} \cup \partial \Omega) \geq r \max_{j \succ k} \min_{i \succ j} \dist(x_k, \{x_i\} \cup \partial \Omega) .
    \end{equation}
\end{definition}
We note that the reverse maximin ordering from \cref{ssec:maximin} is a reverse $1$-maximin ordering;
reverse $r$-maximin orderings with $r<1$ can be computed in computational complexity $\O(N\log(N))$ (see \cref{apsec:sortSparse}).
We define the sparsity patterns $S_{\prec, \ell, \rho}$ and $\tilde{S}_{\prec, \ell, \rho, \lambda}$ analogously to the case of the reverse maximin ordering, and we will write  $L^{\rho}$ for the incomplete Cholesky factors of $\KM^{-1}$ computed using \cref{eqn:defcolL} based on the sparsity pattern $S_{\prec,\ell,\rho}$ or $\tilde{S}_{\prec,\ell,\rho,\lambda}$.

\subsection{Computational complexity}
Our estimates only depends on the the \emph{intrinsic dimension of the dataset} which is defined by counting the number of balls of radius $r$ can be fit into balls of radius $R$, for different $r, R > 0$.
\begin{condition}[Intrinsic dimension]
\label{cond:intrinsicDimension}
    We say that $\left\{x_i\right\}_{i \in \I} \subset \Reals^d$ has intrinsic dimension $\tilde{d}$ if there exists a constant $C_{\tilde{d}}$, independent of $N$, such that for all $r,R > 0$, $x \in \Reals^d$, we have
    \begin{equation}
        \max\left\{ |A|: i,j \in A \Rightarrow 
        \dist(x_i, x), \dist(x_j, x) \leq R, 
        \dist(x_i,x_j) \geq r  \right\} 
        \leq C_{\tilde{d}}\left( R/r\right)^{\tilde{d}}.
    \end{equation}
\end{condition}

\begin{remark}
	Note that the we always have $\tilde{d} \leq d$.
\end{remark}

We also make a mild technical assumption requiring that most of the points belong to the finer scales of the ordering:
\begin{condition}[Regular refinement]
    \label{cond:regref}
    We say that $\left\{x_i\right\}_{i \in \I} \subset \Reals^d$ fulfills the regular refinement condition for $\lambda$ and $\ell$ with constant $C_{\lambda,\ell}$, if 
   	\begin{equation*}
		\sum \limits_{k = \lfloor \log(\ell_1)/ \log(\lambda) \rfloor }^\infty \# \{ i: \lambda^{k} \leq \ell_i \} \leq C_{\lambda,\ell} N 
	\end{equation*}
\end{condition}
This condition excludes pathological cases like $x_{i} = 2^{-i}$ for which each scale contains the same number of points.

Analogously to the results of \cite{Schafer2017}, we obtain the following computational complexity:
\begin{theorem}
    Under \cref{cond:intrinsicDimension} with $C_{\tilde{d}}$ and  $\tilde{d}$,  using an $r$-reverse maximin ordering $\prec$ and  $S_{\prec,\ell,\rho}$, \cref{alg:notAggregated} computes $L^{\rho}$ in  complexity $CN\rho^{\tilde{d}}$ in space and $CN\rho^{3\tilde{d}}$ in time. 
    If we assume in addition that $\left\{x_i\right\}_{i \in \I}$ fulfills \cref{cond:regref} for $\lambda$ and $l$ with constant $C_{\lambda, \ell}$, then, using $\tilde{S}_{\prec,\ell,\rho,\lambda}$ or $\bar{S}_{\prec,\ell,\rho,\lambda}$, \cref{alg:aggregated}
    computes $L^{\rho}$ in complexity $CN\rho^{\tilde{d}}$ in space and $C_{\lambda, \ell}CN\rho^{2\tilde{d}}$ in time. 
    Here, the constant $C$ depends only on $C_{\tilde{d}}$, $\tilde{d}$, $r$, $\lambda$, and the maximal cost of evaluating a single entry of $\KM$, but not on $N$ or $d$.
\end{theorem}

\begin{proof}
    We begin by showing that the number of nonzero entries of an arbitrary column of $S_{\prec, \ell, \rho}$ is bounded above as $C\rho^{\tilde{d}}$.
    Considering the $i$-th column, the reverse $r$-maximin ordering ensures that for all $j,k \succ i$, we have $\dist(x_j, x_i) \geq r \ell_i$.                    
    Since for all $(i,j) \in S_{\prec, \ell, \rho}$ we have $i \prec j$ and $\dist(x_i,x_j) \leq \rho \ell_i$, \cref{cond:intrinsicDimension} implies that $\# \left\{ j : (i,j) \in S_{\prec,\ell,\rho} \right\} \leq C_{\tilde{d}} \left( \frac{\rho \ell_i}{r \ell_i} \right)^{\tilde{d}}$.
    Computing the $i$-th column of $L^\rho$ requires the inversion of the Matrix $\KM_{s_i,s_i}$ which can be done in computational complexity $C \rho^{\tilde{3d}}$, leaving us with a total time complexity of $CN\rho^{\tilde{d}}$.
	We now want to bound the computational complexity when using the aggregated sparsity patterns $\tilde{S}_{\prec,\ell,\rho,\lambda}$ or 
	$\bar{S}_{\prec,\ell,\rho,\lambda}$.
    As before, we write $j \in s$ if $j$ is a child of the supernode $s$, that is if there exists a $i \leadsto s$ such that $(i,j)$ is contained in $\tilde{S}_{\prec,\ell,\rho,\lambda}$ or $\bar{S}_{\prec,\ell,\rho,\lambda}$. 
    We write $\# s$ to denote the number of children of $s$.
	By the same argument as above, the number of \emph{children} in each supernode $s$ is bounded by $C\rho^{\tilde{d}}$. 
	We now want to show that the sum of the numbers of children of all supernodes is bounded as $CN$.
	For a supernode $s$ we write $\sqrt{s} \in \I$ to denote the index that was first added to the supernode (see the construction described in \cref{ssec:aggregated}).
	We now observe that for two distinct supernodes $s$ and $t$ with $c \leq \ell_{\sqrt{s}}, \ell_{\sqrt{t}} \leq c \lambda$, we have $\dist(x_{\sqrt{s}}, x_{\sqrt{t}}) \geq c \rho$, since otherwise we would have either $\sqrt{s} \leadsto t$ or$\sqrt{t} \leadsto s$.
	Thus, for every index $i \in \I$ and $k \in \mathbb{Z}$, there exist at most $C$ supernodes $s$ with $i \in s$, $\lambda^k \leq \ell_{\sqrt{s}} < \lambda^{k+1}$. 
	By using \cref{cond:regref}, we thus obtain
	\begin{align*}
		&\sum \limits_{s \in \tilde{\I}} \# s 
		= \sum \limits_{i \in \I} \# \left\{ s \in \tilde{I} : i \in s\right\} 
		= \sum \limits_{k \in \mathbb{Z}} \sum \limits_{i \in \I} \# \left\{ s \in \tilde{s} : i \in s, \lambda^{k} \leq \ell_{\sqrt{s}} < \lambda^{k+1} \right\} \\
		&\leq \sum \limits_{k \in \mathbb{Z}} \sum \limits_{i\in \I: \ell_i \geq \lambda^k} C
		\leq N C.
	\end{align*}
	We now know that there are at most $CN$ child-parent relationships between indices and supernodes and that each supernode can have at most $C\rho^{\tilde{d}}$ children.
	The worst case is thus that we have $CN/\rho^{\tilde{d}}$ supernodes, each having $C\rho^{\tilde{d}}$ children. 
	This leads to the bounds on time-- and space complexity of the algorithm. 
\end{proof}

\subsection{Approximation accuracy}
Our goal is to prove the following theorem:
\begin{theorem}
\label{thm:ap_accuracy}
Using an $r$-maximin ordering $\prec$ and sparsity patterns $S_{\prec, \ell, \rho}$ or $\tilde{S}_{\prec, \ell, \rho, \lambda}$, there exists a constant $C$ depending only on $d$, $\Omega$, $r$, $\lambda$, $s$, $\|\IK\|$, $\|\IK^{-1}\|$, and $\delta$, such that  for $\rho \geq C \log(N/\epsilon)$, we have 
\begin{equation}
\textstyle    \KL*{\N\left(0,\KM\right)}{ \N(0, \left(L^\rho L^{\rho,\top})^{-1}\right)}  
   \, + \, \left\| \KM - ( L^{\rho} L^{\rho,\top} )^{-1} \right\|_{\FRO}
   \, \leq  \,\epsilon.
\end{equation}
Thus, \cref{alg:notAggregated} computes an $\epsilon$-accurate approximation of $\KM$ in computational complexity $CN\log^d( N/\epsilon )$ in space and $CN \log^{3d}( N/\epsilon )$ in time, from  $CN\log^d( N/\epsilon )$ entries of $\KM$.
Similarly, \cref{alg:aggregated} computes an $\epsilon$-accurate approximation of $\KM$ in computational complexity $CN\log^d( N/\epsilon )$ in space and $CN \log^{2d}( N/\epsilon )$  in time, from $CN\log^d( N/\epsilon )$ entries of $\KM$.
\end{theorem}

\cite{Schafer2017} prove that under the conditions of \cref{thm:accuracy} the Cholesky factor of $\IKM=\KM^{-1}$ decays exponentially away from the diagonal.
\begin{theorem}[{\cite[Thm.~4.1]{Schafer2017}}]
    \label{thm:accIchol}
	In the setting of \cref{thm:accuracy}, there exists a constant $C$ depending only on $\delta, r, d, \Omega, s, \|\IK\|$, and $\|\IK^{-1}\|$, such that for $\rho \geq C \log(N/\epsilon)$,
	\begin{equation}
		S \supset \{ (i,j) \in \I \times \I : \dist(x_i, x_j) \leq \rho \min(\ell_i, \ell_j)\}
	\end{equation}
	and 
	\begin{equation}
		L^S_{ij} \defeq 
		\begin{cases}
			\big(\chol(\IKM)\big)_{ij}, & (i,j) \in S, \\
			0, & \text{otherwise,}
		\end{cases}
	\end{equation}
	we have $\left\|\IKM - L^S L^{S,\top}\right\|_{\FRO} \leq \epsilon$.
\end{theorem}
In order to prove the approximation accuracy of the KL-minimizer, we have to compare the approximation accuracy in Frobenius norm and in KL-divergence.
For brevity, we write $\KL*{A}{B} \defeq \KL{\N\left(0, A\right)}{\N\left(0,B\right)}$.

\begin{lemma}
    \label{lem:compKLFRO}
    Let $\lambda_{\min}$, $\lambda_{\max}$ be the minimal and maximal eigenvalues of $\KM$, respectively.
    Then there exists a universal constant $C$ such that for any matrix $M \in \Reals^{\I \times \I}$, we have 
    \begin{align*}
    \lambda_{\max} \left\|\IKM - M M^{\top} \right\|_{\FRO} \leq C &\Rightarrow \KL*{\Theta}{\left(M M^{\top}\right)^{-1}} \leq \lambda_{\max} \left\|\IKM - M M^{\top} \right\|_{\FRO},  \\
    \KL*{\Theta}{\left(M M^{\top}\right)^{-1}} \leq C &\Rightarrow \left\|\IKM - M M^{\top} \right\|_{\FRO} \leq \lambda_{\min}^{-1} \KL*{\Theta}{\left(M M^{\top}\right)^{-1}}.
    \end{align*}
\end{lemma}         
\begin{proof}
Writing $L \defeq \chol(A)$ and $\phi_{\FRO}(x) \defeq x^2$ and $\phi_{\operatorname{KL}}(x) \defeq (x - \log(1 + x))/2$, we have 
\small
\begin{align*}
    & \lambda_{\min} \left\| \IKM - M M^{\top} \right\|_{\FRO}
    = \lambda_{\min} \left\|L L^{-1} \left(\IKM - M M^{\top} \right) L^{-\top} L^{\top} \right\|_{\FRO}\\
    &\leq \left\|\Id - L^{-1} M M^{\top} L^{-\top} \right\|_{\FRO} 
    = \sum \limits_{k = 1}^N \phi_{\FRO} \left(\lambda_{k} \left( L^{-1} M M^{\top} L^{-\top} \right) - 1\right)\\
    &=  \left\|L^{-1} \left(\IKM - M M^{\top} \right) L^{-\top} \right\|_{\FRO}
    \leq \lambda_{\max} \left\|\IKM - M M^{\top} \right\|_{\FRO}
\end{align*}
and         
\begin{equation}
    \KL*{\KM}{\left(M M^{\top}\right)^{-1}} 
    = \sum \limits_{k = 1}^N \phi_{\operatorname{KL}} \left( \lambda_{k}\left( L^{-1} M M^{\top} L^{-\top} \right) \right),
\end{equation}
where $\left(\lambda_{k}(\cdot)\right)_{1  \leq k \leq N}$ returns the eigenvalues ordered from largest to smallest, while   $\lambda_{\min}(\cdot)$ ($\lambda_{\max}(\cdot)$) returns the smallest (largest) eigenvalue.
The leading-order Taylor expansion of $\phi_{\operatorname{KL}}$ around $0$ is given by $x \mapsto x^2/4$. 
Thus, there exists a constant $C$ such that for $\min(|x|, \phi_{\FRO}(x), \phi_{\operatorname{KL}}(x)) \leq C$ we have $\phi_{\operatorname{KL}}(x) \leq \phi_{\FRO}(x) \leq 8 \phi_{\operatorname{KL}} (x)$. 
Therefore, for $\lambda_{\max} \left\|\IKM - M M^{\top} \right\|_{\FRO} \leq C$ we have $\KL*{\Theta}{\left(M M^{\top}\right)^{-1}} \leq \lambda_{\max} \left\|\IKM - M M^{\top} \right\|_{\FRO}$. 
For $\KL*{\Theta}{\left(M M^{\top}\right)^{-1}} \leq C$ this implies  
$\left\|\IKM - M M^{\top} \right\|_{\FRO} \leq \lambda_{\min}^{-1} \KL*{\Theta}{\left(M M^{\top}\right)^{-1}}$.
\end{proof}

Using \cref{lem:compKLFRO}, we can now use the results of \cite{Schafer2017} to conclude \cref{thm:accuracy}.
\begin{proof}[Proof of \cref{thm:ap_accuracy}]
    \cite[Thm.~3.16]{Schafer2017} implies that there exists a polynomial $\mathbf{p}$ depending only on $(d,s,\delta, \IK)$ such that $\lambda_{\max}, \lambda_{\min}^{-1} \leq \mathbf{p}(N)$. 
    Thus, by choosing $\rho \geq C \log(N)$ we can deduce by \cref{thm:accIchol} that $\lambda_{\max} \left\| \IKM - L^S L^{S,\top} \right\| \leq C$ for $C$ the constant in \cref{lem:compKLFRO}. 
    Thus, We have $\KL*{\KM}{\left(L^{S}L^{S,\top}\right)^{-1}} \leq \lambda_{\max} \left\| \IKM - L^S L^{S,\top} \right\|$.
    The KL-optimality of $L^{\rho}$, implies  $\KL*{\KM}{\left(L^{\rho}L^{\rho,\top}\right)^{-1}} \leq \lambda_{\max} \left\| \IKM - L^S L^{S,\top} \right\| \leq C$.
    Using one more time \cref{lem:compKLFRO}, we also obtain 
    \begin{equation}
        \left\| \IKM - L^{\rho}L^{\rho,\top} \right\| \leq \lambda_{\min}^{-1} \KL*{\KM}{\left(L^{\rho}L^{\rho,\top}\right)^{-1}} \leq \lambda_{\max}/\lambda_{\min} \left\| \IKM - L^{S}L^{S,\top} \right\|.
    \end{equation}

\end{proof}

\section{Computation of the ordering and sparsity pattern}
\label{apsec:sortSparse}

We will now explain how to compute the ordering and sparsity pattern described in \cref{sec:ordSparse},
using only near-linearly many evaluations of an oracle $\disttt(i,j)$ that returns the distance between the points $x_i$ and $x_j$.
To do so efficiently in general, we need to impose a mild additional assumption on the dataset (c.f. \cite{Schafer2017}).
\begin{condition}[Polynomial Scaling]
\label{cond:polynomialScaling}
   There exists a polynomial $\mathbf{p}$ for which 
   \begin{equation*}
       \frac{\max_{i \neq j \in \I} \dist(x_i, x_j)}{\min_{i \neq j \in \I} \dist(x_i, x_j)} 
       \leq \mathbf{p}(N).
   \end{equation*}
\end{condition}
Under \cref{cond:intrinsicDimension,cond:polynomialScaling} , \cite[Alg.~3]{Schafer2017} allows us to compute the maximin ordering $\prec$ and sparsity pattern $\left\{(i,j): \dist(x_i, x_j) \leq \rho \max(\ell_i, \ell_j)\right\}$ in computational complexity $\O(N \log(N) \rho^{\tilde{d}})$ in space and time.
The resulting pattern is larger than the sparsity pattern $S_{\prec, l \rho}$ introduced in \cref{ssec:maximin}, which can thus be obtained by truncating the pattern obtained by \cite[Alg.~3]{Schafer2017}.
By performing the truncation of the sets of children and parents $c,p$ as used by \cite[Alg.~3]{Schafer2017} during execution of the algorithm, as opposed to truncating the sparsity pattern after execution of the algorithm, the space complexity for obtaining $\prec$ and $S_{\prec,l,\rho}$ can be reduced to $\O(N \rho^{\tilde{d}})$.
\cref{alg:sortSparseTrunc} is a minor modification of \cite[Alg.~3]{Schafer2017} that performs such a truncation.
\begin{theorem}[Variant of {\cite[Thm.~A.5]{Schafer2017}}]
    Let $\Omega = \Reals^d$ and $\rho \geq 2$. 
    \cref{alg:sortSparseTrunc} computes  the reverse maximin ordering $\prec$ and sparsity pattern $S_{\prec, l, \rho}$ in computational complexity 
    $C\rho^{\tilde{d}} N$ in space and $CN\log(N)\rho^{\tilde{d}}(\log N + C_{\disttt})$ in time.
    Here, $C = C(\tilde{d}, C_{\tilde{d}}, \mathbf{p})$ depends only on the constants appearing in \cref{cond:intrinsicDimension,cond:polynomialScaling}, and $C_{\disttt}$ is the cost of evaluating $\disttt$.
\end{theorem}
\begin{proof}
    The proof is essentially the same as the proof of \cite[Thm.~A.5]{Schafer2017}.
\end{proof}
Similarly, the proof of \cite[Thm.~A.2]{Schafer2017} can be adapted to show that in the setting of \cref{thm:accuracy}, there exists a constant $C$ depending only on $d$,  $\Omega$, and $\delta$, such that for $\rho > C$,  \cref{alg:sortSparseTrunc} computes the maximin ordering in computational complexity $CN (\log(N) \rho^{d} +C_{\disttt_{\Omega}})$ in time and $CN \rho^{d}$ in space, where $C_{\disttt_{\Omega}}$ is an upper bound on the complexity of computing the distance of an arbitrary point $x \in \Omega$ to $\partial \Omega$.

We furthermore note that a reverse $r$-maximin ordering with $r<1$ (see \cref{def:rmaximin}) can be computed in computational complexity $\O(N\log(N))$ by quantizing the values of $(\log(\ell_i))_{i \in \I}$ in multiples of $\log(r)$, which avoids the complexity incurred by the restoration of the heap property in Line 20 of \cite[Alg.~3]{Schafer2017}.

\begin{algorithm}[!ht]\footnotesize
  \textbf{Input:} A real parameter $\rho \geq 2$ and Oracles $\disttt( \quark, \quark ), \disttt_{\partial \Omega}( \quark )$ such that
    $\disttt ( i, j ) = \dist\left(x_i, x_j\right)$ and
  $\disttt_{\partial \Omega} \left(i\right) =
  \dist\left( x_i, \partial \Omega \right)$ \\
  \textbf{Output:} An array $l[:]$ of distances, an array $P$ encoding the multiresolution
    ordering, and an array of index pairs $S$ containing the sparsity pattern.\\
    \begin{algorithmic}[1]
      \STATE $P = \emptyset$
      \FOR{$i \in \{1,  \dots , N\}$}
        \STATE $l[i] \leftarrow \disttt_{\partial \Omega}(i)$\;
        \STATE $p[i] \leftarrow \emptyset$\;
        \STATE $c[i] \leftarrow \emptyset$\;
      \ENDFOR
      \STATE \COMMENT{Creates a mutable binary heap, containing pairs of indices and
      distances as elements:}
      \STATE $H \leftarrow \mathtt{MutableMaximalBinaryHeap}\left( \{ (i, l[i] ) \}_{i \in \{1, \dots ,N\}} \right)$
      \STATE \COMMENT{Instates the Heap property, with a pair with maximal distance occupying the root of the heap:}
      \STATE $\mathtt{heapSort}!( H )$
    
      \STATE \COMMENT{Processing the first index:}
      \STATE \COMMENT{Get the root of the heap, remove it, and restore the heap property:}
      \STATE $(i, l) = \mathtt{pop}(H)$
      \STATE \COMMENT{Add the index as the next element of the ordering}
      \STATE $\mathtt{push}\left(P, i\right)$
      \FOR{$j \in \{1,  \dots , N\}$ }
        \STATE $\mathtt{push}( c[i], j )$
        \STATE $\mathtt{push}( p[j], i )$
        \STATE $\mathtt{sort!}\left( c[i], \disttt( \quark, i ) \right)$
        \STATE $\mathtt{decrease!}\left(H, j, \disttt( i, j ) \right)$
      \ENDFOR
      \STATE \COMMENT{Processing  remaining indices:}
      \STATE $l_{\operatorname{Trunc}} \leftarrow l$
      \WHILE{$H \neq \emptyset$}
        \label{line-whileSortSparse}
        \STATE \COMMENT{Get the root of the heap, remove it, and restore the heap property:}
        \STATE $(i, l) = \mathtt{pop}(H)$
        \STATE $l[i] \leftarrow l$
        \STATE \COMMENT{Select the parent that has possible children of $i$ amongst its
          children, and is closest to $i$:}
        \STATE $k = \argmin_{j \in p[i]: \disttt( i, j ) + \rho l[i] \leq \rho \min(l_{\operatorname{Trunc}},l[j])}
        \label{line-kArgminSortSparse}
      \disttt\left( i, j \right)$\;
        \STATE \COMMENT{Loop through those children of $k$ that are close enough to $k$ to possibly
        be children of $i$:}
        \FOR{$j \in c[k]: \disttt( j, k ) \leq \disttt( i, k ) +
          \rho l[i] $}
          \label{line-forSortSparse}
          \STATE $\mathtt{decrease!}\left(H, j, \disttt( i, j ) \right)$
          \IF{ $\disttt( i, j ) \leq \rho l[i]  $}
            \STATE $\mathtt{push}( c[i], j )$
            \STATE $\mathtt{push}( p[j], i )$
          \ENDIF
        \ENDFOR
        \STATE \COMMENT{Add the index as the next element of the ordering}
        \STATE $\mathtt{push}\left(P, i\right)$
        \STATE \COMMENT{Sort the children according to distance to the parent node, so that the closest
        children can be found more easily}
        \STATE $\mathtt{sort!}\left( c[i], \disttt( \quark, i ) \right)$
        \STATE \COMMENT{Truncate the sparsity pattern to achieve linear space complexity}
        \IF{$\forall j \notin P, \exists i \in P: \disttt(i,j) < l_{\operatorname{Trunc}}/2$}
            \STATE $l_{\operatorname{Trunc}} \leftarrow l_{\operatorname{Trunc}}/2 $
            \FOR{$j \in c[i] \setminus P, \disttt(i,j) > \rho l_{\operatorname{Trunc}}$}
                \STATE $c[i] \leftarrow c[i] \setminus \{j\}$
                \STATE $p[j] \leftarrow p[j] \setminus \{i\}$
            \ENDFOR
        \ENDIF[c.f.][]
        \label{line-sort!SortSparse}
      \ENDWHILE
      \STATE \COMMENT{Aggregating the lists of children into the sparsity pattern:}
      \FOR{$i \in \{1, \dots ,N\}$}
        \FOR{ $j \in c[i]$}
          \STATE $\mathtt{push!}\left(S, ( i, j ) \right) $
          \STATE $\mathtt{push!}\left(S, ( j, j ) \right) $
        \ENDFOR
      \ENDFOR
   \end{algorithmic}
   \caption{Ordering and sparsity pattern algorithm (see \cite[Alg.~3]{Schafer2017}).}
   \label{alg:sortSparseTrunc}
\end{algorithm}

As described in \cref{ssec:aggregated}, the aggregated sparsity pattern $S_{\prec,l,\rho,\lambda}$ can be computed efficiently from $S_{\prec,l,\rho}$.
However, forming the pattern $S_{\prec,l,\rho}$ using a variant of \cite[Alg.~3]{Schafer2017} has complexity $\O(N \log(N) \rho^{\tilde{d}})$ in time and $\O(N\rho^{\tilde{d}})$ in space, while the aggregated pattern $S_{\prec, l, \rho, \lambda}$ only has space complexity $\O(N)$, begging the question if this computational complexity can be improved. 
Let $\{s_{\tilde{i}}\}_{\tilde{i} \in \tilde{I}}$ be the supernodes as constructed in \cref{ssec:aggregated} and identify each supernodal index $\tilde{i}$ with the first (w.r.t. $\prec$) index $i \in \I$ such that $i \leadsto \tilde{i}$.
We then define 
\begin{equation}
    \bar{S}_{\prec,l,\rho,\lambda} \defeq 
    \bigcup_{\tilde{i} \in \tilde{I}} \left\{(i,j): i \preceq j, i \leadsto \tilde{i}, \dist(x_{\tilde{i}},x_j) \leq \rho (1+\lambda) \tilde{i}\right\}.
\end{equation}
\cref{alg:sparsityAggregated} allows us to construct the sparsity pattern $\bar{S}_{\prec,l,\rho,\lambda} \supset \tilde{S}_{\prec,l,\rho,\lambda}$ in complexity $\O(N \log(N))$ in time and $\O(N)$ in space, given $\prec$ and $l$.
In this algorithm, we will implement supernodes as pairs of arrays of indices $\sigma = (\sigma_m, \sigma_n)$.
This encodes the relationship between all indices in $\sigma_m$ (the \emph{parents}) and all indices in $\sigma_n$ (the \emph{children}). 
Naively, this would require  $\O(\# \sigma_m \# \sigma_n)$ space complexity, but by storing the entries of $\sigma_m$ and $\sigma_n$, the complexity is reduced to $\O(\#\sigma_m + \#\sigma_n)$ space complexity, which improves the asymptotic computational complexity.
\begin{figure}[t]
	\scriptsize
	\begin{minipage}[t]{7.3cm}
		\vspace{0pt}
		\begin{algorithm}[H]
			\textbf{Input:} $\I, \prec, l, \disttt(\cdot,\cdot)$, $\rho$, $\lambda$\\
			\textbf{Output:} Sets of supernodes $\bar{S}$, $\tilde{S}$ \\
			\begin{algorithmic}[1]
    			\STATE $i_N, i_{N-1} \leftarrow$ last two ind. w.r.t. $\prec$
    			\STATE $\N, \N^{\geq}  \leftarrow \{(\{i_N\}, \I)\}, \{ (\{i_N\}, \{i_N\})\}$\;
    			\STATE $r, l_{i_{N}} \leftarrow l_{i_{N-1}}/ \lambda, \infty$
                \WHILE{$r > \min_{i \in \I} l_i$}
                    \STATE $\N, \tilde{\N}^{\geq} \leftarrow \texttt{Refine}(\N, l, r, \rho, \lambda)$
                    \STATE $\N^{\geq} \leftarrow \N^{\geq} \cup \tilde{\N}^{\geq}$
                    \STATE $ r \leftarrow r / \lambda$
                \ENDWHILE
                \STATE $\J, \bar{S} \leftarrow \emptyset, \emptyset$
                \FOR{$i \in \I$ (in increasing order by $\prec$)}
                    \IF{$i \notin \J$}
                        \STATE $\tilde{s} \leftarrow (\emptyset, \emptyset)$,
                        Pick $\sigma \in \N^{\geq}$ : $i \in \sigma_m$
                        \FOR{$j \in \sigma_m$}
                            \IF{$i \preceq j$, $l_j \leq \lambda l_i$, $j \notin \J$}
                                \STATE $\J, \tilde{s}_n  \leftarrow \J \cup \{j\}, \tilde{s}_n \cup \{j\}$
                            \ENDIF
    		            \ENDFOR
    		            \FOR{$\tilde{\sigma} \in \N^{\geq}: \exists j \in \tilde{\sigma}_m: j \in \sigma_{m}$}
    		                \FOR{$k \in \tilde{\sigma}_n: \disttt(i,k) \leq \rho(1 + \lambda)$}
    		                    \STATE $\tilde{s}_m \leftarrow \tilde{s}_m \cup \{k\}$	
    		                \ENDFOR
    		            \ENDFOR  
    		            \STATE $\bar{S} \leftarrow \bar{S} \cup \{\tilde{s}\}$
                    \ENDIF
                \ENDFOR
                \STATE $\tilde{S} \leftarrow \texttt{Reduce}(\rho, \prec, l, \bar{S})$\;
    			\RETURN $\bar{S}, \tilde{S}$\;
    		\end{algorithmic}
			\caption{\label{alg:sparsityAggregated} Computation of $\tilde{S}$ and $\bar{S}$}
		\end{algorithm}
	\scriptsize
	\begin{algorithm}[H]
		\textbf{Input:} $\prec, l, \disttt(\cdot,\cdot)$, $\rho$, $\bar{S}$\\
		\textbf{Output}$\tilde{S} =  \tilde{S}_{\prec,l,\rho}$\\
		\begin{algorithmic}[1]
    		\STATE $\tilde{S} \leftarrow \emptyset$
    		\FOR{$\sigma \in \bar{S}$}
    		    \STATE $\tilde{s} \leftarrow (\sigma_m, \emptyset)$
    	        \FOR{$i \in \sigma_m, j \in \sigma_n$}
    	            \IF{$\disttt(i,j) \leq \rho l_i$ and $i \prec j$}
    	                \STATE $\tilde{s}_n \leftarrow \tilde{s}_n \cup \{j\}$
    	            \ENDIF
    	        \ENDFOR	
    	        \STATE $\tilde{S} \leftarrow \tilde{S} \cup \{\tilde{s}\}$
    	    \ENDFOR
    		\RETURN $\tilde{S}$
    	\end{algorithmic}
		\caption{$\label{alg:reduce} \texttt{Reduce}(\rho, \prec, l, \bar{S})$}
	\end{algorithm}	
	\end{minipage}
	\scriptsize
	\begin{minipage}[t]{5.6cm}
		\vspace{0pt}
		\begin{algorithm}[H]
			\textbf{Input:} Supernodal set $\N$, \\ $\disttt(\cdot,\cdot)$, $l$, $r$, $\rho$, $\lambda$\\ 
			\textbf{Output:} A new set $\mathcal{M}$ of supernodes, set $\N^{\geq}$ of truncated supernodes\\
			\begin{algorithmic}[1]
    			\STATE $\J, \mathcal{N}^{\geq}, \mathcal{M} \leftarrow \emptyset, \emptyset, \emptyset$
    		    \WHILE{$\J \neq \I$}
    		        \STATE Pick $i \in \I \setminus \J$
    		        \STATE $\J \leftarrow \J \cup \{i\}$
    		        \STATE Pick $\sigma \in \mathcal{N}$ satisfying $i \in \sigma_m$
    		        \STATE $\tilde{\sigma}_m, \tilde{\sigma}_n \leftarrow \emptyset, \emptyset$ 
    		        \FOR{$j \in \sigma_n$}
    		            \IF{$\disttt(i,j) \leq \rho \lambda r$}
    		                \STATE $\tilde{\sigma}_m \leftarrow \tilde{\sigma}_m \cup \{j\}$
    		                \STATE $\J \leftarrow \J \cup \{j\}$
    		            \ENDIF
    		            \IF{$\disttt(i,j) \leq 2 \rho \lambda r$}
    		                \STATE $\tilde{\sigma}_n \leftarrow \tilde{\sigma}_n \cup \{j\}$\;
    		            \ENDIF
    		        \ENDFOR
    		        \STATE $\mathcal{M} \leftarrow \mathcal{M} \cup \{(\tilde{\sigma}_m, \tilde{\sigma}_n)\}$
    		    \ENDWHILE
    	        \FOR{$\sigma \in \N$}
    	            \STATE $\sigma^{\geq}_m, \sigma^{\geq}_n \leftarrow \emptyset, \emptyset$
    	            \FOR{$i \in \sigma_m$}
    	                \IF{$r \leq l_i$}
    	                     \STATE $\sigma^{\geq}_m \leftarrow \sigma^{\geq}_m \cup \{i\}$ 
    	                \ENDIF
    	            \ENDFOR
    	            \FOR{$i \in \sigma_n$}
    	                \IF{$r \leq l_i$}
    	                    \STATE $\sigma^{\geq}_n \leftarrow \sigma^{\geq}_n \cup \{i\} $
    	                \ENDIF
    	            \ENDFOR
    	            \STATE $\N^{\geq} \leftarrow \N^{\geq} \cup \{ (\sigma^{\geq}_m, \sigma^{\geq}_n)\}$
    	        \ENDFOR
    			\RETURN $\mathcal{M}$, $\N^{\geq}$\;
    		\end{algorithmic}
    		\caption{\label{alg:aggregateNodes} \newline $\texttt{Refine}(\N, l,r,\rho, \lambda)$}
		\end{algorithm}
	\end{minipage}
	\caption{Algorithm for constructing the aggregated sparsity pattern from the reverse maximin ordering $\prec$ and length-scales $l$} 
\end{figure}

\begin{theorem}
Let the $\{x_i\}_{i \in \I}$ satisfy \cref{cond:intrinsicDimension} with $\tilde{d}$ and $C_{\tilde{d}}$, \cref{cond:regref} with constant $C_{\operatorname{regref}}$, and \cref{cond:polynomialScaling} with $\mathbf{p}$ and let $\lambda >1$. 
Then there exists a constant $C = C_{\tilde{d}, C_{\tilde{d}}, C_{\operatorname{regref}}, \mathbf{p}, \lambda}$ such that \cref{alg:sparsityAggregated} can compute $\bar{S}_{\prec,l,\rho,\lambda}$ in computational complexity $CC_{\disttt}N\log(N)$ in time and $CN$ in space and $\tilde{S}_{\prec,l,\rho,\lambda}$ in computational complexity $C C_{\disttt} N( \log(N) + \rho^{\tilde{d}})$ in time and $CN$ in space.
\end{theorem}
\begin{proof}
To establish correctness, the main observation is that after every application of \cref{alg:aggregateNodes},
each degree of freedom $i$ can be found in exactly one of the supernodes $\sigma \in \N$.
Furthermore, each $\sigma \in \N$ has a \emph{root} $\sqrt{\sigma} \in \I$ such that 
\begin{enumerate}
\item $\sqrt{\sigma} \in \sigma_n, \sigma_m$, 
\item \label{prop:sqrtradius} $j \in \sigma_m \Rightarrow \disttt(\sqrt{\sigma},j) \leq \rho \lambda r$,
\item $j \in \sigma_n \Leftrightarrow \disttt(\sqrt{\sigma},j) \leq 2 \rho \lambda r$,
\item $\sigma \neq \bar{\sigma} \Rightarrow \disttt(\sqrt{\sigma}, \sqrt{\bar{\sigma}}) > \rho \lambda r$.
\end{enumerate}
The main reason why the above could fail to hold true is that the inner for loop does not range over all $j \in \I$, but only over those in $\sigma_n$. 
However, at the first occurrence of \cref{alg:aggregateNodes} we have $\sigma_n = \I$ leading to the observations to hold true. 
For subsequent calls, we can show the invariance of these properties by induction.
The set $\N^{\geq}$ is obtained from the set $\N$ by only selecting the points in a certain range of length scales. 
Therefore, after completion of the while-loop of \cref{alg:sparsityAggregated}, every $i \in \I$ is contained in at least one of the $\{\sigma_m\}_{\sigma \in \N^{\geq}}$, and for $i \leadsto \sigma \in \N^{\geq}$, we have $\left\{j:(i,j) \in \bar{S}\right\} \subset \sigma_n$.
Thus, the for-loop of \cref{alg:sparsityAggregated} indeed computes $\bar{S}$.
Since $\tilde{S}_{\prec,l,\rho,\lambda} \subset \bar{S}_{\prec,l,\rho,\lambda}$, \cref{alg:reduce} correctly recovers $\tilde{S}_{\prec,l,\rho,\lambda}$.

We begin by analyzing the computational complexity of the while-loop of \cref{alg:sparsityAggregated}.
We first note that at every execution of the loop, $r$ is divided by $\lambda$. 
Thus, \cref{cond:polynomialScaling} implies that the loop is entered at most $C\log(N)$ times.
We now claim that that the time complexity of \cref{alg:aggregateNodes} is bounded above by $CC_{\disttt}N$.
To this end it is enough to upper-bound the number of points $i$ for which a given index can be picked as index $j$ in the while-loop of \cref{alg:aggregateNodes}.
By \cref{prop:sqrtradius}, for this to happen we need $\disttt(i,\sqrt{\sigma}) \leq \rho \lambda r$ and $\disttt(j,\sqrt{\sigma}) \leq 2 \rho \lambda r$ and hence, by the triangle inequality, $\disttt(i,j) \leq 3 \rho \lambda r$.
On the other hand, $i$ can not be in $\J$ already, which means that any two distinct $i_1, i_2$ have to satisfy $\disttt(i_1,i_2) > \rho \lambda r$.
By \cref{cond:intrinsicDimension}, we conclude that the maximum number of indices $i$ for which a given index $j$ gets picked is bounded above by a constant $C$ that depends of $C_{\tilde{d}}$ and $d$. 
This upper bounds the computational complexity of the while-loop in \cref{alg:aggregateNodes} by $C N$.
The computational complexity of the outermost for-loop of \cref{alg:aggregateNodes} can be bounded by $C N$, by a similar argument.
Summarizing the above, we have upper-bounded the time complexity of the while-loop in \cref{alg:sparsityAggregated} by $C N \log(N)$.
In order to bound the space complexity of the while-loop, we need to ensure that the size of $\N^{\geq}$ is bounded above as $C N$.
To this end, we notice that by arguments similar to the above, one can show that at all times, the number $\max_{i \in \I} \# \{\sigma \in \N: i \in \sigma_m \text{ or } i \in \sigma_N \}$ is bounded from above by a constant $C$.
By using \cref{cond:regref}, we can show that the space complexity of $\N^{\geq}$ is bounded by $C N$.
By ways of similar ball packing arguments, the complexity of the outer for-loop of \cref{alg:sparsityAggregated} can be bounded by $CN$, as well.
The complexity of \cref{alg:reduce} is bounded by $CN\rho^{\tilde{d}}$, the number of entries of the sparsity pattern $\bar{S}$, since it iterates over all entries of $\bar{S}$.
\end{proof}

\section{Including the prediction points}
\label{apssec:includingPrediction}

\subsection{Ordering the prediction points first}
\label{apsssec:predfirst}

\cref{alg:predfirst} describes how to compute the inverse Cholesky factor when forcing the prediction points to be ordered before the training points. 
In order to compute the ordering of the prediction points after the ordering of the training points has been fixed, we need to compute the distance of each prediction point to the closest training point.
When using \cite[Alg.~3]{Schafer2017}, this can be done efficiently while computing the maximin ordering of the training points by including the prediction points into the initial list of children of $i$, as is done in Line 11 of the algorithm.
Once the joint inverse Cholesky factor $L = \left(\begin{smallmatrix} L_{\Pred, \Pred} & 0\\ L_{\Train,\Pred} & L_{\Train, \Train} \end{smallmatrix}\right)$ has been computed, we have $\Expect\left[X_{\Pred} \middle | X_{\Train} = y\right] = L_{\Pred,\Pred}^{-\top} L_{\Train,\Pred}^\top y$ and $\Cov\left[X_{\Pred} | X_{\Train}\right] = L_{\Pred,\Pred}^{-\top} L_{\Pred,\Pred}^{-1}$.
We note that the conditional expectation can be computed by forming the columns of $L$ one by one, without every having to hold the entire matrix in memory, thus leading to linear space complexity similar to \cref{ssec:lowmemory}, while the same is not possible for the conditional covariance matrix.
\begin{figure}[t]
    \tiny
	\begin{minipage}[t]{6.0cm}
		\vspace{0pt}
		\begin{algorithm}[H]
			\textbf{Input:} $\K$, $\left\{x_i\right\}_{i \in \I_{\Train}}$, $\Omega$, $\left\{x_i\right\}_{i \in \I_{\Train}}$, $\rho$, ($\lambda$)\\
			\textbf{Output:} $L\in \Reals^{N \times N}$ l. triang. in $\prec$\\
		    \begin{algorithmic}[1]	
			    \STATE Comp. $\prec_{\Pred}$, $l_{\Pred}$ from $\left\{x_i\right\}_{i \in \I_{\Pred}}$, $\tilde{\Omega}$
			    \STATE Comp. $\prec_{\Train}$, $l_{\Train}$ from $\left\{x_i\right\}_{i \in \I_{\Train}}$, $\Omega$
			    \STATE $\prec \leftarrow (\prec_{\Pred}, \prec_{\Train})$
			    \STATE $l \leftarrow (l_{\Pred}, l_{\Train})$
			    \STATE $S \leftarrow S_{\prec, l, \rho}$ ($S \leftarrow S_{\prec, l, \rho, \lambda}$)
			    \STATE Comp. $L$ using Algorithm~\ref{alg:notAggregated}(\ref{alg:aggregated})
			    \RETURN $L$
			\end{algorithmic}
			\caption{\label{alg:predfirst}Ordering prediction variables first}
		\end{algorithm}
		\vspace{0pt}
		\begin{algorithm}[H]
			\textbf{Input:} $\K$, $\left\{x_i\right\}_{i \in \I_{\Train}}$, $\Omega$, $\left\{x_i\right\}_{i \in \I_{\Train}}$, $\rho$, $\lambda$ \\
			\textbf{Output:} Cond. mean and var. $\mu, \sigma \in \Reals^{\I_{\Pred}}$\\
		    \begin{algorithmic}[1]
			    \STATE Comp. $\prec$, $S_{\prec,l,\rho,\lambda}$ from $\left\{x_i\right\}_{i \in \I_{\Train}}$, $\Omega$\;
			    \FOR{$k \in \I_{\Pred}$}
			        \STATE $\delta_{k}, \sigma_k \leftarrow \K(x_k, x_k), \K(x_k, x_k)^{-1}$
			        \STATE $\mu_k \leftarrow 0$
			    \ENDFOR
			    \FOR{$\tilde{k} \in \tilde{\I}$}
			        \STATE $U^ \leftarrow P^{\updownarrow}\chol( P^{\updownarrow} \KM_{s_{\tilde{k}}, s_{\tilde{k}}} P^{\updownarrow})P^{\updownarrow}$
			        \FOR{$k \in s_{\tilde{k}}, l \in \I_{\Pred}$}
			            \STATE $B_{kl} \leftarrow \K(x_k, x_l)$
			        \ENDFOR
			        \STATE $B \leftarrow U^{-1} B$
			        \STATE $\tilde{y} \leftarrow U^{-1} y_{s_{\tilde{k}}}$
			     \FOR{$k \leadsto \tilde{k}$}
			         \STATE $\alpha \leftarrow      \tilde{y}_{s_k}^{\top} B_{s_k,\Pred}$
			         \STATE $\beta \leftarrow B_{s_k,\Pred}^{\top} B_{s_k,\Pred}$
			         \STATE $\gamma  \leftarrow \sqrt{ 1 + (\delta - \beta)^{-1} B_{k,\Pred}^{2}}$
			         \STATE $\ell \leftarrow - \delta^{-1} \gamma^{-1} B_{k,\Pred}^{\top} \left(1 + \frac{\beta}{\delta - \beta}\right)$
			         \STATE $\mu \leftarrow \mu + \ell/\gamma  \left(\tilde{y}_{k} + \frac{B_{k,\Pred} \alpha }{ \delta - \beta}\right)$
			         \STATE $\sigma \leftarrow \sigma + \ell^{2}$
			     \ENDFOR
			    \ENDFOR  
			    \STATE $\sigma \leftarrow \sigma^{-1}$
			    \STATE $\mu \leftarrow -\sigma \mu$
			    \RETURN $\mu$, $\sigma$
		    \end{algorithmic}
			\caption{\label{alg:predlastb1} Ordering prediction variables last, with $n_{\Pred} = 1$}
		\end{algorithm}
	\end{minipage}
	\scriptsize
	\begin{minipage}[t]{6.9cm}
		\vspace{0pt}
		\begin{algorithm}[H]
			\textbf{Input:} $\K$, $\left\{x_i\right\}_{i \in \I_{\Train}}$, $\Omega$, $\left\{x_i\right\}_{i \in \I_{\Train}}$, $\rho$, $\lambda$\\
			\textbf{Output:} Per batch cond. mean $\{\mu_b\}_{1 \leq b \leq m_{\Pred}}$ and cov. $\{C_b\}_{1 \leq b \leq m_{\Pred}}$ \\
			\begin{algorithmic}[1]
			    \STATE Comp. $\prec$, $S_{\prec,l,\rho,\lambda}$ from $\left\{x_i\right\}_{i \in \I_{\Train}}$, $\Omega$
			    \FOR{$\tilde{k} \in \tilde{I}$}
			        \STATE $U^{\tilde{k}} \leftarrow P^{\updownarrow}\chol( P^{\updownarrow} \KM_{s_{\tilde{k}}, s_{\tilde{k}}} P^{\updownarrow})P^{\updownarrow}$
			    \ENDFOR
			    \FOR{$b \in \{ 1, \dots, m_{\Pred} \}$}
			        \FOR{$\tilde{k} \in \tilde{\I}$}
			            \FOR{$k \in s_{\tilde{k}}, l \in \J_{b}$}
			                \STATE $(\KM_{s_{\tilde{k}}, b})_{kl} \leftarrow \K(x_k, x_l)$
			            \ENDFOR
			            \FOR{$k, l \in \J_{b}$}
			                \STATE $(\KM_{b, b})_{kl} \leftarrow \K(x_k, x_l)$
			            \ENDFOR
			            \STATE $C_b \leftarrow \KM_{b,b}^{-1}$
			            \STATE $B^{\tilde{k}} \leftarrow U^{\tilde{k},-1} \KM_{\tilde{k},b}$
			            \STATE $y^{\tilde{k}} \leftarrow U^{\tilde{k},-1} y_{s_{\tilde{k}}}$
			    	    \FOR{$k \leadsto \tilde{k}$}
			    	        \STATE $v \leftarrow \newline B_{s_k,b}^{\tilde{k}}\left( \KM_{b,b} - B_{s_k, b}^{\tilde{k},\top} B_{s_{k},b}\right)^{-1} B_{k,b}^{\tilde{k},\top}$
			    	        \STATE $c \leftarrow \sqrt{ 1 + v_k}$
			    	        \STATE $L_{b,k} \leftarrow -\frac{1}{c_k} \KM_{b,b}^{-1} B_{s_k,b}^{\tilde{k},\top} \left(\mathbf{e}_1 + v\right)$
			    	        \STATE $C_b \leftarrow C_b + L_{b,k} L_{b,k}^{\top}$
			    	        \STATE $\mu_b \leftarrow \mu_b + L_{b,k} (\frac{1}{c_k} y_{s_k}^{\tilde{k},\top} \left(\mathbf{e}_1 + v\right))$
			    	    \ENDFOR
			        \ENDFOR  
			    	\STATE $\mu_{b} \leftarrow -C_{b} \mu_{b}$
			    \ENDFOR
			    \STATE $C_b \leftarrow C_b^{-1}$
			    \RETURN $(\mu_b, C_b)_{1 \leq b \leq m_{\Pred}}$
			\end{algorithmic}
			\caption{\label{alg:predlast} Ordering prediction variables last}
		\end{algorithm}
	\end{minipage}
	\caption{The algorithms to use when ordering prediction points first or last. Here, we partition the prediction variables into batches, as $\I_{\Pred} = \bigcup_{b=1}^{m_{\Pred}} \J_b$.}
\end{figure}

\subsection{Ordering the prediction points last, for accurate extrapolation}
\label{apsssec:predlast}

Splitting the prediction set $\I_{\Pred} = \bigcup \limits_{1 \leq b \leq m_{\Pred}} \J_{b}$ into $m_{\Pred}$ batches of $n_{\Pred}$ predictions, we want to compute the conditional mean vector and covariance matrix of the variables in each batch separately, by using the inverse Cholesky factor $\bar{L}^\rho$ of the joint covariance matrix obtained from KL-minimization subject to the sparsity constraint given by $\bar{S} = S_{\prec, l, \rho, \lambda} \cup \left\{ (i,j) : j \in \J_{b}\right\}$.
Naively, this requires us to recompute the inverse Cholesky factor $L$ for every batch, leading to a computational complexity of $\O\left( m_{\Pred} (N + n_{\Pred}) (\rho^{\tilde{d}} + n_{\Pred})^2 \right)$. 
However, by reusing a part of the computational complexity across different batches, \cref{alg:predlast} is to instead achieve computational complexity of $\O\left( (N + n_{\Pred}) ( \rho^{2\tilde{d}} + m_{\Pred}\left(\rho^{\tilde{d}} + n_{\Pred}^2 \right) \right)$.
In the following, we derive the formulae used by this algorithm to compute the conditional mean and covariance.
For a fixed batch $\J_{b}$, define $\bar{\KM}$ as the approximate joint covariance matrix implied by the inverse Cholesky factor $\bar{L}^{\rho}$.
It has the block-structure 
\begin{scriptsize}
\begin{equation}
    \begin{pmatrix}
        \bar{\KM}_{\Train,\Train} & \bar{\KM}_{\Train,b}\\
        \bar{\KM}_{b,\Train} & \bar{\KM}_{b,b}
    \end{pmatrix}
    \eqqcolon 
    \begin{pmatrix}
        \bar{\IKM}_{\Train,\Train} & \bar{\IKM}_{\Train,b}\\
        \bar{\IKM}_{b,\Train} & \bar{\IKM}_{b,b}
    \end{pmatrix}^{-1}
    =
    \begin{pmatrix}
        \bar{L}_{\Train,\Train}^{\top} & \bar{L}_{b,\Train}^{\top}\\
        0 & \bar{L}_{b,b}^{\top}
    \end{pmatrix}^{-1}
    \begin{pmatrix}
        \bar{L}_{\Train,\Train} & 0\\
        \bar{L}_{b,\Train} & \bar{L}_{b,b}
    \end{pmatrix}^{-1}
    \eqqcolon 
    \bar{L}^{\rho, -\top} 
    \bar{L}^{\rho, -1},
\end{equation}
\end{scriptsize}
where $\bar{L}^{\rho}$ is the inverse-Cholesky factor obtained by applying KL-minimization to the joint covariance matrix subject to the sparsity constraint given by $\bar{S}$.
We can then write the posterior mean and covariance of a GP $X \sim \mathcal{N}(0,\bar{\KM})$ as
\begin{footnotesize}
\begin{align}
    &\Expect\left[X_{b}|X_{\Train} = y\right] = \bar{\KM}_{b,\Train} \bar{\KM}_{\Train,\Train}^{-1}y = -\bar{\IKM}_{b,b}^{-1} \bar{\IKM}_{b,\Train}y
    = - \left( \bar{L}_{b,\Train} \bar{L}_{b,\Train}^{\top}  +  \bar{L}_{b,b} \bar{L}_{b,b}^{\top}\right)^{-1} \bar{L}_{b,\Train} \bar{L}_{\Train,\Train}^{\top}y\\
    &\Cov\left[X_{b}|X_{\Train}\right] 
    = \bar{\KM}_{b,b} - \bar{\KM}_{b,\Train} \bar{\KM}_{\Train,\Train}^{-1} \bar{\KM}_{\Train,b} 
    = \bar{\IKM}_{b,b}^{-1}
    =  \left(\bar{L}_{b,\Train} \bar{L}_{b,\Train}^{\top}  +  \bar{L}_{b,b} \bar{L}_{b,b}^{\top} \right)^{-1}
\end{align}
\end{footnotesize}
Expanding the matrix multiplications into sums, this can be rewritten as 
\begin{footnotesize}
\begin{align}
    &\Expect\left[X_{b}|X_{\Train} \right]
    = -\left( \bar{L}_{b,b} \bar{L}_{b,b}^{\top} + \sum \limits_{k \in \I_{\Train}} \bar{L}_{b,k} \otimes \bar{L}_{b,k} \right)^{-1} 
    \sum \limits_{k \in \I_{\Train}} \bar{L}_{b,k} \left( y^{\top} \bar{L}_{\Train,k} \right).\\
    &\Cov\left[X_{b}|X_{\Train}\right] 
    =\left(\bar{L}_{b,b} \bar{L}_{b,b}^{\top} + \sum \limits_{k \in \I_{\Train}} \bar{L}_{b,k} \otimes \bar{L}_{b,k} \right)^{-1}
\end{align}
\end{footnotesize}
$\bar{L}_{b,b}$ is simply the Cholesky factor of $\KM_{b,b}^{-1}$.
Thus, given $\left(y^{\top} \bar{L}_{\Train,k}, L_{b,k}\right)_{k \leadsto \tilde{k}}$, the above expressions can be evaluated in computational complexity $\O(n_{\Pred}^3 + N_{\Train} n_{\Pred}^2 + n_{\Pred} \#S)$ in time and $\O(n_{\Pred}^2+ \max_{\tilde{l} \in \tilde{\I}_{k}} \# \tilde{l})$ in space.
Naively, computing the $\left(y^{\top} L_{\Train,k}, L_{b,k}\right)_{k \leadsto \tilde{k}}$ for each batch has computational complexity $\O(m_{\Pred}(\# \tilde{k} + n_{\Pred})^3)$ which becomes the bottleneck for large numbers of batches.
However, as we will see, $\langle L_{\Train,k}, y \rangle$ and $L_{b,k}$ can be computed in computational complexity $\O( (\#  \tilde{k} + n_{\Pred})^3 + m_{\Pred} (\#  \tilde{k} + n_{\Pred})^2)$ by reusing parts of the computation.
Fix a supernodal index $\tilde{k} \in \tilde{I}$ and define the corresponding exact joint covariance matrix as
\begin{equation}
    \KM^{\tilde{k}} \defeq \left(\KM_{ij}\right)_{\left\{i,j \in \tilde{k} \cup \J_{b}\right\}} 
    = 
    \begin{pmatrix}
        \KM_{\tilde{k},\tilde{k}} && \KM_{\tilde{k},b}\\
        \KM_{b,\tilde{k}} && \KM_{b,b}
    \end{pmatrix}
\end{equation}
For any $k \leadsto \tilde{k}$ the column $L^{\rho}_{:,k}$ is, according to \cref{eqn:defcolL}, equal to 
\begin{equation}
    \frac{\left(\KM^{\tilde{k}}_{k:,k:}\right)^{-1} \mathbf{e}_1}{\sqrt{\mathbf{e}_1^{\top} \left(\KM^{\tilde{k}}_{k:,k:} \right)^{-1} \mathbf{e}_1}}.
\end{equation}
Let as before $U^{\tilde{k}} U^{\tilde{k},\top} = \KM_{\tilde{k},\tilde{k}}$.
Using the Sherman-Morrison-Woodbury matrix identity we can then rewrite $\KM_{k:, k:}^{\tilde{k}, -1} \mathbf{e}_{1}$ as 
\begin{footnotesize}
\begin{align*}
&\KM_{k:, k:}^{\tilde{k}, -1} \mathbf{e}_1
=
\begin{pmatrix}
\Id                 & 0   \\
- \KM_{b,b}^{-1} \KM_{b,s_k} & \Id  
\end{pmatrix}
\begin{pmatrix}
\left(\KM_{s_k,s_k} - \KM_{s_k, b} \KM_{b,b}^{-1} \KM_{b, s_k} \right)^{-1}           & 0   \\
0 & \KM_{b,b}^{-1}
\end{pmatrix}
\begin{pmatrix}
\Id & - \KM_{s_k,b} \KM_{b,b}^{-1}  \\
0   & \Id  
\end{pmatrix}
\mathbf{e}_1\\
=&
\begin{pmatrix}
    &\left( \KM_{s_k,s_k} - \KM_{s_k, b} \KM_{b,b}^{-1} \KM_{b, s_k}\right)^{-1} \mathbf{e}_1 \\
    &- \KM_{b,b}^{-1} \KM_{b,s_k}\left( \KM_{s_k,s_k} - \KM_{s_k, b} \KM_{b,b}^{-1} \KM_{b, s_k}\right)^{-1} \mathbf{e}_1
\end{pmatrix}\\
=&
\begin{pmatrix}
    &\left( \KM_{s_k,s_k}^{-1} - \KM_{s_k,s_k}^{-1} \KM_{s_k,b}\left( - \KM_{b,b} + \KM_{b, s_k} \KM_{s_k,s_k}^{-1} \KM_{s_k, b}  \right)^{-1} \KM_{b,s_k} \KM_{s_k, s_k}^{-1}   \right)\mathbf{e}_1 \\
    &- \KM_{b,b}^{-1} \KM_{b,s_k}\left( \KM_{s_k,s_k}^{-1} - \KM_{s_k,s_k}^{-1} \KM_{s_k,b}\left( - \KM_{b,b} + \KM_{b, s_k} \KM_{s_k,s_k}^{-1} \KM_{s_k, b}  \right)^{-1} \KM_{b,s_k} \KM_{s_k, s_k}^{-1}   \right) \mathbf{e}_1
\end{pmatrix}
\end{align*}
\end{footnotesize}
Using Equation~\eqref{eqn:restrictUUT} and setting $B^{\tilde{k}} \defeq U^{\tilde{k},-1} \KM_{\tilde{k},b}$, we obtain
\begin{equation}
\bar{\KM}_{k:, k:}^{\tilde{k}, -1} \mathbf{e}_1 = \frac{1}{U_{k,k}^{\tilde{k}}}
\begin{pmatrix}
    & U_{s_k,s_k}^{\tilde{k},-\top} \left( \mathbf{e}_1 + B_{s_k,b}^{\tilde{k}} \left(  \KM_{b,b} - B_{s_k,b}^{\tilde{k},\top} B_{s_k,b}^{\tilde{k}}\right)^{-1}B_{k,b}^{\tilde{k},\top} \right)\\
    &- \KM_{b,b}^{-1} B_{s_k,b}^{\tilde{k},\top}\left( \mathbf{e}_1 + B_{s_k,b}^{\tilde{k}} \left( \KM_{b,b} - B_{s_k,b}^{\tilde{k},\top} B_{s_k,b}^{\tilde{k}}\right)^{-1}B_{k,b}^{\tilde{k},\top} \right)
\end{pmatrix}.
\end{equation}
Setting $y^{\tilde{k}} = U^{\tilde{k},-1}y_{s_{\tilde{k}}}$, this yields the formulae
\begin{align}
&y^{\top} \bar{L}_{\Train,k} = \frac{y^{\tilde{k},\top}_{k} + y_{s_k}^{\tilde{k},\top} B_{s_k,b}^{\tilde{k}} \left( \KM_{b,b} - B_{s_k,b}^{\tilde{k},\top} B_{s_k,b}^{\tilde{k}}\right)^{-1}B_{k,b}^{\tilde{k},\top}}{c_k}\\
&\bar{L}_{b,k} = \frac{- \KM_{b,b}^{-1} B_{s_k,b}^{\tilde{k},\top}\left( \mathbf{e}_1 + B_{s_k,b}^{\tilde{k}} \left( \KM_{b,b} - B_{s_k,b}^{\tilde{k},\top} B_{s_k,b}^{\tilde{k}}\right)^{-1}B_{k,b}^{\tilde{k},\top} \right)}{c_k},
\end{align}
where 
\begin{equation}
c_k \defeq \sqrt{1 + B_{k,b}^{\tilde{k}} \left( \KM_{b,b} - B_{s_k,b}^{\tilde{k},\top} B_{s_k,b}^{\tilde{k}}\right)^{-1}B_{k,b}^{\tilde{k},\top}}.
\end{equation}
\cref{alg:predlast} implements the formulae above.
Since $U^{\tilde{k}}$ does not depend on $b$, it only has has to be computed once and can be used to compute the $B^{\tilde{k}}$ and $y^{\tilde{y}}$ for all $1 \leq b \leq m_{\Pred}$.

\end{document}